\documentclass{amsart}
\usepackage[utf8]{inputenc}

\usepackage{geometry}
\usepackage{epsdice}

\usepackage[inline,shortlabels]{enumitem}

\usepackage{color}

\usepackage[OT2,T1]{fontenc}

\usepackage[final]{microtype}

\usepackage{tikz}
\usepackage{tikz-cd}
\usepackage{xypic}

\usepackage[draft=false,linktocpage=true]{hyperref}
\usepackage[capitalise]{cleveref}

\usepackage[normalem]{ulem}

\usepackage{relsize}
\usepackage[bbgreekl]{mathbbol}
\usepackage{amsfonts}
\DeclareSymbolFontAlphabet{\mathbb}{AMSb} 
\DeclareSymbolFontAlphabet{\mathbbl}{bbold}
\newcommand{\Prism}{{\mathlarger{\mathbbl{\Delta}}}} 

\usepackage{mathtools}
\usepackage{amsmath}

\usepackage{amsthm,amssymb}
\numberwithin{equation}{section}

\usepackage{extarrows}

\theoremstyle{plain}
\newtheorem{theorem}[equation]{Theorem}

\newtheorem{proposition}[equation]{Proposition}
\newtheorem{lemma}[equation]{Lemma}
\newtheorem{porism}[equation]{Porism}
\newtheorem{claim}[equation]{Claim}
\newtheorem{corollary}[equation]{Corollary}
\newtheorem{fact}[equation]{Fact}

\theoremstyle{definition}
\newtheorem{definition}[equation]{Definition}
\newtheorem{construction}[equation]{Construction}

\newtheorem{example}[equation]{Example}
\newtheorem{situation}[equation]{Situation}
\newtheorem{notation}[equation]{Notation}
\newtheorem{remark}[equation]{Remark}
\newtheorem{digression}[equation]{Digression}

\usepackage{colonequals}
\newcommand{\DF}{{\mathcal{DF}}}
\newcommand{\Spf}{{\mathrm{Spf}}}
\newcommand{\qrsp}{{\mathrm{qrsp}}}
\newcommand{\qsyn}{{\mathrm{qSyn}}}
\newcommand{\FG}{{\mathrm{F\text{-}Gauge}}}
\newcommand{\fG}{{\mathrm{(F\text{-})Gauge}}}
\newcommand{\fC}{{\mathrm{(F\text{-})Crys}}}
\newcommand{\FC}{{\mathrm{F\text{-}Crys}}}
\newcommand{\vect}{{\mathrm{vect}}}

\newcommand{\perf}{{\mathrm{perf}}}
\newcommand{\Perf}{{\mathrm{Perf}}}

\newcommand{\coh}{{\mathrm{coh}}}
\newcommand{\Fil}{{\mathrm{Fil}}}
\newcommand{\FH}{{\mathrm{FilHiggs}}}
\newcommand{\FConn}{{\mathrm{FilConn}}}
\newcommand{\conj}{{\mathrm{conj}}}
\newcommand{\Rees}{{\mathrm{Rees}}}
\newcommand{\D}{{\mathcal{D}}}
\newcommand{\DG}{{\mathcal{DG}}}
\newcommand{\GH}{{\mathrm{GrHiggs}}}
\newcommand{\Gr}{{\mathrm{Gr}}}

\newcommand{\wt}{{\mathrm{wt}}}
\newcommand{\colim}{{\mathrm{colim}}}
\newcommand{\Red}{{\mathrm{Red}}}

\definecolor{mb}{rgb}{0.36, 0.54, 0.66}

\title{Frobenius height of prismatic cohomology with coefficients}

\author{Haoyang Guo}

\address{H.~Guo: The University of Minnesota Twin cities, Vincent Hall, 206 Church St. SE, Minneapolis, MN 55455, USA}

\email{ghy@umn.edu}

\author{Shizhang Li}
\address{S.~Li: Morningside Center of Mathematics and State Key Laboratory of Mathematics Sciences,
Academy of Mathematics and Systems Science, Chinese Academy of Sciences, Beijing 100190, China}
\email{lishizhang@amss.ac.cn}

\begin{document}

\maketitle

\begin{abstract}
We study the behavior of Frobenius operators on smooth proper pushforwards
of prismatic $F$-crystals. In particular we show that the $i$-th pushforward
has its Frobenius height increased by at most $i$.
Our proof crucially uses the notion of prismatic $F$-gauges introduced by Drinfeld and Bhatt--Lurie
and its relative version, and we give a self-contained treatment
without using the stacky formulation.
\end{abstract}

\tableofcontents

\section{Introduction}

\addtocontents{toc}{\protect\setcounter{tocdepth}{1}}

\subsection*{Background and main result}
Fix a prime number $p$, throughout this paper
we let $K$ be a $p$-adic discretely valued field with perfect residue field $k$,
and let $\mathcal{O}_K$ be the ring of integers.
For smooth proper $p$-adic formal schemes over $\mathcal{O}_K$,
the prismatic cohomology introduced by
Bhatt--Scholze in \cite{BS19} (see also \cite{BMS18} and \cite{BMS19})
can be regarded as the universal cohomology theory.
Recently there has been many works concerning ``general coefficients''
for this cohomology theory, known as \emph{prismatic $F$-crystals},
to name a few: \cite{MT20}, \cite{Wu21}, \cite{BS21}, \cite{DuLiu21}, \cite{MW21},
\cite{DLMS} and \cite{GR22}.

In this paper, we study the behavior of these general coefficients
under smooth proper pushforwards. More precisely, we are interested
in estimating the height of the induced Frobenius operator.
Our main result is the following (combining \Cref{bound of Frob height of cohomology}
and \Cref{Verschiebung}):
\begin{theorem}[Main Theorem]
\label{main}
Let $f \colon X\to Y$ be a smooth proper morphism of smooth $p$-adic formal schemes over $\mathcal{O}_K$
of relative equidimension $n$,
and let $(\mathcal{E},\varphi_{\mathcal{E}})$ be an $\mathcal{I}$-torsionfree
coherent prismatic $F$-crystal over $X$ with Frobenius height in $[a,b]$.

Then for each integer $i$, the $i$-th relative prismatic cohomology $R^i f_{\Prism, *} \mathcal{E}$
induces a coherent prismatic $F$-crystal over $Y_\Prism$, 
whose $\mathcal{I}$-torsionfree quotient has Frobenius height in
$[a+\max\{0 ,i-n\}, b+\min\{i,n\}]$.

Moreover there exists a Verschiebung operator
\[
\psi_i \colon \mathcal{I}^{\otimes (i+b)}\otimes_{\mathcal{O}_{\Prism}} R^i f_{\Prism, *} \mathcal{E} 
\longrightarrow \varphi_{\Prism_Y}^* R^i f_{\Prism, *} \mathcal{E},
\]
which is inverse to the Frobenius operator of $R^i f_{\Prism, *} \mathcal{E}$ up to multiplying $\mathcal{I}^{\otimes (i+b)}$.
\end{theorem}

According to \cite[Theorem 1.3]{DLMS} or \cite[Theorem A]{GR22},
any $\mathbb{Z}_p$-crystalline local system comes from an $I$-torsionfree coherent prismatic $F$-crystal.
Moreover, in \cite[\S 8]{GR22} the first named author and Reinecke have described alternative approaches
to showing that the induced Frobenius operator on $Rf_{\Prism, *} \mathcal{E}$ is an isogeny.
The approach in this paper has the advantage that we can allow torsions in coefficients and cohomology,
and we give precise bounds on Frobenius height.

Notice that prismatic $F$-crystals can give rise to usual crystalline
$F$-crystals in characteristic $p$ as well as local systems in characteristic $0$.
Our estimate of Frobenius height is compatible with
the work of Kedlaya \cite[Thm.~6.7.1]{Ked06} 
bounding Frobenius slopes of rigid cohomology with coefficients
as well as the work of Shimizu \cite[Theorem 5.10]{Shimizu18}
controlling generalized Hodge--Tate weights of cohomology of local systems.

In the case of constant coefficient $\mathcal{E} = \mathcal{O}_{\Prism, X}$,
this result has been established by Bhatt--Scholze in \cite[\S 15]{BS19}.
In fact, our proof is in spirit a generalization of theirs: We found a canonical
way to promote the $F$-crystal to a ``Nygaard-filtered $F$-crystal'', a.k.a.~\emph{prismatic $F$-gauge},
whose graded pieces are controlled and can be estimated by coherent cohomologies.
The framework of prismatic $F$-gauges
had been laid down by Drinfeld \cite{Drin20} and Bhatt--Lurie \cite{BL22a}, \cite{BL22b}.
Many detailed discussions can be found in lecture notes by Bhatt \cite{Bha23},
and our paper is inspired by results therein.
However we choose to not follow their stacky approach,
and instead just discuss prismatic ($F$-)gauges and relative ($F$-)gauges
using slightly more concrete terms.

In the remainder of this introduction, let us briefly indicate the proof idea,
which will also allow us to summarize various sections of this article.
In \Cref{filtered homological algebra}, we lay down some foundational
discussions on homological algebras in graded or filtered setting.
We in particular give several useful criteria of being ``finite projecive'' or ``perfect''
for graded/filtered complexes over graded/filtered rings, which might be of independent interest.

\subsection*{Constructions}
As some preliminary discussions,
in \Cref{The coherent prismatic ($F$-)crystals} we explain the existence
of a standard $t$-structure on the category of prismatic ($F$-)crystals
on $X$ (see \Cref{def coherent crys}), so that we have well-defined functors $R^i f_{\Prism, *}$
sending coherent crystals on $X$ to those on $Y$, and
expressions like ``coherent'' and ``$I$-torsionfree'' appearing in the main theorem above make sense.
 
Here comes the main construction.
In \Cref{Prismatic $F$-gauges}, we give a definition of (\emph{absolute})~($F$-)\emph{gauges} on the formal spectra of qrsp algebras
$S$ and show that the assignment $S \mapsto \fG(\Spf(S))$ is a quasi-syntomic sheaf
(see \Cref{F-gauge sheaf}). 
By the unfolding process (c.f.~\cite[Proposition.~4.31]{BMS19}),
this gives rise to the notion of ($F$-)gauges on our $X$.
Then we show the following:
\begin{theorem}[= \Cref{thm crystal to gauge}]
\label{Theorem 1.2}
There is a fully faithful functor 
\[
\Pi_X \colon \{I\text{-torsionfree coherent } F\text{-crystals on } X\} \longrightarrow 
\{\text{coherent } F\text{-gauges on } X\},
\]
whose composition with the forgetful functor onto the underlying $F$-crystal is the identity.
\end{theorem}
The functor is defined by equipping an $F$-crystal with the ``saturated Nygaardian filtration'' on quasiregular semiperfectoid rings, and the essential work is to show that the filtration can be glued for quasi-syntomic topology\footnote{The smoothness assumption is crucial in order to show that the saturated Nygaardian filtration glues. 
We do not know if this is true for $F$-crystals over general quasi-syntomic schemes, except when it is \emph{admissible} of weight $[0,1]$ (cf. \Cref{p-div and Fgauge})}.
This is inspired
by the special case of $X=\Spf(\mathcal{O}_K)$ in \cite[Theorem 6.6.13]{Bha23} 
and extends the result in loc.~cit.~to arbitrary 
smooth $p$-adic formal schemes over $\Spf(\mathcal{O}_K)$
\footnote{Analogous to \cite[Proposition.~6.6.3]{Bha23}, one can describe the essential image of the functor $\Pi_X$ in terms of modules over the ring $A[u,t]/(ut-d)$, where $(A,(d))$ is a perfect prism.}.

Next in \Cref{Relative prismatic $F$-gauges}, we study the analogous notion of \emph{relative ($F$-)gauges}
on formal schemes $Z$ that are quasi-syntomic over the reduction of some base prism $(A, I)$
(see also \cite[\S 2]{Tang22}), 
which again satisfies the quasi-syntomic sheaf property (see \Cref{F-gauge sheaf relative}).
By a natural base change process, we have:
\begin{theorem}[= \Cref{restriction functor}]
Let $(A, I)$ be a bounded prism in $Y_{\Prism}$.
There is a natural functor
\[
BC \colon \{F\text{-gauges on } X\} \longrightarrow 
\{\text{relative } F\text{-gauges on } (X_{\overline{A}}/A)\}.
\]
\end{theorem}
Moreover, we can associate a relative $F$-gauge with a \emph{filtered Higgs field} and a \emph{filtered flat connection}.
This is analogous to Hodge--Tate comparison and de Rham comparison of prismatic cohomology in \cite{BS19},
 except that we work with more general coefficients.
Below for a quasi-syntomic formal scheme $Z$ over $\overline{A}$, we let $\FH(Z/A)$ and $\FConn(Z/\overline{A})$ be the categories of filtered Higgs fields over $Z/A$ and filtered flat connections over $Z/\overline{A}$ separately.
These are filtered derived enhancements of the usual notion of Higgs fields
and flat connections (see \Cref{def Higgs} and \Cref{def connections} for details).
Then we have the following constructions:
\begin{theorem}[see \S \ref{sec filtered Higgs} and \S \ref{Completeness of Nygaard filtration via filtered de Rham realization}
for details]
\label{realization functors}
Let $(A,I)$ be a bounded prism, and let $Z$ be a quasi-syntomic formal scheme over $\overline{A}$.
\begin{enumerate}[label=(\roman*)]
\item \emph{Hodge--Tate realization:}  There is a natural functor 
\[
\{\text{relative } F\text{-gauges on } (Z/A)\} \longrightarrow \FH(Z/A)
\]
\item \emph{de Rham realization:} There is a natural functor
\[
\{\text{relative } F\text{-gauges on } (Z/A)\} \longrightarrow \FConn(Z/\overline{A}).
\]
\end{enumerate}
\end{theorem}
Generalizing Nygaard--conjugate comparison in \cite[Thm.~1.14.(2)]{BS19}, the functors above allow us to describe the graded pieces of a relative $F$-gauge in terms of the filtrations on the associated filtered Higgs complex  and the filtered connection.
For our application, we discuss the prismatic--Higgs relation in details and refer the reader to \Cref{Higgs associated to F-gauge cohomology}.
	
To summarize: at this point we have functorially attached filtrations to our prismatic
$F$-crystal restricted to $(X_{\overline{A}}/A)$.
It remains to study the induced filtration.

\subsection*{Properties}
A key observation introduced in \Cref{absolute weight filtration subsection} is that there is a natural increasing \emph{weight filtration} on the graded pieces of an absolute gauge $E$.
An important feature for this weight filtration is that its graded pieces are generated by coherent complexes over $\mathcal{O}_X$.
Precisely, we have:
\begin{theorem}[= \Cref{thm weight filtration of gauge}]
    Let $X$ be a quasi-syntomic formal scheme, and let $E$ be an $F$-gauge in perfect complexes over $X$.
    Then one can naturally associate the following data to $E$:
    \begin{itemize}
        \item two integers $a\leq b$,
        \item a perfect $\mathcal{O}_X$-complex $\Red_i(E)$ for each integer $i\in [a,b]$, and
        \item a finite exhaustive increasing filtration $\Fil^\wt_i$ on $\Gr^\bullet E$ of range $[a,b]$,
    \end{itemize}
    such that the $i$-th graded piece $\Gr^\wt_i(\Gr^\bullet E)$ is naturally isomorphic to $\Red_i(E) \otimes_{\mathcal{O}_X} \Gr^{\bullet}_N \mathcal{O}_{\Prism}$.
\end{theorem}
The index of the weight filtration is called the \emph{weights}.
Moreover, when an $F$-gauge comes from an $I$-torsionfree coherent
$F$-crystal with Frobenius height in $[a, b]$ as in \Cref{Theorem 1.2}, 
we can relate the weight filtrations on the
graded piece $\Gr^\bullet E$ with the knowledge of the original Frobenius action; see \Cref{Reduction is the graded of twisted filtration}.
In particular, the weight of the associated $F$-gauge also lies in $[a, b]$.
In fact, completely analogous statements hold as well in the relative setting and for filtered Higgs fields.
Moreover the weight filtrations in absolute and relative settings are compatible
under the realization functor in \Cref{realization functors},
see \Cref{wt filtration rel F-gauge}. 
	We can thus understand cohomology of the filtered Higgs field $M$ using coherent cohomology of $\Red_{\bullet}(M)$ and that of sheaves of differentials $\Omega^i_{Z/\overline{A}}$:
	\begin{theorem}[= \Cref{conjugated graded of HT complex}]\label{graded of HT}
		Assume $Z$ is smooth of relative dimension $n$ over $\overline{A}$, and let $M$ be a filtered Higgs field of weight $[a,b]$.
		For each $i\in \mathbb{Z}$, the complex $R\Gamma(Z/A, \Gr_i M)$ admits a finite increasing exhaustive filtration of range $[a,b]$, such that each graded piece is isomorphic to $R\Gamma(Z, \Red_u(M)\otimes_{\mathcal{O}_Z} \Omega^v_{Z/\overline{A}})$ for some $u,v\in \mathbb{Z}$.
	\end{theorem}
	Combining these facts, we also get the following byproduct:
the Higgs field coming from a coherent $I$-torsionfree
$F$-crystal has nilpotent Higgs structure, and the order of  nilpotence is bounded above by the length
of the range of Frobenius heights.
	
On the other hand, recall in \cite[\S 15]{BS19} the authors showed
that the graded pieces of Nygaard filtration of prismatic cohomology are
conjugate filtrations of Hodge--Tate cohomology.
Analogous results hold true in the setting of relative $F$-gauges:
Namely	for a relative $F$-gauge $E$ over $Z/A$, 
its Frobenius structure induces isomorphisms between the graded pieces of
$E^{(1)}\colonequals \varphi_A^* E$ and the filtrations of the associated filtered Higgs field
(\Cref{Higgs associated to F-gauge cohomology}).
Combine this with the analysis in \Cref{graded of HT},
we can thus understand the cohomology of graded pieces of $F$-gauges using coherent cohomology.
	
To understand the filtration of $F$-gauges, another ingredient is the fact that we work in the world of perfect
filtered/graded complexes  for the entire construction process. 
In \Cref{Completeness of Nygaard filtration via filtered de Rham realization}
we utilize this fact, coupled with the completeness of the relative Nygaard filtration on 
$\mathrm{R\Gamma}(X_{\overline{A}}/A, \Prism^{(1)})$ 
and some descendability results, 
to show that the induced ``Nygaard filtration''
on the cohomology of $E^{(1)}$ is also complete.
Thanks to the completeness, to bound the cohomology of the filtration of $E^{(1)}$,
it suffices to do so on its graded pieces, 
and hence reduce the calculation to aforementioned Nygaard-conjugate isomorphism.
	
Finally in \Cref{Height of prismatic cohomology} we harvest the fruits of the above discussion
and deduce several results, including those described in the above main theorem.
	
\subsection*{Byproducts and comments}
At the end of this introduction, we mention several byproducts of our work and give some comments on related topics.
\begin{remark}[Torsion of \'etale cohomology with coefficients]
Let $X$ be a smooth $p$-adic formal scheme.
As showed in \cite[Theorem 6.1]{GR22},
taking derived direct image along proper smooth morphism commutes with 
\'etale realization functor on the category of prismatic $F$-crystal in perfect complexes
(see also \Cref{direct image vs etale realization}). 
Moreover, we observe (\Cref{etale realization is t-exact}) that
\'etale realization functor is $t$-exact.
Combine the above observations with \Cref{main}, we get the following
refinement of ``weak \'{e}tale comparison'' (c.f.~\cite[Thm.\ 6.1]{GR22}):
\begin{corollary}[= \Cref{direct image vs etale realization} + \Cref{weak etale comparison}]
Let $f \colon X \to Y$ be a smooth proper morphism
between smooth formal schemes over $\Spf(\mathcal{O}_K)$, then derived pushforward
of $F$-crystals in perfect complexes on $X$ are $F$-crystals in perfect complexes on $Y$,
and the following diagram commutes functorially:
\[
\xymatrix{
\FC^{\perf}(X) \ar[d]_{Rf_{\Prism, *}} \ar[r]^{T(-)} & 
D^{(b)}_{lisse}(X_{\eta}, \mathbb{Z}_p) \ar[d]^{Rf_{\eta, *}} \\
\FC^{\perf}(Y) \ar[r]^{T(-)} & 
D^{(b)}_{lisse}(Y_{\eta}, \mathbb{Z}_p).
}
\]
Moreover we have
$T(R^i f_{\Prism, *} \mathcal{E}) = R^i f_{\eta, *}(T(\mathcal{E}))$.
\end{corollary}
This allows us to study torsions of individual \'etale cohomology groups of a given lisse \'etale complex over $\mathbb{Z}_p$ that is \emph{crystalline}
\footnote{We refer the reader to \cite[\S 1.2]{Bha23} for the notion and the related discussions.}, using the Frobenius structure and its bound on prismatic cohomology.
\end{remark}
\begin{remark}[$p$-divisible group and $F$-gauges]
Given a quasi-syntomic $p$-adic formal scheme $X$,
it is shown in \cite[Thm~1.16]{ALB23} that there is a natural equivalence for the following categories:
\[
\{\text{admissible prismatic } F\text{-crystals in vector bundles of height } [0,1] \text{ on }X\} \simeq \{p\text{-divisible groups on }X\}.
\]
We show in \Cref{p-div and Fgauge} the aforementioned categories are further equivalent to the category of $F$-gauges in vector bundles of weight $[0,1]$.
\end{remark}
We also mention two natural variants of our results.
\begin{remark}[Coherent $F$-gauge]
As shown in \Cref{boundedness results on twist graded}, an $F$-gauge that arises from an $I$-torsionfree coherent $F$-crystal satisfies some extra condition on the associated graded pieces, and we call such a condition \emph{weakly reflexive}.
This is related to the notion of \emph{reflexiveness} introduced in \cite[Def.\ 6.6.4]{Bha23}, but is slightly more genereal
(as we do not assume the underlying $F$-crystal to be locally free).
We expect that our method can be used to obtain an analogous result of \Cref{main} concerning these general $F$-gauges.
\end{remark}

\begin{remark}[Positive characteristic case]
Although we only work with smooth $p$-adic formal schemes $X$ over $\mathcal{O}_K$ and $F$-crystals/gauges on $X$,
the entire article can be carried to the characteristic $p$ setting: 
Namely we may instead assume $X$ to be a smooth variety over a perfect field $k$,
and consider cohomology of crystalline $F$-crystals/gauges on $X$.
In this case, $p$-torsionfree coherent $F$-crystals naturally give rise to coherent $F$-gauges,
and the Frobenius operator of their cohomology satisfies the same bound as in \Cref{main}.
\end{remark}

\subsection*{Notation and conventions}
We work in infinity categorical language throughout the article, and assume the basics of prisms and quasi-syntomic rings, which we refer the reader to \cite{Lur09}, \cite{Lur17}, \cite{BS19} and \cite{BMS19} for details.
Throughout the paper, the notation $I$ is always a locally principal ideal related to prismatic
constructions.
To ease notation, as long as we believe there is no risk of confusion,
we shall drop ``$\widehat{-}$'' in the notation of
derived $p$-complete or $(p,I)$-complete version of the usual constructions:
Theses include tensor products, pullback functors, Hodge and de Rham complexes.
	
A (decreasingly) filtered complex $\Fil^\bullet C$ is defined as an object in the filtered $\infty$-category $\mathrm{Fun}(\mathbb{Z}^\mathrm{op}, D(\mathbb{Z}))$, 
where the latter is equivalent to the $\infty$-category of quasi-coherent complexes over the algebraic stack $\mathbb{A}^1/\mathbb{G}_m$ and is naturally equipped with a symmetric monoidal structure.
A filtered algebra is then defined as an $\mathbb{E}_\infty$-algebra over the filtered $\infty$-category.
Similarly, we can define the notion of a graded complex and a graded algebra in the derived content.
Here for a filtered complex $\Fil^\bullet C$, we use $\Gr^\bullet C=\bigoplus \Gr^i C$ to denote its associated graded complex,
and use $\Fil^\bullet C \langle i\rangle$ to denote its $i$-th shift of filtered degree, 
whose value at $n\in \mathbb{Z}$ is $\Fil^{i+n}C$.
For our applications, we also use $\Fil_\bullet C$ and $\Gr_\bullet C$ to denote an increasingly filtered complex and its associated graded complex.
For more details, we refer to \cite[\S 5.1]{BMS19} and \cite[\S 2.2.1]{Bha23} for brief introductions on filtered and graded objects.
	
We always fix $K$ to be a complete $p$-adic discretely valued field with perfect residue field $k$, and let $\mathcal{O}_K$ be the ring of integers. 
To differentiate the notions, we use $\mathcal{E}$ to denote an ($F$-)crystal, and use $E$ to denote an ($F$-)gauge.
 
\subsection*{Acknowledgements}
The intellectual debt of this paper to Bhatt's notes \cite{Bha23} should be obvious to readers.
We have benefited from helpful conversations with the following 
mathematicians: Johannes Ansch\"{u}tz, Bhargav Bhatt, Heng Du, Hui Gao, Arthur-C\'{e}sar Le Bras,
Tong Liu, Shubhodip Mondal, Sasha Petrov,
Emanuel Reinecke, Peter Scholze, Yupeng Wang, and Daxin Xu.
We would like to express special thanks to Bhargav Bhatt for discussions on various parts
of this paper.
We are grateful to Tong Liu for mentioning this problem to us.
During the writing of this paper, 
HG is supported by Max Planck Institut f\"ur Mathematik, and
SL is supported by the National Key R $\&$ D Program of China No.~2023YFA1009701 and
the National Natural Science Foundation of China (No.~12288201).

\addtocontents{toc}{\protect\setcounter{tocdepth}{2}}
\section{Graded and filtered homological algebra}
\label{filtered homological algebra}
In this section, we collect various preliminary results concerning graded (resp.~filtered)
complexes over graded (resp.~filtered) rings.
Throughout the whole section, we use the following notation:
\begin{notation}
Let $A$ be a ordinary commutative ring with a finitely generated ideal $J \subset A$
such that $A$ is derived $J$-complete.\footnote{
Later on, whenever we use the results of this section, we are always in the situation where either
$A = \mathbb{Z}_p$ and $J = (p)$, or $(A, I)$ is a bounded prism (whose choice shall be obvious
from the context) and $J = (p, I)$.}

We shall refer to $\mathrm{CAlg}(\mathrm{Fun}(\mathbb{Z}^{\mathrm{disc}}, \mathcal{D}_{J\text{-comp}}(A)))$
as the category of derived $J$-complete graded $\mathbb{E}_{\infty}$-$A$-algebras
indexed by $\mathbb{Z}$.
We say such a graded algebra $B_{\bullet} = \oplus_{i \in \mathbb{Z}} B_i$ is \emph{indexed by $\mathbb{N}$} if 
$B_i = 0$ for all $i < 0$.
We say such a graded algebra $B_{\bullet} = \oplus_{i \in \mathbb{Z}} B_i$ is \emph{connective}
if $B_i \in \mathcal{D}(A)$ is connective for all $i \in \mathbb{Z}$.

We shall refer to $\mathrm{CAlg}(\mathrm{Fun}(\mathbb{Z}^{\mathrm{op}}, \mathcal{D}_{J\text{-comp}}(A)))$
as the category of derived $J$-complete (decreasing) filtered $\mathbb{E}_{\infty}$-$A$-algebras
indexed by $\mathbb{Z}$.
We say such a filtered algebra $\Fil^{\bullet}B$ is \emph{indexed by $\mathbb{N}$} if the induced map
$\Fil^i B \to \Fil^{i-1} B$ is an equivalence for all $i \leq 0$.
We say such a filtered algebra $\Fil^{\bullet}B$ is \emph{connective}
if $\Fil^i B \in \mathcal{D}(A)$ is connective for all $i \in \mathbb{Z}$.
We use $D^{\mathrm{cn}}(A)$ to denote the category of connective complexes.

Whenever there is no risk of confusion, we shall ease the notation by dropping
$L$ and $\widehat{-}$ from the top of $\otimes$: Unless said otherwise, all the tensors are
($J$-complete) derived tensors.
\end{notation}

\subsection{Finite projective modules and perfect complexes}

In this subsection, we shall define the notion of finite projective modules
and perfect complexes in the context of graded or filtered algebras.
We begin with discussing the notion of finite projective modules,
for simplicity we shall restrict ourselves to the connective situation.

\begin{digression}[{\cite[\S 7.2.2]{Lur17}}]
\label{VB digression}
Let $R$ be a connective $\mathbb{E}_1$-ring, then define an object $M \in \mathcal{D}(R)$
to be a \emph{finite projective module} if it is a retract of $R^{\oplus i}$ for some finite natural number $i$.
This is equivalent to any of the following\footnote{See \cite[Corollary 7.2.9 and Proposition 7.2.2.18]{Lur17} for the first two conditions,
and the equivalence to the last condition follows from next sentence.}:
\begin{itemize}
\item being compact projective in $\mathcal{D}^{\mathrm{cn}}(R)$;
\item being flat over $R$ and $\pi_0(M)$ is finitely generated over $\pi_0(R)$;
\item the base change $\pi_0(R) \otimes_R M$ is finite projective over $\pi_0(R)$ (in the usual sense).
\end{itemize}
Suppose $R \to R'$ is a map of connective $\mathbb{E}_1$-rings inducing
$\pi_0(R) \xrightarrow{\cong} \pi_0(R')$, then the base change functor induces an equivalence of
homotopy categories $\mathrm{hProj^{\mathrm{fin}}}(R) \xrightarrow[\cong]{R' \otimes_R -} \mathrm{hProj^{\mathrm{fin}}}(R')$,
see \cite[Corollary 7.2.2.19]{Lur17}.
\end{digression}

Let $R_{\bullet}$ be a connective graded $\mathbb{E}_1$-ring.

\begin{definition}
\label{Definition: graded finite projective}
Define an object $M_{\bullet} \in \DG(R_{\bullet})$ to be 
a graded finite projective module if it is a retract of a finite direct sum of 
$R_{\bullet}\langle i \rangle$'s.
Here $(-)\langle i \rangle$ denotes shift of grading by $i$.
\end{definition}

Recall that given a graded $\mathbb{E}_1$-ring $R_{\bullet}$, we may form its
\emph{direct sum totalization} (abbreviated as totalization below)
$\bigoplus_{i \in \mathbb{Z}} R_i$, and
similarly for a graded left $R_{\bullet}$-complex, we may form 
$\bigoplus_{i \in \mathbb{Z}} M_i$. These are $\mathbb{E}_1$-ring
and left modules over it respectively.

\begin{lemma}
\label{VB in terms of totalization}
An object $M_{\bullet}\in \DG(R_{\bullet})$ is a graded finite projective
module if and only if the totalization
$\bigoplus_{i \in \mathbb{Z}} M_i$ is a finite projective module over
$\bigoplus_{i \in \mathbb{Z}} R_i$.
\end{lemma}

\begin{proof}
The ``only if'' part is clear, let us prove the ``if'' part.
When the totalization is finite projective, choose a finite set of generators of
$\pi_0(\bigoplus_{i \in \mathbb{Z}} M_i) = \bigoplus_{i \in \mathbb{Z}} \pi_0(M_i)$
and consider their grading-wise components,
we see that there is a map 
$f \colon \bigoplus_{i = 0}^m R_{\bullet}\langle n_i \rangle \to M_{\bullet}$
inducing surjection after passing to $\pi_0$ on each graded piece.
Our condition implies that the totalization of the above map admits a section $s$.
Using the following decomposition
\[
\mathrm{Map}_{\bigoplus_{i \in \mathbb{Z}} R_i}(\bigoplus_{i \in \mathbb{Z}} M_i, \bigoplus_{i \in \mathbb{Z}} N_i)
= \bigoplus_{i \in \mathbb{Z}} \mathrm{Map}_{R_{\bullet}}(M, N_{\bullet}\langle i \rangle),
\]
let us look at the component of $s$ in degree $0$, denote it by 
$s_0 \colon M_{\bullet} \to \bigoplus_{i = 0}^m R_{\bullet}\langle n_i \rangle$.
We need to show that $s_0 \circ f$ is an isomorphism on $M_{\bullet}$ or, what is the same,
it induces an isomorphism on totalization.
By construction, we know that it is an endomorphism of flat $(\bigoplus_{i \in \mathbb{Z}} R_i)$-module 
inducing an isomorphism on $\pi_0$, therefore it must induce isomorphisms on all $\pi_i$'s
by \cite[Definition 7.2.2.10]{Lur09}.
\end{proof}

Let $B_{\bullet}$ be a connective derived $J$-complete graded $\mathbb{E}_{\infty}$-$A$-algebras
indexed by $\mathbb{Z}$.

\begin{definition}
\label{Definition: graded complete finite projective}
Define an object $M_{\bullet} \in \DG_{J\text{-comp}}(B_{\bullet})$ to be 
a \emph{graded $J$-completely finite projective module} if its base change $A/J \otimes_A M_{\bullet}$ is
a graded finite projective $A/J \otimes_A B_{\bullet}$-module.
\end{definition}

\begin{lemma}
\label{graded VB only depends on V(J)}
The notion of being graded $J$-completely finite projective only depends on $V(J) \subset \mathrm{Spec}(A)$:
Namely if $J'$ is another finitely generated $A$-algebra with $V(J) = V(J')$, then
being graded $J$-completely finite projective is equivalent to being
graded $J'$-completely finite projective.
\end{lemma}

Note that by \cite[\href{https://stacks.math.columbia.edu/tag/091Q}{Tag 091Q}]{stacks-project},
the notion of being derived $J$-complete only depends on the underlying subset of $V(J)$.

\begin{proof}
It suffices to show that being graded $J$-completely finite projective implies being
graded $J^2$-completely finite projective.
Using \Cref{VB in terms of totalization} and the third characterization in
\Cref{VB digression}, it suffices to show:
If $R \twoheadrightarrow R/J$ is a square-zero thickening of ordinary commutative rings with finitely generated kernel, 
then an $R$-complex
is finite projective if and only if its base change to $R/J$ is so.
Let $M$ be such an $R$-complex, using the triangle
$J \otimes_{R/J} (R/J \otimes_R M) \to M \to R/J \otimes_R M$, we see that $M$ is an ordinary $R$-module,
and our statement follows from \cite[\href{https://stacks.math.columbia.edu/tag/051C}{Tag 051C}]{stacks-project}.
\end{proof}

\begin{remark}
Let us point out that, even in the ungraded setting,
being $J$-completely finite projective does not imply finite projective when the ring is not connective,
see \Cref{Bhargav and Emanuel example} for such an example.
\end{remark}

For our purpose, we need analogous results in the filtered setting.
Via Rees's construction \cite[\S 2.2]{Bha23}, we are led to consider analogs
of the above in the graded setting.
Let $\Fil^{\bullet}R$ be a connective filtered $\mathbb{E}_1$-ring.

\begin{definition}
\label{Definition: filtered finite projective}
Define an object $\Fil^{\bullet}M \in \DF(\Fil^{\bullet}R)$ to be 
a \emph{filtered finite projective module} if it is a retract of a finite direct sum of 
$\Fil^{\bullet}R\langle i \rangle$'s.
Here $(-)\langle i \rangle$ denotes shift of filtration by $i$.
\end{definition}

Recall that given a filtered $\mathbb{E}_1$-ring $\Fil^{\bullet}R$, we may first view it 
as a graded ring by forgetting the arrows $\Fil^i \to \Fil^{i-1}$, then we can form its
totalization $\bigoplus_{i \in \mathbb{Z}} \Fil^{i}R$, this process is nothing but the Rees's construction.
Similarly we may perform it for filtered modules over filtered rings.

\begin{lemma}
\label{VB in terms of Rees}
An object $\Fil^{\bullet}M \in \DF(\Fil^{\bullet}R)$ is a filtered finite projective
module if and only if its Rees's construction
$\bigoplus_{i \in \mathbb{Z}} \Fil^{i}M$ is a finite projective module over
$\bigoplus_{i \in \mathbb{Z}} \Fil^{i}R$.
\end{lemma}

\begin{proof}
The proof is analogous to that of \Cref{VB in terms of totalization},
since we still have the following decomposition
\[
\mathrm{Map}_{\bigoplus_{i \in \mathbb{Z}} \Fil^{i}R}(\bigoplus_{i \in \mathbb{Z}} \Fil^{i}M, \bigoplus_{i \in \mathbb{Z}} \Fil^{i}N)
= \bigoplus_{i \in \mathbb{Z}} \mathrm{Map}_{\Fil^{\bullet}R}(M, N_{\bullet}\langle i \rangle).
\]
\end{proof}

Let $\Fil^{\bullet}B$ be a connective derived $J$-complete filtered $\mathbb{E}_{\infty}$-$A$-algebras
indexed by $\mathbb{Z}$.

\begin{definition}
\label{Definition: filtered complete finite projective}
Define an object $\Fil^{\bullet}M \in \DF_{J\text{-comp}}(\Fil^{\bullet}B)$ to be 
a filtered $J$-completely finite projective module if its base change $A/J \otimes_A M_{\bullet}$ is
a filtered finite projective $A/J \otimes_A \Fil^{\bullet}B$-module.
\end{definition}

Following the exact same proof of \Cref{graded VB only depends on V(J)}, we get the following:

\begin{lemma}
\label{filtered VB only depends on V(J)}
The notion of being filtered $J$-completely finite projective only depends on $V(J) \subset \mathrm{Spec}(A)$:
Namely if $J'$ is another finitely generated $A$-algebra with $V(J) = V(J')$, then
being filtered $J$-completely finite projective is equivalent to being
filtered $J'$-completely finite projective.
\end{lemma}

Next, let us define perfect complexes in the context of graded or filtered algebras.

\begin{digression}[{\cite[\S 7.2.4]{Lur17}}]
\label{Perf digression}
We begin by recalling \cite[Definition 7.2.4.1]{Lur17}:
Let $R$ be an $\mathbb{E}_1$-ring, then the subcategory of left perfect $R$-complexes
in the $\infty$-category of left $R$-complexes is defined as the smallest stable subcategory
containing $R$ and closed under retracts, objects of this subcategory are called
left perfect $R$-complexes.
Immediately after this definition, Lurie proved \cite[Proposition 7.2.4.2]{Lur17}:
In the above setting, left perfect $R$-complexes are the same as compact objects
among left $R$-complexes, and they generate the $\infty$-category of left $R$-complexes.
In fact there is a third characterization of perfectness \cite[Proposition 7.2.4.4]{Lur17}:
it is equivalent to being dualizable.
\end{digression}

Mimicking the above, we arrive at the following:

\begin{definition}
\label{Definition: graded perfect}
Let $R_{\bullet}$ be a graded $\mathbb{E}_1$-ring, then 
the subcategory of graded left perfect $R_{\bullet}$-complexes
in the $\infty$-category of graded left $R_{\bullet}$-complexes is defined as the smallest stable subcategory
containing $R_{\bullet}\langle i \rangle$ for all $i$ and closed under retracts.
We denote this subcategory by $\mathcal{G}\mathrm{LMod}^{\perf}_{R_{\bullet}}$,
objects of which are called
graded left perfect $R_{\bullet}$-complexes.
\end{definition}

Here is an analog of \cite[Proposition 7.2.4.2]{Lur17}.

\begin{proposition}
\label{graded perfect criterion}
Let $R_{\bullet}$ be a graded $\mathbb{E}_1$-ring. Then:
\begin{enumerate}
\item The $\infty$-category $\mathcal{G}\mathrm{LMod}_{R^{\bullet}}$
is compactly generated.
\item An object of $\mathcal{G}\mathrm{LMod}_{R^{\bullet}}$
is compact if and only if it is perfect.
\end{enumerate}
Moreover the property of being perfect/compact is equivalent to
the totalization being perfect/compact as a complex over the totalization
of $R_{\bullet}$.
\end{proposition}

\begin{proof}
The proof of (1) and (2) are exactly the same as that in loc.~cit:
Tracing through the proof there, we just need to observe that if
$\mathrm{Map}_{R_{\bullet}}(R_{\bullet}\langle i \rangle[j], M_{\bullet}) = 0$
for all $i, j$, then $M_{\bullet} = 0$.
As for the last statement: it simply follows from the fact
that taking totalization has a right adjoint preserving filtered colimit:
\[
\mathrm{LMod}_{\bigoplus_{i} R_i} \to \mathcal{G}\mathrm{LMod}_{R_{\bullet}},
M \mapsto M \otimes_{(\bigoplus_{i} R_i)} 
\left(\bigoplus_{j \in \mathbb{Z}} R_{\bullet}\langle j \rangle\right),
\]
where the map $\bigoplus_{i} R_i \to \bigoplus_{j \in \mathbb{Z}} R_{\bullet}\langle j \rangle$
is the canonical map realizing the source as the $\deg = 0$ piece of the target.
\end{proof}

\begin{remark}
If $R_{\bullet}$ is a graded discrete commutative ring, then graded $R_{\bullet}$-complexes
are the same as quasi-coherent complexes on the stack $[\mathrm{Spec}(\bigoplus_i R_i)/\mathbb{G}_m]$,
whose perfectness is defined by that of its pullback to $\mathrm{Spec}(\bigoplus_i R_i)$.
This is a well-defined notion
due to flat descent of perfectness 
(see \cite[\href{https://stacks.math.columbia.edu/tag/066X}{Tag 066X}]{stacks-project}
and \cite[\href{https://stacks.math.columbia.edu/tag/068T}{Tag 068T}]{stacks-project}).
Our proposition above shows that in this special case, the perfectness in classical sense
agrees with our definition here.
\end{remark}

Via Rees's construction \cite[\S 2.2]{Bha23}, we can transport the above discussion from
graded setting to filtered setting.

\begin{definition}
\label{Definition: filtered perfect}
Let $\Fil^{\bullet} S$ be a filtered $\mathbb{E}_1$-ring,
then we may view it as a graded $\mathbb{E}_1$-ring.
We then define the $\infty$-category of filtered left $\Fil^{\bullet} S$-modules
as $\mathcal{F}\mathrm{LMod}_{\Fil^{\bullet} S} \coloneqq \mathcal{G}\mathrm{LMod}_{\Fil^{\bullet} S}$
(c.f.~\cite[Proposition 2.2.6]{Bha23}).
We then define a filtered left complex $\Fil^{\bullet} M$ to be perfect if its image under the above
equivalence is perfect.
Since the above is an equivalence, we immediately see that the subcategory of filtered left perfect
$\Fil^{\bullet} S$-complexes can be alternatively defined as the smallest stable subcategory
containing $(\Fil^{\bullet} S)\langle i \rangle$ for all $i$ and closed under retracts.
Here once again, we use $(-)\langle i \rangle$ to denote shift of filtration by $i$.
\end{definition}

The following is a simple translation of \Cref{graded perfect criterion}
via the Rees equivalence.

\begin{proposition}
\label{filtered perfect criterion}
Let $\Fil^{\bullet} S$ be a filtered $\mathbb{E}_1$-ring. Then:
\begin{enumerate}
\item The $\infty$-category $\mathcal{F}\mathrm{LMod}_{\Fil^{\bullet} S}$
is compactly generated.
\item An object of $\mathcal{F}\mathrm{LMod}_{\Fil^{\bullet} S}$
is compact if and only if it is perfect.
\end{enumerate}
Moreover the property of being perfect/compact is equivalent to
its underlying Rees construction being perfect/compact as a $\Rees(\Fil^{\bullet} S)$-complex.
\end{proposition}

We also get the following useful property of filtered left perfect complexes over
a complete filtered ring.

\begin{proposition}
\label{perfect over complete is complete}
Let $\Fil^{\bullet} S$ be a filtered $\mathbb{E}_1$-ring which is complete with respect
to its filtration.
Any filtered left perfect $\Fil^{\bullet} S$-complex is complete with respect to its filtration.
\end{proposition}

\begin{proof}
We simply observe that the subcategory spanned by filtered complexes being complete with respect
to its filtration is a stable subcategory
containing $(\Fil^{\bullet} S)\langle i \rangle$ for all $i$ and closed under retracts.
Therefore it must contain the subcategory of perfect complexes.
\end{proof}

Since we work with derived complete complexes, we need to introduce the following variant.
Let $\Fil^{\bullet} B$ be a derived $J$-complete (decreasing) filtered $\mathbb{E}_{\infty}$-$A$-algebras
indexed by $\mathbb{Z}$.
\begin{definition}
\label{Definition: filtered complete perfect}
Define an object $\Fil^{\bullet}M \in \DF_{J\text{-comp}}(\Fil^{\bullet}B)$ to be 
\emph{filtered $J$-completely perfect} if its base change $A/J \otimes_A \Fil^{\bullet}M$ is
filtered perfect over $A/J \otimes_A \Fil^{\bullet}B$.
We denote this full subcategory by $\DF^{\perf}_{J\text{-comp}}(\Fil^{\bullet}B)$.
%
%
\end{definition}

Note that the inclusion 
$\DF^{\perf}(\Fil^{\bullet}B) \subset 
\DF^{\perf}_{J\text{-comp}}(\Fil^{\bullet}B)$ need not be an equivalence,
below we sketch an example due to Bhatt and Reinecke which we learn from
private communications.

\begin{remark}
\label{Bhargav and Emanuel example}
Bhargav Bhatt and Emanuel Reinecke gave the following example where perfectness and $p$-completely
perfect genuinely differs. Let us sketch their idea:
Let $R = \mathrm{dR}_{\mathbb{Z}_p\langle x \rangle/\mathbb{Z}_p}^{\widehat{}}$, and take any connection
$\nabla \colon \mathbb{Z}_p\langle x \rangle \to \mathbb{Z}_p\langle x \rangle \cdot dx$ such that
\begin{enumerate}
\item $\nabla \equiv d$ modulo $p$; and
\item $\nabla = 0$ has no solution.
\end{enumerate}
Then we claim that 
the complex $M \coloneqq[\mathbb{Z}_p\langle x \rangle \xrightarrow{\nabla} \mathbb{Z}_p\langle x \rangle \cdot dx]$
gives an example of a $p$-completely perfect $R$-complex which is not perfect.
It is $p$-completely free of rank $1$ by the condition (1) above, let us show it is not perfect.
Suppose otherwise, consider $M \otimes^L_R \mathbb{Z}_p\langle x \rangle$,
which is a perfect $\mathbb{Z}_p\langle x \rangle$-complex, whose reduction modulo $p$
is $\mathbb{F}_p[x]$ due to condition (1) above.
Therefore we see that the $\mathbb{Z}_p$-complex $M \otimes^L_R \mathbb{Z}_p\langle x \rangle$ is non-canonically isomorphic
to $\mathbb{Z}_p\langle x \rangle$.
On the other hand, since $\mathrm{H}^0(R[1/p]) = \mathbb{Q}_p$, one can show that
the $R[1/p]$-complex $\mathbb{Q}_p\langle x \rangle$ admits an exhaustive increasing filtration
with graded pieces given by (infinite) direct sums of $R[1/p]$.
Therefore we see that $\bigl(M \otimes^L_R \mathbb{Z}_p\langle x \rangle\bigr)[1/p]$
admits an exhaustive increasing filtration with graded pieces given by 
direct sums of $M[1/p]$, which has only cohomology in degree $1$ by condition (2) above.
This gives a contradiction as desired.
In fact, one can equip both $M$ and $R$ with Hodge filtration
so that it becomes an example illustrating
$\DF^{\perf}(\Fil^{\bullet}B) \subsetneqq
\DF^{\perf}_{J\text{-comp}}(\Fil^{\bullet}B)$, with $(A, J) = (\mathbb{Z}_p, (p))$.
\end{remark}

In familiar situations, the difference between perfectness and complete perfectness disappears.

\begin{lemma}[{\cite[\href{https://stacks.math.columbia.edu/tag/09AW}{Tag 09AW}]{stacks-project}}]
\label{perf vs completely perf in classical situation}
Let notation be as in \Cref{Definition: filtered complete perfect}.
If $R$ is concentrated in degree $0$ and is $J$-adically complete,
then a $J$-complete $R$-complex
is perfect if and only if it is $J$-completely perfect.
\end{lemma}

It is easy to see that in the above situation, the difference between being finite projective
versus being $J$-completely finite projective also disappears.
Let us show that the same is true as long as the $R$ is a connective derived $J$-complete $\mathbb{E}_{\infty}$-$A$-algebra.

\begin{proposition}
\label{vbperf vs cplt vbperf}
Let $R$ be a connective derived $J$-complete $\mathbb{E}_{\infty}$-$A$-algebra.
Let $M \in \mathcal{D}_{J\text{-comp}}(R)$, then
\begin{enumerate}
\item it is finite projective if and only if it is $J$-completely finite projective;
\item it is perfect if and only if it is $J$-completely perfect.
\end{enumerate}
\end{proposition}

\begin{proof}
The only if part is clear. Let us prove the if part for (1).
Our assumption implies that $M$ is connective, and $\pi_0(M)/J$ is finite projective over
$\pi_0(R)/J$. Choose a finite set of generators and lift to $\pi_0(M)$, we obtain a map
$f \colon R^{\oplus n} \to M$.
Our task is to show that this map admits a splitting given that it does so after applying
$A/J \otimes_A -$.
Choose a finite set of generators $J = (a_1, \ldots, a_m)$, for each $k$, let us consider the
base change of $f$ along $\mathrm{Kos}(R; a_i^k) \otimes_R -$.
Combining \cite[\href{https://stacks.math.columbia.edu/tag/051C}{Tag 051C}]{stacks-project}
and \cite[Corollary 7.2.2.19]{Lur17}, we see that the map $\mathrm{Kos}(R; a_i^k) \otimes_R f$
is again a map from finite free $\mathrm{Kos}(R; a_i^k)$-complex to a finite projective
$\mathrm{Kos}(R; a_i^k)$-complex inducing surjection on $\pi_0$, therefore it admits a section.
Now we compute the mapping space:
\[
\mathrm{Map}_R(M, R^{\oplus n}) \cong \lim_k \mathrm{Map}_R(M, \mathrm{Kos}(R; a_i^k)^{\oplus n})
\cong \lim_k \mathrm{Map}_{\mathrm{Kos}(R; a_i^k)}(\mathrm{Kos}(M; a_i^k), \mathrm{Kos}(R; a_i^k)^{\oplus n}),
\]
and observe that the last term has $\pi_0$, for each $k$, given by
$\mathrm{Hom}_{\pi_0(R)/(a_i^k)}(\pi_0(M)/(a_i^k), \pi_0(R)/(a_i^k)^{\oplus n})$.\footnote{
This follows from the fact that $\mathrm{Kos}(M; a_i^k)$ is a finite projective
$\mathrm{Kos}(R; a_i^k)$-complex.
}
In particular, by Milnor's exact sequence $\mathrm{Hom}_R(M, R^{\oplus n}) \twoheadrightarrow 
\lim_k \mathrm{Hom}_{\pi_0(R)/(a_i^k)}(M/(a_i^k), \pi_0(R)/(a_i^k)^{\oplus n})$,
we may successively lift the section when $k = 1$ along the inverse system.
In conclusion, we have found a map $s \colon M \to R^{\oplus n}$ such that
$\pi_0(R/J) \otimes_R (f \circ s)$ is identity.
This implies first that $A/J \otimes_A (f \circ s)$, being an endomorphism of a finite projective
$A/J \otimes_A R$-module inducing identity on $\pi_0$, is an isomorphism.
Then using derived Nakayama, this shows that $f \circ s$ is actually an isomorphism.

It is straightforward to deduce (2) out of (1): Let us perform induction on the length of
the $\pi_0(R)/J$-Tor-amplitude interval of $\pi_0(R)/J \otimes_R M$.
When the length is $0$, it follows from (2) that $M$ is a shift of a finite projective $R$-module.
Suppose the case of length $\ell$ is proved, then without loss of generality assume
the said Tor-amplitude is in $[- \ell - 1, 0]$. By derived completeness, we see that
$M$ is connective. Let us choose a map $R^{\oplus n} \to M$ inducing a surjection on $\pi_0(-)/J$.
Then the cone is derived $J$-complete, and the said Tor-amplitude length is drop by $1$, so we win
by induction.
\end{proof}

Later on in this section, we shall see that the difference between
graded/filtered perfectness (resp.~finite projective) and completely perfectness
(resp.~completely finite projective) also disappears
under mild condition: see \Cref{graded: perf vs cplt perf}, \Cref{graded: vb vs cplt vb},
\Cref{criteria for filtered perfect}, and \Cref{criteria for filtered vb}.

\subsection{Graded homological algebra}
Throughout this subsection, let $B_{\bullet}$
be a derived $J$-complete graded $\mathbb{E}_{\infty}$-$A$-algebras
\emph{indexed by $\mathbb{N}$}.
To start, we first observe the following reduction functor:

\begin{construction}
\label{reduction of graded B-cplx}
Since $B_{\bullet}$ is indexed by $\mathbb{N}$,
projecting to its degree $0$ piece
is a map of derived $J$-complete graded $\mathbb{E}_{\infty}$-$A$-algebras.
Base changing ($J$-completely) along this map gives us a ``reduction functor'':
\[
\begin{tikzcd}
\Red_B \colonequals  \DG_{J\text{-comp}}(B_{\bullet}) 
\arrow[rr, "- \otimes_{B_{\bullet}} B_0"] &&  \DG_{J\text{-comp}}(B_0);\\
M^\bullet \arrow[rr, mapsto] && M^\bullet \otimes_{B_{\bullet}} B_0.
\end{tikzcd}
\]
Here let us emphasize again that $\otimes$ denotes the $J$-completely derived tensor product.
The $i$-th degree piece of $\Red_B(-)$ is denoted as $\Red_{i,B}(-)$.
When the $B_{\bullet}$ is clear from the context, we omit $B$ in the subscripts and use $\Red$ and $\Red_i$
to abbreviate $\Red_B$ and $\Red_{i,B}$.
\end{construction}

We are interested in the following subcategory of graded $B_{\bullet}$-complexes.
\begin{definition}
We say an $M_{\bullet} \in \DG_{J\text{-comp}}(B_{\bullet})$ has \emph{bounded below grading}
if $M_i = 0$ for all $i \ll 0$.
\end{definition}


Using the reduction functor, we now introduce a \emph{canonical weight}
filtration on those $M_{\bullet}$'s with bounded below grading.
\begin{proposition}
\label{prop weight filtration of graded B-cplx}
Let $M_{\bullet} \in \DG_{J\text{-comp}}(B_{\bullet})$ has bounded below grading,
there is a unique increasing filtration
$\Fil^\wt_i (M_\bullet)$ on $M_\bullet$ satisfying the following axioms:
\begin{enumerate}
\item The filtration is exhaustive;
\item the filtration ``starts at $0$'': i.e.~$\Fil^\wt_{\ll 0} (M_\bullet) = 0$; and
\item the graded piece $\Gr^\wt_i(M_\bullet) \simeq N_i \otimes_{B_0} B_{\bullet}$ 
where $N_i$ is a $J$-complete $B_0$-complex with grading $i$.
\end{enumerate}
Moreover, the induced filtration $\Fil^\wt_i(\Red(M_{\bullet})) \coloneqq \Red(\Fil^\wt_i (M_\bullet))$
is the one induced by grading: 
\[
\Fil^\wt_i(\Red(M_{\bullet})) \cong \bigoplus_{j \leq i} \Red_j(M_{\bullet}).
\]
In particular, there are natural isomorphisms 
$N_i \cong \Red(\Gr^\wt_i(M_\bullet)) \cong \Gr^\wt_i(\Red(M_{\bullet})) \cong \Red_i(M_{\bullet})$.
\end{proposition}

\begin{proof}
We show the uniqueness first.
Let $c$ be the least integer such that $M_c \neq 0$, which exists by assumption.
Let $\Fil_{\ast}(M_{\bullet})$ be any filtration satisfying the list of axioms.
Let $d$ be the least integer such that $\Fil_{d}(M_{\bullet}) \not= 0$, which exists
as $M_{\bullet} \not= 0$ and the exhaustive axiom (1).
By our condition (3), we see that the grading $\leq d$ piece
of the map $\Fil_{i}(M_{\bullet}) \to \Fil_{i + 1}(M_{\bullet})$ is an isomorphism
whenever $i \geq d$. Our assumption on $d$ implies that
\[
0 \not= N_d \cong (\Gr_{d}(M_{\bullet}))_d \cong (\Fil_{\geq d}(M_{\bullet}))_d \cong M_d
\]
and
\[
0 = (\Fil_{\geq d}(M_{\bullet}))_{< d} \cong M_{< d}.
\]
where the last isomorphisms follow from the axiom (1). This shows that $d = c$.
By the above reasoning, we necessarily have $\Fil_{< c}(M_{\bullet}) = 0$
and $\Fil_c(M_{\bullet}) = M_c \otimes_{B_0} B_{\bullet}$ with map to $M_{\bullet}$
induced by the natural map $M_c = M_c \subset M_{\bullet}$ of graded $B_0$-complexes.
We observe that the choice of $\Fil_c(M_{\bullet})$ is exactly so that $(M_{\bullet}/\Fil_c(M_{\bullet}))_{\leq c} = 0$.

Suppose $\Fil_{\leq d}(M_{\bullet})$ is determined such that $(M_{\bullet}/\Fil_d(M_{\bullet}))_{\leq d} = 0$,
then to determine $\Fil_{d+1}$ is the same as to determine $\Fil_{d+1}(M_{\bullet}/\Fil_d(M_{\bullet}))$,
which we have seen that it must be given by $(M_{\bullet}/\Fil_d(M_{\bullet}))_{d+1} \otimes_{B_0} B_{\bullet}$.
Moreover this unique choice exactly makes that $(M_{\bullet}/\Fil_{d +1}(M_{\bullet}))_{\leq (d+1)} = 0$.
Therefore we have the algorithm of determining $\Fil_\ast(M_{\bullet})$, which is uniqueness.

To see existence, we simply notice that the filtration produced by the algorithm from the previous paragraph
satisfies axioms by how it is constructed: (2) and (3) is easy to see, the exhaustive axiom (1) follows from
the property that $(M_{\bullet}/\Fil_d(M_{\bullet}))_{\leq d} = 0$.

The induced filtration on $\Red(M_{\bullet})$ is an increasing exhaustive filtration, starting at $0$,
whose graded pieces are concentrated in different grading. There is no choice but the filtration
induced by grading.
\end{proof}

\begin{corollary}
\label{graded: perf vs cplt perf}
Suppose that $B_{\bullet}$ is indexed by $\mathbb{N}$.
Let $M_{\bullet} \in \DG_{J\text{-comp}}(B_{\bullet})$, then it is graded ($J$-completely) perfect if and only if
it is has bounded below grading and $\Red_B(M_{\bullet}) \in \DG_{J\text{-comp}}(B_{0})$ is graded ($J$-completely) perfect.
\end{corollary}

\begin{proof}
Only if part follows from the fact that $B_{\bullet}$ is indexed by $\mathbb{N}$ (hence has bounded below grading)
and that $\Red_B$ is defined as a base change hence preserves properties of being graded ($J$-completely) perfect.
The reverse follows immediately from \Cref{prop weight filtration of graded B-cplx}
and the fact that a $J$-complete graded $B_0$-complex is graded ($J$-completely) perfect
if and only if it is concentrated in finitely many grading and each grading piece
is ($J$-completely) perfect.
\end{proof}

\begin{corollary}
\label{graded: vb vs cplt vb}
Suppose that $B_{\bullet}$ is indexed by $\mathbb{N}$ and connective.
Then $M_{\bullet} \in \DG_{J\text{-comp}}(B_{\bullet})$
is graded ($J$-completely) finite projective if and only if it has bounded below grading and 
$\Red_B(M_{\bullet}) \in \DG_{J\text{-comp}}(B_{\bullet})$ is graded ($J$-completely) finite projective.
Moreover $M_{\bullet}$ is graded finite projective (resp.~graded perfect)
if and only if it is graded $J$-completely finite projective (resp.~graded $J$-completely perfect).
\end{corollary}

\begin{proof}
The equivalence of graded perfect and graded $J$-completely perfect follows from
\Cref{vbperf vs cplt vbperf} and \Cref{graded: perf vs cplt perf}.
Among all four variants of graded finite projective, the condition of being graded finite projective
is the strongest, and the condition of having bounded below grading with reduction being graded $J$-completely finite projective
is the weakest, it suffices to know that the latter implies the former.
Using \Cref{vbperf vs cplt vbperf} again, we see that $\Red_B(M_{\bullet})$ being graded $J$-completely finite projective
is equivalent to it being graded finite projective.
Combining with \Cref{prop weight filtration of graded B-cplx}, we are done since
any extension $N_1 \otimes_{B_0} B_{\bullet} \to M_{\bullet} \to N_2 \otimes_{B_0} B_{\bullet}$,
where $N_i$'s are finite projective $B_0$-modules placed in a single degree, must split.
\end{proof}

For later purposes, let us observe the following graded version of Nakayama lemma.
\begin{lemma}
\label{graded Nakayama}
Assume $M_\bullet \in \DG_{J\text{-comp}}(B_\bullet)$ has bounded below grading.
\begin{enumerate}
\item Assume $M_\bullet \not= 0$ and let $c$ be the smallest integer $n$ such that $M_n\neq 0$.
Then $\Red_{< c, B}(M_{\bullet}) = 0$, and the natural map
$M_c \to \Red_{c, B}(M_{\bullet})$ is an isomorphism.
\item We have that $M_{\leq m} = 0$ if and only if $\Red_{\leq m, B}(M_{\bullet}) = 0$.
\item In particular, we have $M_{\bullet} = 0$
if and only if $\Red_B(M_{\bullet}) = 0$.
\end{enumerate}
\end{lemma}

\begin{proof}
(2) and (3) follow from (1), below we prove (1).
We take the graded tensor product of $M_\bullet$ with the following fiber sequence of $B_\bullet$-graded complexes
\[
B_{\geq 1} \longrightarrow B_\bullet \longrightarrow B_0,
\]
resulting a fiber sequence of graded $B_0$-complexes:
\[
M_\bullet \otimes_{B_\bullet} B_{\geq 1} \to M_\bullet \to \Red_B(M_{\bullet}).
\]
We now claim that $(M_\bullet \otimes_{B_\bullet} B_{\geq 1})$
has no degree $\leq c$ piece, which implies the first statement.
To show the claim, we use the bar resolution of $M_\bullet \otimes_{B_\bullet} B_{\geq 1}$, 
which says that it can be computed by the colimit of the following simplicial diagram 
in the world of derived $J$-complete $A$-complexes:
\[
M_\bullet \otimes_{B_\bullet} B_{\geq 1} \simeq \colim_{[n] \in \Delta} 
(M_\bullet \otimes_{A} B_\bullet \otimes_{A} \cdots \otimes_{A} B_\bullet \otimes_{A} B_{\geq 1}),
\]
where the $n$-th term in the right hand side has $n$ copies of $B_\bullet$ in the tensor product.
Notice that for a given $[n] \in \Delta$, the degree $i$ piece of the $n$-th term is isomorphic to the direct sum
\[
\bigoplus_{u+w_1+\cdots+w_n+v=i} M_u \otimes_{A} B_{w_1} \otimes_{A} \cdots \otimes_{A} B_{w_n} \otimes_{A}B_v,
\]
where each direct summand is zero unless $u \geq c$, $w_j \geq 0$ and $v \geq 1$.
As a consequence, the $\deg \leq c$ piece of the entire simplicial diagram vanishes.
\end{proof}

\subsection{Filtered homological algebra}
Throughout this subsection, let $\Fil^{\bullet}B$ be a \emph{connective}
derived $J$-complete (decreasing) filtered $\mathbb{E}_{\infty}$-$A$-algebras
\emph{indexed by $\mathbb{N}$}.
Denote $B \coloneqq \Fil^0 B = \Fil^i B$ for any $i \leq 0$.

\begin{definition}
\label{constant filtration definition}
An object $\Fil^{\bullet}M \in \DF_{J\text{-comp}}(\Fil^{\bullet} B)$
is called \emph{constant after $i$-th filtration}
if the transition maps $\Fil^j M \to \Fil^{j-1} M$ are isomorphisms for any $j \leq i$.
In such a situation, we have that 
$\Fil^j M \to \Fil^{-\infty} M \coloneqq \bigl(\mathrm{colim}_{i \to -\infty} \Fil^i M\bigr)^{\wedge}_{J}$
is an isomorphism for any $j \leq i$.

We say an object $\Fil^{\bullet}M \in \DF_{J\text{-comp}}(\Fil^{\bullet} B)$ is \emph{$0$ after $i$-th filtration}
if it is constant after $i$-th filtration and $\Fil^{-\infty} M = 0$,
or equivalently if $\Fil^j M = 0$ for any $j \leq i$.

We say a filtered object is \emph{eventually constant} (resp.~\emph{eventually $0$})
if it is constant (resp.~$0$) after $i$-th filtration for some $i$.
We use $\DF^{\text{ev.const}}_{J\text{-comp}}(\Fil^{\bullet}B)$ to denote the full subcategory spanned by 
eventually constant objects.
\end{definition}

Below for a filtered algebra $\Fil^\bullet B$ that is indexed by $\mathbb{N}$,
we shall view the quotient $\Gr^0 B$ as a filtered $\Fil^\bullet B$-algebra with
its positive filtrations being zero.
The following is a key technical result of this subsection.

\begin{proposition}
\label{criteria for filtered perfect}
Assume that every finite projective $\Gr^0 B$-module
can be lifted to a finite projective $B$-module.
Then for an object $\Fil^{\bullet} M \in \DF_{J\text{-comp}}(\Fil^{\bullet} B)$,
the following four conditions are equivalent:
\begin{enumerate}
\item It is filtered perfect;
\item It is filtered $J$-completely perfect;
\item It is eventually constant with perfect $\Fil^{-\infty} M$
and the filtered base change $\Fil^{\bullet} M \otimes_{\Fil^{\bullet}B} \Gr^0 B$ is filtered perfect;
\item It is eventually constant with $J$-completely perfect $\Fil^{-\infty} M$
and the filtered base change $\Fil^{\bullet} M \otimes_{\Fil^{\bullet}B} \Gr^0 B$ is filtered $J$-completely perfect.
\end{enumerate}
\end{proposition}

Note that we are assuming $\Fil^{\bullet}B$ to be connective,
hence it makes sense to talk about finite projective modules over
$B$ or $\Gr^0 B$. 
Moreover we remind readers that our $\otimes$ is always $J$-completely
derived tensor product, so the filtered base change lives in
$\DF_{J\text{-comp}}(\Gr^0 B)$.

It is clear that (1) implies (2), (2) implies (4).
Moreover (3) and (4) are equivalent due to \Cref{vbperf vs cplt vbperf}.
So we just need to show that (3) implies (1).
%
%
%
To that end, we need some preparations.
First we notice that the filtered base change is compatible with taking graded pieces
via graded base changes:

\begin{lemma}
\label{filtered base change and graded base change}
The following diagram commutes:
\[
\xymatrix{
\DF_{J\text{-comp}}(\Fil^{\bullet}B) \ar[rr]^-{- \otimes_{\Fil^{\bullet}B} \Gr^0 B} 
\ar[d]_{\Gr^{\bullet}} & & \DF_{J\text{-comp}}(\Gr^0 B) \ar[d]^{\Gr^{\bullet}} \\
\DG_{J\text{-comp}}(\Gr^\bullet B) \ar[rr]^-{- \otimes_{\Gr^\bullet B} \Gr^0 B}
& & \DG_{J\text{-comp}}(\Gr^0 B).
}
\]
\end{lemma}

\begin{proof}
It follows from the general fact that taking graded is symmetric monoidal, see \cite[Lemma 3.10]{GL21}.
\end{proof}

\begin{lemma}
\label{general lemma on eventual constancy}
Let $\Fil^{\bullet} M \in \DF_{J\text{-comp}}(\Fil^{\bullet}B)$, we have:
\begin{enumerate}
\item It is constant after $i$-th filtration if and only if
$\Gr^{j} M = 0$ for any $j \leq (i-1)$;
\item It is constant after $i$-th filtration if and only if
it is eventually constant and the filtered base change
$\Fil^{\bullet} M \otimes_{\Fil^{\bullet} B} \Gr^0 B$ is constant after
$i$-th filtration.
\item The natural map
\[
\Fil^{-\infty} M \otimes_B \Gr^0 B \to \Fil^{-\infty}(\Fil^{\bullet} M \otimes_{\Fil^{\bullet} B} \Gr^0 B)
\]
is an isomorphism.
\item Assume $\Fil^{\bullet} M$ is 
$0$ after $i$-th filtration, then the filtered base change 
$\Fil^{\bullet} M \otimes_{\Fil^{\bullet} B} \Gr^0 B$ is
$0$ after $i$-th filtration.
\item $\Fil^{\bullet} M=0$ if and only if it is eventually $0$ and
the filtered base change $\Fil^{\bullet} M \otimes_{\Fil^{\bullet} B} \Gr^0 B=0$.
\end{enumerate}
\end{lemma}

\begin{proof}
(1): the transition map $\Fil^j M \to \Fil^{j-1} M$ is an isomorphism if and only if
its cone, which by definition is $\Gr^{j-1} M$, is $0$.

(2): Using (1), the statement follows from \Cref{graded Nakayama}.(2) and
\Cref{filtered base change and graded base change}.

(3): By derived Nakayama, we may reduce to the case where $J = 0$.
In that case, it is a general fact that filtered base change commutes with taking underlying module of the factors
and taking underlying ring of the tensor base.

(4): Let $\Fil^{\bullet} M \in \DF^{\text{ev.const}}_{I\text{-}comp}(I^{\bullet} A)$
and suppose it is $0$ after $i$-th filtration.
Then the filtered base change $\Fil^{\bullet} M \otimes_{\Fil^{\bullet} B} \Gr^0 B$ is:
\begin{itemize}
\item Constant after $i$-th filtration by (2);
\item Eventually $0$ by (3).
\end{itemize}
Therefore the filtered base change is $0$ after $i$-th filtration.

(5): We only need to see the ``if'' part: by combining the condition with
\Cref{graded Nakayama}.(2) and
\Cref{filtered base change and graded base change},
we see that $\Gr^{\bullet} M = 0$. Since $\Fil^{\bullet} M$ is eventually zero,
we see that $\Fil^i M = \Fil^{-\infty} M = 0$.
\end{proof}

\begin{lemma}
\label{first nonzero filtration}
Let us denote the forgetful functor
$\D(\Gr^0 B) \to \D(B)$ by $\iota_*$.
Assume $\Fil^{\bullet} M \in \DF_{J\text{-comp}}(\Fil^{\bullet}B)$ is $0$ after $i$-th filtration,
then the natural arrow $\Fil^{i+1} M \to \iota_*(\Fil^{i+1}(\Fil^{\bullet} M \otimes_{\Fil^{\bullet} B} \Gr^0 B))$
is an isomorphism.
\end{lemma}

\begin{proof}
By the compatibility in \Cref{filtered base change and graded base change}, we have
\[
\Gr^i(\Fil^{\bullet} M \otimes_{\Fil^{\bullet} B} \Gr^0 B) \cong 
\Gr^i(\Gr^{\bullet} M \otimes_{\Gr^\bullet B} \Gr^0 B).
\]
Our claim now follows from \Cref{graded Nakayama}.(1) and 
\Cref{general lemma on eventual constancy}.(4).
\end{proof}

Using the above preparations, we can show that (3) implies (1)
in \Cref{criteria for filtered perfect} when $\Fil^{-\infty}M = 0$.

\begin{proposition}
\label{criteria for filtered perfect (special)}
Assume that every finite projective $\Gr^0 B$-module
can be lifted to a finite projective $B$-module.
Then an object $\Fil^{\bullet} M \in \DF_{J\text{-comp}}(\Fil^{\bullet}B)$
is filtered perfect if
\begin{enumerate}
\item It is eventually $0$; and
\item the filtered base change $\Fil^{\bullet} M \otimes_{\Fil^{\bullet} B} \Gr^0 B$
is filtered perfect.
\end{enumerate}
\end{proposition}

\begin{proof}
Without loss of generality, we may assume that $\Fil^{\bullet} M$ is $0$ after
$0$-th filtration. By our assumption of filtered perfectness of 
$\Fil^{\bullet} M \otimes_{\Fil^{\bullet} B} \Gr^0 B \eqqcolon \Fil^{\bullet} \overline{M}$,
combined with \Cref{general lemma on eventual constancy}.(4), we see that $\Fil^{\bullet} \overline{M}$ looks like
\[
\left(\cdots \xrightarrow{=} 0 = \Fil^n\overline{M} \to \Fil^{n-1}\overline{M} \to \cdots \to \Fil^1\overline{M}
\to 0 = \Fil^0\overline{M} \xrightarrow{=} \cdots \right),
\]
where all $\Fil^i\overline{M} \in \Perf(\Gr^0 B)$.
We shall prove by a double induction, the first layer of which is on the natural number $n$.
When $n = 0$, by \Cref{general lemma on eventual constancy}.(5) we see that
$\Fil^{\bullet} M = 0$, hence trivially filtered perfect.

We need to prove the induction step, we claim that there exists a filtered perfect 
$\Fil^{\bullet}N \in \DF_{J\text{-comp}}(\Fil^{\bullet}B)$, which is $0$
after $0$-th filtration, together with a map
$\Fil^{\bullet}N \to \Fil^{\bullet} M$ such that its cone $\Fil^{\bullet} C$ has reduction
looking like
\[
\left(\cdots \xrightarrow{=} 0 = \Fil^n\overline{C} \to \Fil^{n-1}\overline{C} \to \cdots \to 0 = \Fil^1\overline{C}
\xrightarrow{=} \cdots \right).
\]
Granting the claim, then by shifting the filtration of $\Fil^{\bullet}C$, we see by induction that
$\Fil^{\bullet}C$ is filtered perfect, therefore $\Fil^{\bullet} M$ is filtered perfect as well.

Next we prove the above claim, using induction on the length of tor-amplitude of $\Fil^1\overline{M} \in \Perf(\Gr^0 B)$.
Since we may perform cohomological shift of $\Fil^{\bullet} M$, let us assume the tor-amplitude of
$\Fil^1\overline{M} \in \Perf(\Gr^0 B)$ is in $[0, m]$.
The base case is $m = 0$, namely when $\Fil^1\overline{M}$ is a finite projective $\Gr^0 B$-module.
Our assumption says that there exists a finite projective $B$-module
$V$ with $V \otimes_B \Gr^0 B \cong \Fil^1\overline{M}$. Applying \Cref{first nonzero filtration} with $i = 0$, we see that
there is a map $V \to \Fil^1M \cong \iota_* \Fil^1\overline{M}$ in $D(B)$ inducing
$\iota^*V = V \otimes_B \Gr^0 B \xrightarrow{\cong} \Fil^1\overline{M}$. This implies that there is a map
from the filtered perfect
\[
\Fil^{\bullet} N' \coloneqq (\cdots \xrightarrow{=} 0 = \Fil^2  \to V = \Fil^1 \to 0 = \Fil^0 \xrightarrow{=} \cdots)
\longrightarrow \Fil^{\bullet} M
\]
in $\DF_{J\text{-comp}}(B)$ where the target is viewed as an object via forgetful functor 
$\DF_{J\text{-comp}}(\Fil^{\bullet}B) \to \DF_{J\text{-comp}}(B)$.
By adjunction, we get a map $\Fil^{\bullet} N \coloneqq \Fil^{\bullet} N' \otimes_B (\Fil^{\bullet} B) \to \Fil^{\bullet}M$.
Its reduction is given by
\[
\Fil^{\bullet} \overline{N} = (\cdots \xrightarrow{=} 0 = \Fil^2  \to V \otimes_B \Gr^0 B = \Fil^1 \to 0 = \Fil^0 \xrightarrow{=} \cdots)
\to \Fil^{\bullet} \overline{M}
\]
which in $\Fil^1$ is exactly the isomorphism $\iota^*V = V \otimes_B \Gr^0 B \xrightarrow{\cong} \Fil^1\overline{M}$.
This proves the claim for the case $m=0$.

Lastly, we prove the case for general $m$: Choose a surjection from a finite free
$\pi_0(\Gr^0 B)^{\oplus r} \twoheadrightarrow H^m(\Fil^1\overline{M})$, which can be lifted to a map
$(\Gr^0 B)^{\oplus r}[-m] \to \Fil^1\overline{M}$ whose cone has tor-amplitude $[0, m-1]$.
Repeat the previous argument, we have a map in $\DF_{J\text{-comp}}(\Fil^{\bullet}B)$:
\[
\Fil^{\bullet} L \coloneqq (\cdots \xrightarrow{=} 0 = \Fil^2  \to B^{\oplus r}[-m] = \Fil^1 
\to 0 = \Fil^0 \xrightarrow{=} \cdots) \otimes_B \Fil^{\bullet}B \longrightarrow \Fil^{\bullet} M,
\]
whose cone has its reduction of the form
\[
\left(\cdots \xrightarrow{=} 0 = \Fil^n\overline{M} \to \Fil^{n-1}\overline{M} \to \cdots \to 
\mathrm{Cone}(\Gr^0 B^{\oplus r}[-m] \to \Fil^1\overline{M})
\to 0 = \Fil^0\overline{M} \xrightarrow{=} \cdots \right).
\]
Since $\Fil^{\bullet} L$ is $0$ after $0$-th filtration, filtered perfect,
and the tor-amplitude of $\Fil^1\overline{(M/L)}$ 
is now $[0, m-1]$, we are done by induction.
\end{proof}

We are finally ready to prove \Cref{criteria for filtered perfect}.

\begin{proof}[Proof of \Cref{criteria for filtered perfect}]
According to the discussion after its statement, it suffices to show that (3) implies (1).
Without loss of generality, we assume that $\Fil^{\bullet}M$ is constant after $0$-th filtration.
Then $\Fil^0M \cong \Fil^{-\infty}M$ is perfect over $B$.
Similar to the argument as in the third paragraph of the proof of 
\Cref{criteria for filtered perfect (special)}, we have a map
$(\cdots \to \Fil^1 M = 0 \to \Fil^0 M \xrightarrow{=} \Fil^{-1} M \xrightarrow{=} \cdots)
\otimes_B \Fil^{\bullet} B \to \Fil^{\bullet}M$ which is an isomorphism
in $\Fil^{\geq 0}$. Hence its cone is $0$ after $0$-th filtration and 
has filtered perfect reduction.
By \Cref{criteria for filtered perfect (special)} we see that this cone is filtered perfect, hence we win.
\end{proof}

We have the following analog for filtered finite projective modules.

\begin{proposition}
\label{criteria for filtered vb}
Assume that for every finite projective $B$-module $P$,
every direct summand $\overline{P_1} \subset \overline{P} \coloneqq P \otimes_B \Gr^0 B$
can be lifted to a direct summand $P_1 \subset P$.
Then for an object $\Fil^{\bullet} M \in \DF_{J\text{-comp}}(\Fil^{\bullet} B)$,
the following four conditions are equivalent:
\begin{enumerate}
\item It is a finite direct sum of $N_i \otimes_B \Fil^{\bullet} B \langle n_i \rangle$
where each $N_i$ is a finite projective $B$-module.
\item It is filtered finite projective;
\item It is filtered $J$-completely finite projective;
\item It is eventually constant with finite projective $\Fil^{-\infty} M$
and the filtered base change $\Fil^{\bullet} M \otimes_{\Fil^{\bullet}B} \Gr^0 B$ is filtered finite projective;
\item It is eventually constant with $J$-completely finite projective $\Fil^{-\infty} M$
and the filtered base change $\Fil^{\bullet} M \otimes_{\Fil^{\bullet}B} \Gr^0 B$ is filtered $J$-completely 
finite projective.
\end{enumerate}
\end{proposition}

\begin{proof}
Similar to the proof of \Cref{criteria for filtered perfect}, it is easy to see that
(1) implies (2), (2) implies (3), and (3) implies (5). Moreover, for any connective derived $J$-complete
$\mathbb{E}_{\infty}$-$A$-algebra $R$, an object
$\Fil^{\bullet} M \in \DF_{J\text{-comp}}(R)$ is filtered (resp.~$J$-completely)
filtered finite projective if and only if the filtration is complete and exhaustive
with finitely many nonzero graded pieces and each graded piece is finite projective
(resp.~$J$-completely finite projective).
Therefore we see that (4) and (5) are again equivalent, by applying \Cref{vbperf vs cplt vbperf}
to $R = B$ and $R = \Gr^0 B$. Below let us show that (4) implies (1).

Without loss of generality, let us assume that $\Fil^{\bullet} M$ is constant after $0$-th filtration.
By \Cref{general lemma on eventual constancy}.(2), we know that 
$\Fil^{\bullet} M \otimes_{\Fil^{\bullet} B} \Gr^0 B \eqqcolon \Fil^{\bullet} \overline{M}$
is also constant after $0$-th filtration, therefore it looks like
\[
\left(\cdots \xrightarrow{=} 0 = \Fil^n\overline{M} \to \Fil^{n-1}\overline{M} \to \cdots \to \Fil^1\overline{M}
\to \Fil^0\overline{M} \xrightarrow{=} \cdots \right)
\]
where each $\Fil^i \overline{M} \to \Fil^{i-1} \overline{M}$ is a direct summand of
finite projective $\Gr^0 B$-modules.
Let us perform an induction on $n$, the base case being $n = 1$.
When $n = 1$, we consider the cone $C$ of $\Fil^0 M \otimes_B \Fil^{\bullet} B \to \Fil^{\bullet} M$
induced by the natural map $\Fil^0 M \xrightarrow{=} \Fil^0 M$.
The cone is $0$ after $0$-th filtration, and its filtered base change 
\[
\Fil^{\bullet}\overline{C} \cong \mathrm{Cone}\Bigl((\cdots \xrightarrow{=} 0 = \Fil^1 \to \Fil^{0}= \Fil^{0}M \otimes_B \Gr^0 B\xrightarrow{=} \cdots) \to \Fil^{\bullet} \overline{M}\Big)
\]
is $0$ by \Cref{general lemma on eventual constancy}.(3).
Therefore we see that $C = 0$, by \Cref{general lemma on eventual constancy}.(5).
Since $\Fil^{-\infty} M=\Fil^0 M$ is a finite projective $B$-module, we get that
$\Fil^0 M \otimes_B \Fil^{\bullet} B = \Fil^{\bullet} M$ is also filtered finite projective
over $\Fil^{\bullet} B$.

In general, choose a splitting $\Fil^0\overline{M} \cong \Fil^1\overline{M} \oplus \Gr^0 \overline{M}$
and lift it to $\Fil^0 M \cong N_1 \oplus N_2$.
Now we consider the cone $C$ of $N_2 \otimes_B \Fil^{\bullet} B \to \Fil^{\bullet} M$
induced by the decomposition in $\Fil^0$. It is constant after $0$-th step and this constant is
$N_1$ which is a finite projective $B$-module.
Moreover, by construction, its filtered base change $\Fil^{\bullet} \overline{C}$
has constant filtration after the $1$-st filtration and the induced map
$\Fil^{\geq 1} \overline{M} \to \Fil^{\geq 1} \overline{C}$ is an isomorphism.
Therefore, by induction, we see that $C$ is of the form described in (1).
Lastly, we only need to know that a triangle in $\DF_{J\text{-comp}}(\Fil^{\bullet} B)$
of the form
\[
\Fil^{\bullet} M' \to \Fil^{\bullet} M
\to V \otimes_B \Fil^{\bullet} B \langle n \rangle
\]
must split if the first term is connective and $V$
is finite projective over $B$: Indeed, the first condition implies
$\pi_0(\Fil^{-n}M) \twoheadrightarrow \pi_0(V)$ is surjective,
hence there is a splitting $s \colon V \to \Fil^{-n}M$ of the map of $B$-complexes $\Fil^{-n}M \to V$,
therefore we get the desired filtered splitting $s \otimes_B \Fil^{\bullet} B$.
\end{proof}

\subsection{Descent statements}

In this subsection, we establish descent statements in the filtered
and derived $J$-complete setting.

\begin{digression}
Recall that a map of connective $\mathbb{E}_{1}$-ring $R \to S$ is said to be 
(faithfully) flat if $\pi_0(R) \to \pi_0(S)$ is (faithfully) flat
and the induced maps $\pi_i(R) \otimes_{\pi_0(R)} \pi_0(S) \to \pi_i(S)$
are isomorphisms for all $i$, see \cite[Definition 7.2.2.10]{Lur17}.
In \cite[Theorem 7.2.2.15]{Lur17} one can find many equivalent conditions characterizing
flatness.
\end{digression}

\begin{definition}
We call a map of connective derived $J$-complete $\mathbb{E}_{\infty}$-$A$-algebra
$R \to S$ to be \emph{$J$-completely (faithfully) flat} if its base change
$A/J \otimes_A R \to A/J \otimes_A S$ is so.
\end{definition}

By \cite[Proposition 2.7.3.2]{Lur18}, for any connective $R'$ with a factorization $R \to R' \to A/J \otimes_A R$ such that $\pi_0(R') \to \pi_0(R)/J$
has nilpotent kernel, the $J$-completely (faithful) flatness is equivalent to having the base change
$R' \to R' \otimes_R S$ being (faithfully) flat.
In particular, the above definition only depends on
$V(J) \subset \mathrm{Spec}(A)$.

\begin{proposition}
\label{completely filtered fpqc descent}
Let $\Fil^{\bullet}B \to \Fil^{\bullet}B^{(0)}$ be a map of connective
derived $J$-complete (decreasing) filtered $\mathbb{E}_{\infty}$-$A$-algebras
indexed by $\mathbb{N}$, and form the $i$-th Cech nerve
$\Fil^{\bullet}B^{(i)}$ for $[i] \in \Delta$. 
Assume that the map of Rees's construction
$(\bigoplus_{m \in \mathbb{Z}} \Fil^{m}B) \to (\bigoplus_{m \in \mathbb{Z}} \Fil^{m}B^{(0)})$
is $J$-completely faithfully flat.
Then the following natural functors induced by base change:
\begin{enumerate}
\item $\DF_{J\text{-comp}}(\Fil^{\bullet} B) \to \lim_{[i] \in \Delta} \DF_{J\text{-comp}}(\Fil^{\bullet} B^{(i)})$;
\item $\DF_{J\text{-comp}}^{\perf}(\Fil^{\bullet} B) \to \lim_{[i] \in \Delta} 
\DF_{J\text{-comp}}^{\perf}(\Fil^{\bullet} B^{(i)})$; and
\item $\DF_{J\text{-comp}}^{\vect}(\Fil^{\bullet} B) \to \lim_{[i] \in \Delta} 
\DF_{J\text{-comp}}^{\vect}(\Fil^{\bullet} B^{(i)})$
\end{enumerate}
are all equivalences.
Here $\vect$ and $\perf$ denote the full subcategory of filtered $J$-completely finite projective
and filtered $J$-completely perfect objects respectively.
\end{proposition}

\begin{proof}
Let us show (1) first. It is equivalent to showing:
\begin{enumerate}
\item For any $\Fil^{\bullet} M \in \DF_{J\text{-comp}}(\Fil^{\bullet} B)$, the natural arrow
$\Fil^{\bullet} M \to \lim_{[i] \in \Delta} \Fil^{\bullet} B^{(i)} \otimes_{\Fil^{\bullet} B} \Fil^{\bullet} M$
is an equivalence; and
\item For any $\Fil^{\bullet} M^{(i)} \in \lim_{[i] \in \Delta} \DF_{J\text{-comp}}(\Fil^{\bullet} B^{(i)})$,
the natural arrow $\Fil^{\bullet} B^{(m)} \otimes_{\Fil^{\bullet} B} (\lim_{[i] \in \Delta} \Fil^{\bullet} M^{(i)})
\to \Fil^{\bullet} M^{(m)}$ is an equivalence for all $[m] \in \Delta$.
\end{enumerate}
By derived Nakayama, it suffices to show the above are equivalences after applying
$\mathrm{Kos}(A; f_i) \otimes_A -$, where $\{f_i\}$ is a chosen finite set of generators of $J$.
Since $\mathrm{Kos}(A; f_i)$ is a perfect complex over $A$, the functor commutes with limit.
Therefore we are reduced to the case where $J = 0$.
Since Rees's construction commutes with cosimplicial limit, and \emph{when $J = 0$} it
is symmetric monoidal: Namely, it sends the filtered tensor product
to the usual tensor product, we are finally reduced to the usual fpqc descent 
\cite[Corollary D.6.3.3]{Lur18}.

Next we show (2) and (3): Given (1), it suffices to show that
$\Fil^{\bullet} B^{(0)} \otimes_{\Fil^{\bullet} B} \Fil^{\bullet} M$ being filtered $J$-completely
finite projective (resp.~perfect) implies that the original $\Fil^{\bullet}M$
is filtered $J$-completely finite projective (resp.~perfect).
Using \Cref{VB in terms of Rees} and \Cref{filtered perfect criterion},
we are reduced to the usual equivalence of being finite projective (resp. perfect)
before and after the faithfully flat base change \cite[Proposition 2.8.4.2]{BL22b}.
\end{proof}

We are also interested in descent even when the filtered rings involved are \emph{not}
connective. For that purpose.
we shall use the notion of ``descendability'' by Mathew, see \cite[Definition 3.18]{Mat16}.
Below let us collect some useful facts about descendable maps between $\mathbb{E}_{\infty}$-rings.

\begin{digression}[{\cite[\S 3]{Mat16}, see also \cite[\S 11.2]{BS17}}]
\label{descendable discussion}
Descendability is preserved under base change, see \cite[Corollary 3.21]{Mat16}.
If $R \to S$ is a descendable map between $\mathbb{E}_{\infty}$-rings,
and let $S^{(\bullet)}$ be the Cech nerve,
then \cite[Proposition 3.22]{Mat16} gives the descent statement:
the natural functor induced by base change
$\mathcal{D}(R) \to \lim_{[i] \in \Delta} \mathcal{D}(S^{(\bullet)})$
is an equivalence.
By \cite[Proposition 3.28]{Mat16}, we see that an $R$-complex $M$
is perfect if and only if $S \otimes_R M \in \mathcal{D}(S)$
is perfect.
\end{digression}

In fact, Mathew works with general ``\emph{stable homotopy theory}'', examples of which include
the $\DF_{J\text{-comp}}(A)$ that we have been working on throughout this section.
Below we explicate the consequences that one can draw out of Mathew's theory of descendability,
when specialized to $\DF_{J\text{-comp}}(A)$.

\begin{proposition}
\label{completely filtered descendable descent}
Let $\Fil^{\bullet}B \to \Fil^{\bullet}B^{(0)}$ be a \emph{descendable} map of
derived $J$-complete (decreasing) filtered $\mathbb{E}_{\infty}$-$A$-algebras
indexed by $\mathbb{Z}$, and form the $i$-th Cech nerve
$\Fil^{\bullet}B^{(i)}$ for $[i] \in \Delta$. 
Then the following natural functors induced by base change:
\begin{enumerate}
\item $\DF_{J\text{-comp}}(\Fil^{\bullet} B) \to \lim_{[i] \in \Delta} \DF_{J\text{-comp}}(\Fil^{\bullet} B^{(i)})$; and
\item $\DF_{J\text{-comp}}^{\perf}(\Fil^{\bullet} B) \to \lim_{[i] \in \Delta} 
\DF_{J\text{-comp}}^{\perf}(\Fil^{\bullet} B^{(i)})$;
\end{enumerate}
are both equivalences.
\end{proposition}

\begin{proof}
The first statement follows from \cite[Proposition 3.22]{Mat16}.
For (2), we need to know that $\Fil^{\bullet}M$ is filtered $J$-completely perfect as soon as
$\Fil^{\bullet} B^{(0)} \otimes_{\Fil^{\bullet} B} \Fil^{\bullet} M$ is filtered $J$-completely
perfect. By definition, we need to perform $A/J \otimes_A -$ and ask filtered perfectness,
hence we are reduced to the case of $J = 0$. In this case, the Rees's construction is symmetric monoidal.
By \cite[Corollary 3.21]{Mat16}, we see that the map of Rees's construction
$(\bigoplus_{m \in \mathbb{Z}}\Fil^m B) \to (\bigoplus_{m \in \mathbb{Z}}\Fil^m B^{(0)})$
is descendable in $\mathcal{D}(A)$. The perfectness of $\Fil^{\bullet}M$ now follows from
the combination of \Cref{filtered perfect criterion} and \cite[Proposition 3.28]{Mat16}.
\end{proof}

\section{Absolute prismatic $F$-crystals and prismatic $F$-gauges}
In this section, we recall the notion of absolute prismatic $F$-crystals and absolute prismatic $F$-gauges, and consider their relation.

\subsection{The coherent prismatic ($F$-)crystals}
\label{The coherent prismatic ($F$-)crystals}
Let $X$ be a quasi-syntomic $p$-adic formal scheme over $\mathcal{O}_K$.
Recall from \cite{BS19} and \cite{BS21} that we can attach to $X$ the following ringed sites:
\begin{definition}\label{def prismatic site}
 The \emph{(absolute) prismatic site of $X$} is defined as the opposite category $X_\Prism$ of bounded prisms $(A,I)$ over $X$ with the $(p,I)$-complete flat topology, and is equipped with a sheaf of rings and a sheaf of ideals
\[
\mathcal{O}_\Prism \colonequals (A,I) \longmapsto A;~\mathcal{I}_\Prism \colonequals (A,I) \longmapsto I.
\]
The prismatic structure sheaf $\mathcal{O}_\Prism$ admits a natural Frobenius action which lifts the usual
Frobenius on its reduction mod $p$.
\end{definition}
\begin{definition}\label{def qrsp site}
	The	\emph{quasiregular semiperfectoid site of $X$} is the opposite category $X_\qrsp$ of quasiregular semiperfectoid algebras 
 which are quasi-syntomic over $X$, equipped with the quasi-syntomic topology.
 It comes with a sheaf of prisms:
	\[
	(\Prism_\bullet, I_\bullet) \colonequals S \longmapsto (\Prism_S, I_{\Prism_S})
	\]
	sending $S\in X_\qrsp$ to the initial prism of $\Spf(S)_\Prism$ as in \cite[Prop.~7.2]{BS19}.
 The fact that they are quasi-syntomic sheaves follows from the theory of Nygaard filtration
 and quasi-syntomic descent for conjugate filtrations, see \cite[\S 12]{BS19}.
 Notice that for every $S$, the ideal $I_{\Prism_S}$ is non-canonically principal.
	The Frobenius endomorphism of prisms induces a natural endomorphism $\varphi$ on $\Prism_\bullet$.
\end{definition}
The above allows us to define various notions of coefficients associated to $X$.
Below we let $*\in \{\vect, \perf,\emptyset\}$ to denote the corresponding category with coefficients in either vector bundles, perfect complexes or general complexes.
\begin{definition}\label{def Fcrystal}
	Denote by $\fC^*(X)$ the category of \emph{prismatic} ($F$-)\emph{crystals} 
 (in either vector bundles or perfect complexes or general complexes) over $X$,
 in the sense of \cite[Definition 4.1]{BS21}.
\end{definition}

Concretely, an $F$-crystal consists of a vector bundle/perfect complex/general complex
$\mathcal{E}$ over the prismatic site $X_\Prism$ together with a Frobenius-isogeny, 
namely an $\mathcal{O}_\Prism$-linear isomorphism
\[
\varphi_\mathcal{E}: (\varphi_{\mathcal{O}_\Prism}^*\mathcal{E})[1/\mathcal{I}]  \simeq \mathcal{E}[1/\mathcal{I}].
\]
By \cite[Proposition 2.13 and Proposition 2.14]{BS21}, 
one also has the following descriptions when $*\in \{\vect, \perf\}$
\[
\fC^*(X) \simeq \lim_{S\in X_\qrsp} \fC^*(\Spf(S)).
\]

Among objects in $X_\Prism$, the following type of prisms plays a very useful role.
\begin{definition}
\label{BK prism}
	Assume $X$ is smooth over $\mathcal{O}_K$.
	A \emph{Breuil--Kisin} prism $(A,I)$ in $X_{\Prism}$ is a prism such that $\Spf(\overline{A})\to X$ is an open immersion
 (c.f.~\cite[Example 3.4]{DLMS} and \cite[proof of Theorem 5.10]{GR22}).
\end{definition}

Below we show that such $A$ is always with a faithfully flat Frobenius endomorphism $\phi_A:A\to A$.
We have the following simple description of Breuil--Kisin prisms.

\begin{lemma}
\label{concrete description of BK prisms}
Let $(A,I)$ be a Breuil--Kisin prism over $X$ as above,
and let $\widetilde{R}$ be a smooth lift of $\overline{A}/\pi$ over $W$.
\begin{enumerate}
    \item If $e > 1$, 
    then there is an isomorphism of pairs $(A, I) \simeq (\widetilde{R}[\![u]\!], E(u))$,
    where $E(u)$ is the Eisenstein polynomial of a uniformizer $\pi$.
    \item If $e = 1$, assuming $\Spf(A)$ is oriented, then we have an isomorphism of pairs
    $(A, I) \simeq (\widetilde{R}[\![u]\!], u - p)$.
\end{enumerate}
\end{lemma}

\begin{proof}
By deformation theory, we can lift the isomorphism $\widetilde{R}/p \simeq \overline{A}/\pi$
to an isomorphism
$\widetilde{R} \otimes_{W} \mathcal{O}_K \xrightarrow{\simeq} \overline{A}$.
Again by deformation theory, we can lift the morphism 
$\widetilde{R} \to \overline{A}$
further to $\widetilde{R} \to A$.

Now we can start the proof. Assume $e > 1$, pick a uniformizer 
$\pi \in \mathcal{O}_K \subset \overline{A}$,
lift it to an element $\widetilde{\pi} \in A$. Noticing that $\widetilde{\pi}$
is topologically nilpotent in $A$, we get a map
$\widetilde{R}[\![u]\!] \to A$, sending $u$ to $\widetilde{\pi}$.
Now we observe that the Eisenstein polynomial $E(u)$ is sent to zero in $\overline{A}$,
hence $E(u)$ necessarily lands in $I \subset A$, and the source is $E(u)$-adically
complete whereas the target is $I$-adically complete.
Therefore to finish our proof, we just need to check that $E(\widetilde{\pi})$ generates $I$
in $A$. Writing $E(\widetilde{\pi}) = \widetilde{\pi}^n + p \cdot f$,
where $f$ is a unit, and let us compute 
\[
p \delta(E(\widetilde{\pi})) = 
\big(\widetilde{\pi}^p + p \delta(\widetilde{\pi})\big)^n + p \cdot \varphi(f)
- E(\widetilde{\pi})^p.
\]
Using the fact that $n > 1$, one can check that the right hand side is
$p \cdot \varphi(f) + p \cdot g$ where $g \in (p, \widetilde{\pi})$.
Since $A$ is $p$-torsion free, this implies that
$\delta(E(\widetilde{\pi}))$ is a unit. Therefore \cite[Lemma 2.24]{BS19}
gives $E(\widetilde{\pi}) = I$.

If instead $e = 1$ and $(A, I)$ is oriented.
Let us choose a generator $d$ of $I$. Then we get an isomorphism
$\widetilde{R}[\![\widetilde{u}]\!] \simeq A$ by sending $\widetilde{u}$ to $d$.
Now we just let $u = \widetilde{u} + p$.
\end{proof}

\begin{corollary}
\label{BK is regular with flat Frob}
Let $(A,I)$ be a Breuil--Kisin prism over $X$ as above,
then $A$ is a $p$-torsionfree regular ring and the Frobenius endomorphism $\phi_A:A\to A$ is 
quasi-syntomic, finite, and faithfully flat.
\end{corollary}

\begin{proof}
The fact that $A$ is $p$-torsionfree and regular follows immediately from the concrete description of $A$
in \Cref{concrete description of BK prisms}.
Also from the description, we see that the Frobenius on $A/p$ is finite and lci,
therefore its lift is also finite and quasi-syntomic thanks to $p$-completeness of $A$.
Using regularity, finiteness, and
miracle flatness \cite[\href{https://stacks.math.columbia.edu/tag/00R4}{Tag 00R4}]{stacks-project},
we get that $\phi_A$ is flat. Faithfulness follows from $p$-completeness:
the induced map on spectrum is a homeomorphism on the locus $V(p)$ which contains all
closed points by $p$-completeness, noticing that flat map is open we see that
the induced map on spectrum has to be surjective.
\end{proof}

If $X=\Spf(R)$ is further assumed to be affine, then we can find such an $A$ so that $(A,I)$
covers the final object in $Shv(X_\Prism)$.
It is because with mild assumptions, we can construct absolute product of any prism with Breuil--Kisin prisms
in $X_{\Prism}$, and they enjoy good complete flatness properties.

\begin{proposition}
\label{BK covering property}
Assume $X$ is separated and smooth over $\Spf(\mathcal{O}_K)$, let $(A,I)$ be a Breuil--Kisin
prism in $X_{\Prism}$ and let $(A', I')$ be a prism in $X_{\Prism}$.
Then their product in $Shv(X_{\Prism})$ exists and is given by a prism $(C, K)$.
It has the following properties:
\begin{enumerate}
\item The map $(A', I') \to (C, K = I' C)$ is $(p, I')$-completely flat;
\item If $\mathrm{Im}(\Spf(\overline{A'}) \to X) \subset \Spf(\overline{A})$, 
then the map above is $(p, I')$-completely faithfully flat;
\item If the map $\Spf(\overline{A'}) \to X$ is $p$-completely flat,
then the map $(A, I) \to (C, K = I C)$ is $(p,I)$-completely flat.
\item If in (3) we assume furthermore that $\Spf(\overline{A}) \subset \mathrm{Im}(\Spf(\overline{A'}) \to X)$, 
then the map above is $(p, I)$-completely faithfully flat.
\end{enumerate}
Consequently, if $X=\Spf(R)$ is affine, then we can find a Breuil--Kisin prism $(A, I)$ which covers
the final object $* \in Shv(X_{\Prism})$.
\end{proposition}

This Proposition is certainly well-known to experts, for the convenience of the reader
we spell out the proof in detail.
Here we temporarily abuse the notation of $K$.

\begin{proof}
Let us construct the prism $(C,K)$ first: it is supposed to be the absolute pushout of $(A, I)$
and $(A', I')$ in the category of prisms in $X_{\Prism}$.
Unwinding definition, we see that $(C, K)$ is the initial object of prisms $(D, L)$
receiving map from both $(A, I)$ and $(A', I')$ with the property that the induced map
$\overline{A} \widehat{\otimes}_{W(k)} \overline{A'} \to D/L$ factors through 
$\overline{A} \widehat{\otimes}_{\mathcal{O}_X} \overline{A'}$.
Here we make two remarks:
\begin{itemize}
    \item For any prism $(D, L) \in X_{\Prism}$, using the fact that $D/L^n$
    are $p$-complete, we see that the map $W(k) \to \mathcal{O}_X \to D/L$
    lifts uniquely to a map $W(k) \to D$ of $\delta$-rings.
    This explains why the first tensor product above is over $W(k)$.
    \item Since $X$ was assumed to be separated, the (completed) tensor product
    $\overline{A} \widehat{\otimes}_{\mathcal{O}_X} \overline{A'}$ is given by a (derived) $p$-complete ring
    corresponding to the formal scheme $\Spf(\overline{A}) \times_X \Spf(\overline{A'})$.
\end{itemize}
Back to the construction of $(C, K)$, from the above description we see that it is necessarily given
by the delta envelope of $(B, J)$, whose meaning we shall explain below: 
Here $B \coloneqq A \widehat{\otimes^L}_{W(k)} A'$
denotes the derived $(p, 1 \otimes I', I \otimes 1)$-completed tensor product,
and one checks that it is concentrated in degree $0$ and is a $(p, I')$-completely flat $\delta$-$A'$-algebra.
Indeed if we derived mod $B$ by $1 \otimes I'$, it becomes the derived $(p, I)$-completed tensor product
$A \widehat{\otimes^L}_{W(k)} \overline{A'}$ which is seen to be a $p$-completely flat
$\overline{A'}$-algebra due to \Cref{concrete description of BK prisms}.
Let $J$ denote the kernel ideal of the map $B \to \overline{A} \widehat{\otimes}_{\mathcal{O}_X} \overline{A'}$.
Below we shall check that our ideal $J$ meets the regularity condition in \cite[Proposition 3.13]{BS19}.
Granting this the cited proposition constructs the prismatic envelope and moreover checks that
it satisfies the property (1). 
Now let us check the regularity condition.
Due to the Zariski locality of the regularity condition,
we may assume that $A$ is given by the concrete description in \Cref{concrete description of BK prisms}.
Now we see that the ideal $J/(1 \otimes I')$ is Zariski locally on $\Spf(B/(1 \otimes I'))$
given by $u - 1 \otimes \pi$ and (the lift of) a regular sequence corresponding to the kernel of 
$\widetilde{R} \widehat{\otimes}_{W(k)} \overline{A'} \twoheadrightarrow 
\overline{A} \widehat{\otimes}_{\mathcal{O}_X} \overline{A'}$, where both completed tensor product are $p$-completed tensor product.
The said kernel is Zariski locally a regular sequence because $\widetilde{R}$ is $p$-completely smooth
$W(k)$-algebra and $\overline{A}$ defines an open of $X$, so the surjection is
from a $p$-completely smooth $\overline{A'}$-algebra
to a $p$-completely Zariski open $\overline{A'}$-algebra, which is lci mod $p$.

By the above paragraph, we have constructed $(C, K)$ and it satisfies the properties listed
in \cite[Proposition 3.13]{BS19}. In particular, it satisfies (1) of our proposition.
Let us check property (2): we need to utilize the perfection 
$(A_{\infty}, I_{\infty}) \coloneqq (A, I)_{\mathrm{perf}}$
of $(A, I)$ introduced in \cite[Lemma 3.9]{BS19}.
Now \Cref{BK is regular with flat Frob}
says that the Frobenius $\phi_A$ is quasi-syntomic and faithfully flat, by construction spelled out
in loc.~cit.~we see that $(A, I) \to (A_{\infty}, I_{\infty})$ is also quasi-syntomic and $(p, I)$-completely
faithfully flat. Let us stare at the following diagram:
\[
\xymatrix{
(A', I') \ar[r] & (C, K) \ar[r] & (E, M) \\
& (A, I) \ar[u] \ar[r] & (A_{\infty}, I_{\infty}) \ar[u],
}
\]
where the square is a $(p, I)$-completely pushing out defining the prism $(E, M)$.
By the above discussion, we know the rightward arrows in this square are $(p, I)$-completely
faithfully flat. Therefore we reduce our problem to checking the composite $(A', I') \to (E, M)$
is $(p, I')$-completely faithfully flat.
To that end, let us study $(E, M)$ from a slightly different perspective:
By how it is defined, we see that $(E, M)$ is the absolute pushout of $(A_{\infty}, I_{\infty})$
and $(A', I')$ in the category of prisms in $X_{\Prism}$.
Since perfect prisms are initial prisms of their reductions (\cite[Theorem 3.10]{BS19}),
we see that $(E, M)$ is the initial object in 
$\big(S/A'\big)_{\Prism}$,
where $S \coloneqq \overline{A'} \widehat{\otimes}_{\overline{A}} \overline{A_{\infty}}$.
This is where we use the assumption that $\mathrm{Im}(\Spf(\overline{A'}) \to X) \subset \Spf(\overline{A})$.
Now we check that $S$ is a large quasi-syntomic $p$-completely faithfully flat $\overline{A'}$-algebra 
(see \cite[Definition 15.1]{BS19} for the meaning of being ``large''):
It is the $p$-complete base change of a quasi-syntomic
$p$-completely faithfully flat
morphism; as for largeness, just notice that
any perfectoid algebra is large (being a quotient of the Witt vector of a perfect ring),
hence so are their base changes.
By \cite[Theorem 15.2]{BS19}, we learn that $E$ is given by the derived prismatic cohomology
$\Prism_{S/A'}$ explained in \cite[Construction 7.6]{BS19}.
Recall that we wanted to show $\overline{A'} \to \bar{E}$ is $p$-completely faithfully flat,
the map factors through the $p$-completely faithfully flat map $\overline{A'} \to S$,
hence it suffices to show $S \to \bar{E}$ is $p$-completely faithfully flat.
To that end, one just notices that this map is the $0$-th conjugate filtration on the 
derived Hodge--Tate cohomology $\bar{E} = \overline{\Prism}_{S/A'}$,
and the higher graded pieces of the conjugate filtration are $p$-completely flat $S$-modules:
indeed they are given by $\Gamma^{\bullet}_S(\mathbb{L}_{S/A'}[-1])^{\wedge}$ and
the shifted cotangent complex $\mathbb{L}_{S/A'}[-1]$ is a $p$-completely flat $S$-module
(c.f.~the discussion right after \cite[Definition 15.1]{BS19}).

The proof of (3) is fairly easy: it is equivalent to asking the map
$\overline{A} \to \overline{C}$ to be $p$-completely flat.
We stare at the following commutative diagram:
\[
\xymatrix{
\Spf(\overline{C}) \ar[r] \ar[d] & \Spf(\overline{A}) \ar[d] \\
\Spf(\overline{A'}) \ar[r] & X.
}
\]
Our assumption, together with (1) above,
implies that the map $\Spf(\overline{C}) \to X$ is $p$-completely flat.
Since the map $\Spf(\overline{A}) \to X$ is an open immersion,
we get what we want.

As for (4): since $X$ is separated,
the preimage of $\Spf(\overline{A}) \subset X$ defines 
an affine open in $\Spf(\overline{A'})$ which is also an affine open $U \subset \Spf(A')$,
we may replace $A'$ by $\mathcal{O}_{\Spf(A')}(U)$ and replace $X$ by $\Spf(\overline{A})$
without changing $(C, K)$. Now we are reduced to the case where 
$X = \Spf(\overline{A}) = \mathrm{Im}(\Spf(A'))$.
Looking again at the above diagram, we see that $\Spf(\overline{C}) \to X = \Spf(\overline{A})$
is $p$-completely faithfully flat due to (2) and (3).
This finishes the proof of properties (1)-(4) listed in this Proposition.

For the last sentence, simply notice that when $X$ is affine,
one can find a Breuil--Kisin prism $(A, I)$ with $\Spf(A/I) \xrightarrow{\simeq} X$,
see for instance \cite[Example 3.4]{DLMS} or \cite[proof of Theorem 5.10]{GR22}.
The statement now follows from property (2) of product with any other
prism in $X_{\Prism}$.
\end{proof}

This allows us to put a $t$-structure on prismatic ($F$-)crystals
in perfect complexes on $X$ smooth over $\mathcal{O}_K$, as follows.
\begin{construction}
\label{construction of t-structure}
Let us assume first that $X$ is separated and smooth over $\mathcal{O}_K$.
By \Cref{BK covering property}, we see that one can find a family of Breuil--Kisin
prisms $(A_{\lambda}, I_{\lambda})$ indexed by ${\lambda \in \Lambda}$ 
which jointly covers $X_{\Prism}$.
Now we may form the Cech nerve of the cover 
$\bigsqcup_{\lambda \in \Lambda} (A_{\lambda}, I_{\lambda}) \twoheadrightarrow *$,
whose $n$-th spot is given by 
$(A_{\lambda_1}, I_{\lambda_1}) \times \cdots \times (A_{\lambda_{n+1}}, I_{\lambda_{n+1}})$
indexed by $\Lambda^{n+1}$.
Below we will simply denote this Cech nerve by $\Spf(A^{[\bullet]})$
with the corresponding family of rings $A^{[\bullet]}$,
so each $A^{[n]}$ really stands for the above family of rings indexed by $\Lambda^{n+1}$.
By $(p, I)$-completely faithfully flat descent (c.f.~\cite[Theorem 5.8]{Mat22} and \cite[Theorem 2.2]{BS21}),
the $\infty$-category $\fC^{\perf}(X)$ is given by the cosimplicial limit of the 
derived $\infty$-categories of ($F$-)perfect complexes on $\Spf(A^{[\bullet]})$,
we define a $t$-structure by requiring an ($F$-)crystals in perfect complexes
lives in $\leq 0$ part (resp.~$\geq 0$ part) if the underlying complexes
on $\Spf(A^{[\bullet]})$ is concentrated in cohomological degrees $\leq 0$
(resp.~$\geq 0$).
\end{construction}

By \Cref{BK covering property} (2), the maps from $A^{[0]}$
to any of $A^{[\bullet]}$
are all $(p, I)$-completely flat, so are all the maps composed with the Frobenius
on $A^{[\bullet]}$.
Hence all these maps are already flat as $A^{[0]}$ consists
of $(p, I)$-complete Noetherian rings.
Therefore it suffices to check
the cohomological degrees for the underlying complexes on $\Spf(A_{\lambda})$'s.
We call this the \emph{standard $t$-structure}, below we shall check that
it is indeed a $t$-structure and is independent of the choice of the 
family $(A_{\lambda}, I_{\lambda})$.
We refer readers to \cite[Chapter 1]{BBDG82} and \cite[Section 1.2.1]{Lur17} for a general discussion on $t$-structures.

\begin{lemma}
\Cref{construction of t-structure} defines a $t$-structure on 
$\fC^{\perf}(X)$.
When $X$ is quasi-compact, this $t$-structure is exhaustive and separated.
\end{lemma}

\begin{proof}
Given $\mathcal{F}$ (resp.~$\mathcal{G}$) in $\fC^{\perf}(X)$
with its values on $\Spf(A_{\lambda})$'s concentrated in
degrees $\leq 0$ (resp.$\geq 1$),
we see immediately that $\mathrm{Map}_{\fC^{\perf}(X)}(\mathcal{F}, \mathcal{G})$
is the cosimplicial limit of contractible spaces: This is simply because
their values on each $\Spf(A^{\bullet})$'s have contractible mapping spaces
for degree reason.
Therefore we get that the mapping space above is itself a contractible space.
To see the existence of $\tau^{\leq 0}$ truncation, one just notices that term-wise
$\tau^{\leq 0}$ truncation still satisfies the crystal condition
thanks to the discussion prior to this Lemma.
Moreover, the resulting object is a perfect complex: By crystal condition
it suffices to check this on the family of rings 
$A^{[0]} \coloneqq \{A_{\lambda}\}_{\lambda \in \Lambda}$,
but this is automatic as Breuil--Kisin prisms 
are regular by \Cref{BK is regular with flat Frob}.
Due to flatness of $A^{[0]} \to A^{[i]} \xrightarrow{\varphi_{A^{[i]}}} A^{[i]}$,
we see that the $\leq 0$ truncation
of $\varphi_{A^{[i]}}^* \mathcal{F}(A^{[i]})$ is the same as the Frobenius twist
of the $\leq 0$ truncation of $\mathcal{F}(A^{[i]})$.
Hence the linearized Frobenius extends canonically to the truncations, proving the existence
of truncations for objects in $\FC^{\perf}(X)$.

When $X$ is quasi-compact, one can find a disjoint union of finitely many
Breuil--Kisin prisms to cover $X_{\Prism}$, and the exhaustiveness and separatedness
follows from the same properties of the standard $t$-structure on
$\Perf(R)$ for any regular ring $R$.
\end{proof}

Granting the claim that in the separated case
the standard $t$-structure is independent of the choice of
covering families of Breuil--Kisin prisms,
we make the following definition.

\begin{definition}
\label{def coherent crys}
Let $X$ be smooth over $\mathcal{O}_K$,
we define the \emph{standard $t$-structure} on $\fC^{\perf}(X)$ by:
An object lives in $\leq 0$ part (resp.~$\geq 0$ part) if its restriction
to $\fC^{\perf}(U)$ is so for any separated open $U \subset X$.
We say an object in $\fC^{\perf}(X)$ is \emph{coherent}
if it belongs to the heart of the standard $t$-structure,
we denote this heart by $\fC^{\coh}(X)$.
By taking its associated homotopy category, we may regard $\fC^{\coh}(X)$
as an abelian category.
\end{definition}


Applying the usual Zariski descent to Breuil--Kisin prisms,
one sees that it suffices to check the above condition for a family
of separated opens $U \subset X$ which jointly covers $X$.
What is left to show is the claim about independence of the
choice of covering family of Breuil--Kisin prisms.
In fact, we prove the following stronger statement.

\begin{proposition}
\label{Evaluating at flat prism is t-exact}
Let $X$ be separated and smooth over $\mathcal{O}_K$.
Fix a covering family of Breuil--Kisin prisms 
$\bigsqcup_{\lambda \in \Lambda} (A_{\lambda}, I_{\lambda}) \twoheadrightarrow *$
which gives rise to a $t$-structure as in \Cref{construction of t-structure}.
Let $(A', I') \in X_{\Prism}$
and suppose that $\Spf(\overline{A'}) \to X$ is $p$-completely flat,
then the evaluation functor 
$\mathrm{Crys}^{\perf}(X) \to \Perf(A')$
is $t$-exact with respect to the standard $t$-structures on both sides.
\end{proposition}

In particular, being in the heart of the $t$-structure with respect to
one covering family of Breuil--Kisin prisms forces its value
on any other Breuil--Kisin prism to concentrate in degree $0$.
This way, we see that the $t$-structure we get is indeed independent
of the choice of covering families of Breuil--Kisin prisms.

\begin{proof}
Let $(C_{\lambda}, K_{\lambda})$ be the product of $(A_{\lambda}, I_{\lambda})$
and $(A', I')$ as in \Cref{BK covering property}.
Then property (2) in said Proposition implies that $\Spf(C_{\lambda})$
gives rise to a $(p, I')$-completely flat cover of $\Spf(A')$.
By definition, the product $(C^{\bullet}, K^{\bullet})$
of $(A^{\bullet}, I^{\bullet})$ with $(A', I')$
exists and is given by the Cech nerve of the cover
$\Spf(C^{[0]}) \coloneqq \bigsqcup_{\lambda \in \Lambda} \Spf(C_{\lambda})
\twoheadrightarrow \Spf(A')$.

Now let $\mathcal{F} \in \mathrm{Crys}^{\coh}(X)$ be in the heart of the $t$-structure
defined by $(A_{\lambda}, I_{\lambda})$.
Then its values on $C^{[\bullet]}$ defines a descent datum giving rise to
the value on $A'$ (here we suppresses the ideal for simplicity).
By \Cref{BK covering property} (3), we know that all of $C^{[\bullet]}$
are $(p, I)$-completely flat over $A^{[0]}$.
Since $A^{[0]}$ consists of derived $(p, I)$-complete rings which are Noetherian,
we see that all of $C^{[\bullet]}$ are flat over $A^{[0]}$.
The complexes in descent datum are (non-canonically) given by
$\mathcal{F}(C^{[\bullet]}) \simeq \mathcal{F}(A^{[0]}) \otimes_{A^{[0]}} C^{[\bullet]}$,
we learn that they consist of modules sitting in degree $0$.
Now \Cref{complete descent of modules is a module} below ensures that the descent
$\mathcal{F}(A')$ also lives in degree $0$, hence proving the assertion.
\end{proof}

\begin{lemma}
\label{complete descent of modules is a module}
Let $A' \to C^{[0]}$ be a $J$-completely faithfully flat map
of derived $J$-complete rings, where $J = (f_1, \cdots, f_r) \subset A'$ is a finitely generated
ideal. Denote the Cech nerve by $C^{[\bullet]}$.
Then in the equivalence $D_{J-\text{comp}}(A') \simeq 
\lim_{\Delta} D_{J-\text{comp}}(C^{[\bullet]})$ given by $J$-complete faithfully flat descent,
a descent datum of the right hand side consisting of complexes $M^{[\bullet]}$
concentrated in cohomological degree $0$ has its descent $M \in D_{J-\text{comp}}(A')$
concentrated in cohomological degree $0$ as well.
\end{lemma}

\begin{proof}
Since the descent $M$ is calculated by $\lim_{\Delta} M^{[\bullet]}$, it is concentrated
in cohomological degrees $\geq 0$.
Below we show the reverse cohomological degree estimate.

We look at the following base change formula for any natural number
$m$:
\[
M \otimes^{L}_{A'} \mathrm{Kos}(A'; f_i^m) \otimes^{L}_{\mathrm{Kos}(A'; f_i^m)}
\mathrm{Kos}(C^{[0]}; f_i^m) \cong M^{[0]} \otimes^{L}_{C^{[0]}}
\mathrm{Kos}(C^{[0]}; f_i^m).
\]
The right hand side lives in degrees $\leq 0$ by our condition that $M^{[0]}$
lives in degree $0$.
Therefore we see $M \otimes^{L}_{A'} \mathrm{Kos}(A'; f_i^m)$
also lives in degrees $\leq 0$,
as $A' \to C^{[0]}$ is $J$-completely faithfully flat.
By the same argument, the transition maps
$\pi_0(M \otimes^{L}_{A'} \mathrm{Kos}(A'; f_i^{m+1})) \to
\pi_0(M \otimes^{L}_{A'} \mathrm{Kos}(A'; f_i^m))$
are all surjective.
Using the fact that $M$ is derived $J$-complete,
we get that $M = \mathrm{R}\lim_m
\big( M \otimes^{L}_{A'} \mathrm{Kos}(A'; f_i^m) \big)$,
which lives in degrees $\leq 0$ due to the above analysis.
\end{proof}

Let us point out that we do not know if the converse to the above holds true:
Namely given a $J$-complete $A$-module $M$, base change $M$ along a $J$-completely
flat ring map $A \to B$ might not be concentrated in degree $0$ anymore.
Now we can show the aforementioned independence of the choice of Breuil--Kisin prisms.

\begin{corollary}
\label{value of t-inequality crystal on flat object}
    Let $X$ be a smooth formal scheme over $\mathcal{O}_K$.
    An object $\mathcal{F}\in \fC^\perf(X)$ lives in $\leq 0$ part (resp. $\geq 0$ part) if and only if any of the following equivalent condition holds:
    \begin{enumerate}
        \item There exists a covering family of Breuil--Kisin prisms $(A_\lambda, I_\lambda)$ in $X_\Prism$, 
        such that $\mathcal{F}(A_\lambda, I_\lambda)$ lives in $D^{\leq 0}(A_\lambda)$ (resp. $D^{\geq 0}(A_\lambda)$).
        \item For any Breuil--Kisin prism $(A,I)\in X_\Prism$, 
        we have $\mathcal{F}(A,I)\in D^{\leq 0}(A)$ (resp. $D^{\geq 0}(A)$).
        \item There exists a covering family of prisms $(A,I)\in X_\Prism$ with $\Spf(\overline{A})\to X$ being $p$-completely flat,
        such that $\mathcal{F}(A,I)\in D^{\leq 0}(A)$ (resp. $D^{\geq 0}(A)$).
        \item For any prism $(A,I)\in X_\Prism$ such that $\Spf(\overline{A})\to X$ is $p$-completely flat, 
        we have $\mathcal{F}(A,I)\in D^{\leq 0}(A)$ (resp. $D^{\geq 0}(A)$).
    \end{enumerate}
\end{corollary}

\begin{proof}
We shall prove the ``living in $\leq 0$ part'' statement, as the other part can be proved similarly.
It is easy to see that (4) implies (3) and (2), (2) implies (1), and (1) implies that $\mathcal{F}$
lives in $\leq 0$ part.
The last condition implies (4): This is exactly \Cref{Evaluating at flat prism is t-exact}.

What is left to show is that (3) also implies that $\mathcal{F}$
lives in $\leq 0$ part. Now we look at the truncation $\mathcal{G} \coloneqq \tau^{>0}(\mathcal{F})$.
Equivalently, we need to show $\mathcal{G}$ is zero. By
\Cref{Evaluating at flat prism is t-exact} and condition (3), we see that $\mathcal{G}$
has value $0$ on a covering family of prisms, hence it is zero.
\end{proof}

\begin{definition}[{$I$-torsionfree crystal}]
We say an $\mathcal{E} \in \fC^{\coh}(X)$ is $I$-torsionfree if the map of coherent crystals:
$\mathcal{I}_{\Prism} \otimes_{\mathcal{O}_{\Prism}} \mathcal{E} \to \mathcal{E}$ 
is injective, i.e.~has no kernel with respect to the standard $t$-structure.
Equivalently, this is asking $\mathcal{E}(A)$ to be $I$-torsionfree
for all (or a covering family of) Breuil--Kisin prisms $(A, I) \in X_{\Prism}$.
\end{definition}

\begin{remark}
We have the following fully faithful inclusion of categories:
\[
\fC^\vect(X) \subset \fC^\mathrm{an}(X) \subset \fC^{I\text{-tf}} \subset \fC^\coh(X),
\]
for the last inclusion, where the target is regarded as an abelian category (since it is the heart
of the standard $t$-structure), see \cite[Theorem 5.10]{GR22}.
By definition, we have a natural fully faithful inclusion of $\infty$-categories
\[
\fC^\coh(X) \subset \fC^\perf(X).
\]
\end{remark}

Recall in \cite[\S 3]{BS21} the authors introduced the notion of Laurent $F$-crystals
(see Definition 3.2 in loc.~cit.). According to the Corollary 3.7 in loc.~cit.,
Laurent $F$-crystals on a bounded $p$-adic formal scheme $X$
are functorially identified with the locally constant derived category 
of the adic generic fiber of $X$:
\[
D_{\mathrm{perf}}(X_{\Prism}, \mathcal{O}_{\Prism}[1/I_{\Prism}]^{\wedge}_p)^{\varphi = 1}
\simeq D^{(b)}_{lisse}(X_{\eta}, \mathbb{Z}_p).
\]
Notice that the right hand side has its own standard $t$-structure.

\begin{lemma}
\label{etale realization is t-exact}
Let $X$ be a smooth formal scheme over $\Spf(\mathcal{O}_K)$, then the base change
functor compose with the above \'{e}tale realization functor
$\big(- \otimes_{\mathcal{O}_{\Prism}} \mathcal{O}_{\Prism}[1/I_{\Prism}]^{\wedge}_p\big)^{\varphi = 1}
\colon \FC^{\perf}(X) \to D^{(b)}_{lisse}(X_{\eta}, \mathbb{Z}_p)$
is $t$-exact.
\end{lemma}

\begin{proof}
It suffices to prove the statement locally on $X$, so we may assume $X = \Spf(R)$
is affine which has a Breuil--Kisin prism $(A, I)$ with $A/I = R$ and the Frobenius $\varphi_A \colon A/I[1/p] \to A/(\varphi(I))[1/p]$
is finite \'{e}tale:
Indeed, we may assume that $X$ admits a toric chart, then the Breuil--Kisin prism constructed in
\cite[Example 3.4]{DLMS} has these properties.
Suppose $\mathcal{E} \in \FC^{\coh}(X)$, we need to show its \'{e}tale realization
$T(\mathcal{E}) \in D^{(b)}_{lisse}(X_{\eta}, \mathbb{Z}_p)$ sits in cohomological degree $0$.

To that end, let $R_{\infty} \coloneqq A_{\perf}/I$ be the reduction of the perfection of $(A, I)$,
whose adic generic fiber is an affinoid perfectoid pro-finite-\'{e}tale cover of $X$.
It suffices to show that the restriction of
$T(\mathcal{E})$ to $\Spf(R_{\infty})_{\eta}$ sits in cohomological degree $0$.
Tracing through the proof of \cite[Corollary 3.7]{BS21}, we see that the restriction of $T(\mathcal{E})$
is computed (via the tilting equivalence) by the $\varphi$-invariants (as a pro-\'{e}tale sheaf) of 
$\mathcal{E}(A_{\perf})[1/I]^{\wedge}_p = \mathcal{E}(A) \otimes_A A_{\perf}[1/I]^{\wedge}_p$.
By definition, the perfect complex $\mathcal{E}(A)$ is given by a finitely presented $A$-module
in degree $0$. The ring map $A \to A_{\perf}[1/I]^{\wedge}_p$ is $p$-completely
flat, hence is flat thanks to Noetherianity of $A$. 
Therefore we see that $\mathcal{E}(A_{\perf})[1/I]^{\wedge}_p$ is also concentrated in degree $0$.
Now our claim is equivalent to the statement that $\varphi - 1$ on $\mathcal{E}(A_{\perf})[1/I]^{\wedge}_p$
is surjective (as pro-\'{e}tale $\mathbb{Z}_p$-sheaves on $A_{\perf}[1/I]^{\wedge}_p$).
This is well-known, but let us still sketch a proof: Since $\mathcal{E}(A_{\perf})[1/I]^{\wedge}_p$
is derived $p$-complete, the cokernel is also derived $p$-complete, therefore it suffices
to show its mod $p$ reduction is $0$. 
Equivalently it suffices to show the (derived) $\varphi$-invariants of the derived mod $p$
reduction $\mathcal{E}(A_{\perf})[1/I]/p$ lives in degrees $\leq 0$,
let us denote it by $T(\mathcal{E}/^L p)$.

Finally we may use \cite[Proposition 3.6]{BS21}: The authors show that after trivializing $T(\mathcal{E} /^L p)$
on an \'{e}tale neighborhood $U$ of $\mathrm{Spec}(A_{\perf}[1/I]/p)$,
so that $T(\mathcal{E} /^L p)(U)$ is just an $\mathbb{F}_p$ complex (which we need to show to be concentrated in degrees $\leq 0$),
we have an isomorphism
$T(\mathcal{E} /^L p)(U) \otimes_{\mathbb{F}_p} \mathcal{O}_U \simeq \mathcal{E}(A_{\perf})[1/I]/^L p \otimes_{A_{\perf}[1/I]/p} \mathcal{O}_U$, as pro-\'{e}tale sheaves.
Now we claim victory as the right hand side lives in degrees $\leq 0$.
\end{proof}

Another notion of central interest to our article is the \emph{height} of an $I$-torsionfree 
prismatic $F$-crystal.
\begin{definition}
\label{def height}
Let $X$ be a smooth formal scheme over $\Spf(\mathcal{O}_K)$,
and let $(\mathcal{E},\varphi_\mathcal{E}) \in \FC^{I\text{-tf}}(X)$.
We say it has \emph{height} in $[a, b]$ if the linearized Frobenius
\[
\varphi_\mathcal{E} \colon \varphi_{\mathcal{O}_\Prism}^* \mathcal{E} \longrightarrow \mathcal{E}[1/I]
\]
has its image squeezed in the following manner:
$I^b \mathcal{E} \subset \mathrm{Im}(\varphi_\mathcal{E}) \subset I^a \mathcal{E}$,
viewed as sub-$\mathcal{O}_\Prism$-modules in $\mathcal{E}[1/I]$.
When $[a,b] = [0,h]$, we simply say it is effective of height $\leq h$,
this terminology is compatible with \cite[Definition 3.14]{DLMS}.
\end{definition}

\begin{remark}
One can check height of $(\mathcal{E}, \varphi_\mathcal{E})$ on a covering family of prisms
whose reduction is flat over $X$.
\end{remark}

\begin{example}
Fix notation as above, the prismatic structure sheaf with its Frobenius is effective 
and has height $\leq 0$.
Assuming that $X$ is smooth proper of relative equidimension $d$ over $\Spf(\mathcal{O}_K)$,
then the $i$-th derived pushforward of $\mathcal{O}_{\Prism}$ has its $I$-torsionfree quotient
being effective of height $\leq \min\{i, d\}$:
\cite[Corollary 15.5]{BS19} shows the $\leq i$ inequality, and the $\leq d$
inequality can be directly proved using \cite[Lemma 7.8]{LL20}.
In fact the bound in terms of dimension $d$ holds more generally,
one can translate an argument due to Bhatt (see \cite[Remark 8.11]{GR22})
to the following concrete estimate:
if $X \to Y$ is an equi-$d$-dimensional smooth (without properness!)
map of smooth formal schemes over $\Spf(\mathcal{O}_K)$
and let $(\mathcal{E}, \varphi_\mathcal{E})$ be an $I$-torsionfree effective $F$-crystal on $X$ with height $\leq h$,
then for any integer $i$ the $i$-th derived pushforward $(R^i f_{\Prism, *}\mathcal{E}, \varphi_\mathcal{E})$
has its $I$-torsionfree quotient of height $\leq (d + h)$.
Our paper aims at generalizing these height estimates, see \Cref{bound of Frob height of cohomology}.
\end{example}

\subsection{Prismatic $F$-gauges}
\label{Prismatic $F$-gauges}
In this subsection we shall recall the notion of prismatic $F$-gauges, introduced by Drinfeld and Bhatt--Lurie.
Let us start with $F$-gauges over a quasiregular semiperfectoid ring.
\begin{definition}[{\cite[Definition 5.5.17 and Example 6.1.7]{Bha23}}]
\label{def F-gauge qrsp}
	Let $S$ be a quasiregular semiperfectoid ring.

A \emph{prismatic gauge} on $\Spf(S)$ is a $(p,I)$-complete filtered complex $\Fil^\bullet E$ over $\Fil_N^\bullet \Prism_S$. 
A \emph{prismatic $F$-gauge} $E=(E, \Fil^\bullet E, \widetilde{\varphi}_{E})$ on $\Spf(S)$ consists of the following data:
	\begin{itemize}
		\item a prismatic gauge $(E, \Fil^\bullet E)$, called \emph{the underlying gauge} of $E$;
		\item a map of filtered complexes
		\[
		\widetilde{\varphi}_{E}: \Fil^\bullet E \longrightarrow I^\mathbb{Z}\Prism_S\otimes_{\Prism_S} E,
		\]
which is linear over the filtered map $\varphi_{\Prism_S} \colon \Fil_N^\bullet \Prism_S \to I^\bullet\Prism_S$;
		\item such that the filtered linearization 
		\[
		\varphi_{E}: \Fil^\bullet E\otimes_{\Fil_N^\bullet \Prism_S} I^\mathbb{Z} \Prism_S \longrightarrow I^\mathbb{Z} E
		\]
		is a filtered isomorphism in $\DF_{(p,I)\text{-comp}}(I^\mathbb{Z}\Prism_S) \simeq \mathrm{D}_{(p,I)\text{-comp}}(\Prism_S)$.
	\end{itemize}
\end{definition}
Following the convention of \cite{Bha23}, we call $\Fil^\bullet E$ a \emph{Nygaardian filtration} on $E$.
So a prismatic gauge is nothing but a prismatic crystal equipped with a Nygaardian filtration.
\begin{definition}\label{def F-gauge category}
	Let $S$ be a quasiregular semiperfectoid ring.
\begin{enumerate}
\item We use $\fG(\Spf(S))$ to denote the category of prismatic ($F$-)gauges over $\Spf(S)$.
\item We use $\fG^*(\Spf(S))$ with $*\in\{\vect, \perf\}$ to denote the subcategory of $\fG(\Spf(S))$
  consisting of those $E\in \fG(\Spf(S))$ such that the underlying filtered 
  $\Fil^\bullet_N\Prism_S$-complex $\Fil^\bullet E$ is 
  filtered $(p, I)$-completely finite projective (resp.~filtered $(p, I)$-completely perfect),
  see \Cref{Definition: filtered complete finite projective} and \Cref{Definition: filtered complete perfect}.
\item We use $\fG^\coh(\Spf(S))$ to denote the full subcategory of $\fG^\perf(\Spf(S))$ 
  such that the underlying gauge is a module over $\Prism_S$ filtered by submodules.
\end{enumerate}


\begin{remark}
The mapping space between two coherent ($F$-)gauges on  $\Spf(S)$ is discrete.
\end{remark}

Below let us show that for gauges over quasiregular semiperfectoid rings,
there is no difference between filtered completely
finite projective (resp.~filtered completely perfect) versus filtered finite projective
(resp.~filtered perfect).

\begin{proposition}
\label{Gauge: cplt vs alg conditions}
Let $S$ be a quasiregular semiperfectoid ring and let 
$\Fil^\bullet E \in \mathrm{Gauge}(\Spf(S))$. Then we have:
\begin{enumerate}
\item $\Fil^{\bullet} E$ is filtered finite projective over $\Fil^{\bullet}_N\Prism_S$ if and
only if it is filtered $(p, I)$-completely finite projective over $\Fil^{\bullet}_N\Prism_S$; and
\item $\Fil^{\bullet} E$ is filtered perfect over $\Fil^{\bullet}_N\Prism_S$ if and
only if it is filtered $(p, I)$-completely perfect over $\Fil^{\bullet}_N\Prism_S$.
\end{enumerate}
\end{proposition}

\begin{proof}
According to \cite[Lemma 4.28]{ALB23}, the pair $(\Prism_S, \Fil^1_N \Prism_S)$
is henselian. Then by \cite[Lemma 4.27]{ALB23} 
(or \cite[\href{https://stacks.math.columbia.edu/tag/0D4A}{Tag 0D4A}]{stacks-project}),
we see that every finite projective $(\Gr^0_N \Prism_S = S)$-module
can be lifted to a finite projective $\Prism_S$-module.
Moreover, we claim that for every finite projective $\Prism_S$-module $P$,
every direct summand $\overline{P_1} \subset \overline{P} \coloneqq P \otimes_{\Prism_S} S$
can be lifted to a direct summand $P_1 \subset P$: By the previous sentence,
we may first lift $\overline{P}_1$ to a finite projective $P_1$
over $\Prism_S$, then (by the finite projectiveness of $P_1$) we may lift the map $\overline{P_1} \subset \overline{P}$
to a map $P_1 \to P$, which is necessarily also a direct summand because
$\Fil^1_N \Prism_S$ is contained in the Jacobson ideal of $\Prism_S$.
Our proposition now follows from \Cref{criteria for filtered vb}
and \Cref{criteria for filtered perfect} respectively.
\end{proof}

\end{definition}
\begin{remark}\label{functorial abs}
	For a map of quasiregular semiperfectoid rings $S_1\to S_2$ in $X_\qrsp$ and $*\in\{\emptyset, \vect, \perf\}$, the completed filtered base change induces a natural functor
	\begin{align*}
		\Phi_{(S_1,S_2)}:\fG^*(\Spf(S_1)) &\longrightarrow \fG^*(\Spf(S_2)),\\
		(E, \Fil^\bullet E, (\widetilde{\varphi}_{E})) &\longmapsto 
  (E\bigotimes_{\Prism_{S_1}} \Prism_{S_2}, \Fil^\bullet E\bigotimes_{\Fil^\bullet_N \Prism_{S_1}} \Fil^\bullet_N \Prism_{S_2}, (\widetilde{\varphi}_{E}\otimes \varphi_{\Prism_{S_2}})).
	\end{align*}
    From the construction, the induced functor on underlying ($F$-)crystals is the natural pullback functor.
\end{remark}
By taking limit of the above categories, we obtain the definition of ($F$-)gauges for more general schemes.
\begin{definition}\label{def F-gauge general}
	Let $X$ be a quasi-syntomic formal scheme.
	The category of \emph{prismatic ($F$-)gauges} over $X$ is defined as the limit of $\infty$-categories
	\[
	\fG^*(X)\coloneqq\lim_{S\in X_\qrsp} \fG^*(\Spf(S)),
	\]
	where $*\in \{\emptyset, \vect, \perf, \coh\}$.
\end{definition}
\begin{remark}
There is a natural functor from ($F$-)gauges to ($F$-)crystals by forgetting the filtration,
and it is compatible with the limit formula in \cite[Proposition 2.14]{BS21}.
\end{remark}
\begin{remark}
In the stacky formalism of \cite{Bha23}, the categories $\fG^*(\Spf(S))$ for $*\in\{\emptyset, \vect, \perf\}$ 
are equivalent to the categories of complexes (resp.~vector bundles, resp.~perfect complexes) over the 
filtered prismatization $X^{\mathcal{N}}$ (resp.~syntomification $X^\mathrm{syn}$) of $X$. 
However the analogous statement is probably not true when $* = \coh$: Our notion of coherent $F$-gauges
should be stronger than being a ``coherent sheaf on the syntomification of $X$''.
\end{remark}

A key fact is that the category of ($F$-)gauges form a sheaf under quasi-syntomic topology. 
Before proving this, let us record some useful flatness criteria in commutative algebra.

\begin{lemma}[{\cite[\href{https://stacks.math.columbia.edu/tag/0H7N}{Tag 0H7N}]{stacks-project}}]
\label{useful flatness criterion 1}
Let $A \to B$ be a ring map, let $f \in A$ be a nonzerodivisor in both $A$ and $B$.
Assume that $A/f \to B/f$ and $A[1/f] \to B[1/f]$ are (faithfully) flat, then $A \to B$ is (faithfully) flat.
If $J \subset A$ is a finitely generated ideal, then the analogous statement with all ``(faithfully) flat''
replaced by ``$J$-completely (faithfully) flat'' also holds.
\end{lemma}

\begin{proof}
Let us prove the ``flat'' statement first, the faithful part is easy.
Let $M$ be an $A$ module, denote by $N$ the submodule consisting of $f$-power torsion elements,
then $M/N$ is an $f$-torsion free $A$ module.
Since $\mathrm{Tor}_i^A(M, B)$ is squeezed between $\mathrm{Tor}_i^A(N, B)$
and $\mathrm{Tor}_i^A(M/N, B)$, it is equivalent to showing the later two vanish for $i > 0$.

Let us treat $N$ first. Since $N$ is the filtered colimit of its finitely generated submodules $N'$
each of which is annihilated by a power (depending on the finitely generated submodule) of $f$,
we are further reduced to showing: If $f^c$ annihilates $N'$, then $N' \otimes^L_A B$ lives in degree $0$.
Now we write $N' \otimes^L_A B = N' \otimes^L_{A/f^c} A/f^c \otimes^L_A B = N' \otimes^L_{A/f^c} B/f^c$:
For the second equality we use the fact that $f$ is a nonzerodivisor in $B$.
We may appeal to \cite[\href{https://stacks.math.columbia.edu/tag/051C}{Tag 051C}]{stacks-project}
and see that $A/f^c \to B/f^c$ is flat.

Now we treat $M/N$. First we observe $M/N \otimes^L_A B \otimes^L_B B/f = (M/N)/^L f \otimes^L_{A/f} B/f$
which lives in degree $0$ by $f$-torsion free assumption on $M/N$ and flatness of $A/f \to B/f$. Therefore we see that
$\mathrm{Tor}_i^A(M/N \otimes^L_A B)$ are $f$-invertible for all $i > 0$:
Indeed one simply observes that $\mathrm{Tor}_i^A(M/N \otimes^L_A B)/f$ and $\mathrm{Tor}_{i-1}^A(M/N \otimes^L_A B)[f]$
are subquotients of $\mathrm{H}_i((M/N \otimes^L_A B)\otimes^L_B B/f)$.
Hence we may invert $f$ and get $\mathrm{Tor}_i^A(M/N, B) = \mathrm{Tor}_i^{A[1/f]}(M/N[1/f], B[1/f]) = 0$ for all $i > 0$,
here the first equality follows from previous sentence and the second equality follows from the flatness assumption
on $A[1/f] \to B[1/f]$.

For the $J$-complete analog, one simply runs the above argument with the starting assumption that $M$ is an $A/J$ module.
\end{proof}

\begin{lemma}
\label{useful flatness criterion 2}
Let $(A, \Fil_{\bullet}(A)) \to (B, \Fil_{\bullet}(B))$ be a filtered map of commutative
unital rings equipped with increasing exhaustive multiplicative $\mathbb{N}$-indexed filtrations.
If the graded ring map $\Gr_{\bullet}(A) \to \Gr_{\bullet}(B)$ is (faithfully) flat, then $A \to B$ is (faithfully) flat.
If $J \subset \Fil_0(A)$ is a finitely generated ideal and if the increasing filtrations on $A$ and $B$
are $J$-completely exhaustive, then the analogous statement with all ``(faithfully) flat''
replaced by ``$J$-completely (faithfully) flat'' also holds.
\end{lemma}

\begin{proof}
It is equivalent to showing the following: let $K \subset A$ be an ideal, then
$K \otimes_A B \to B$ is injective (for the faithfully flat statement
we need to further show that this map is not surjective unless $K = A$ is the unit ideal).
Now we equip $K$ with the induced filtration $\Fil_i(K) = K \cap \Fil_i(A)$,
and consider the following filtered map $\Fil_{\bullet}(K \otimes_A B) \coloneqq 
\Fil_{\bullet}(K) \otimes^L_{\Fil_{\bullet}(A)} \Fil_{\bullet}(B) \to \Fil_{\bullet}(B)$.
The source is equipped with exhaustive increasing $\mathbb{N}$-indexed filtration and the above map has its graded pieces given by
$\Gr_{\bullet}(K \otimes_A B) = \Gr_{\bullet}(K) \otimes^L_{\Gr_{\bullet}(A)} \Gr_{\bullet}(B) \to \Gr_{\bullet}(B)$.
As $\Gr_{\bullet}K$ is a graded ideal inside of $\Gr_{\bullet}A$, the
(faithfully) flatness of $\Gr_{\bullet}(A) \to \Gr_{\bullet}(B)$ implies that the above has its source
living in degree $0$ and the induced map is injective (and not surjective unless $\Gr_{\bullet}(K) = \Gr_{\bullet}(A)$).
This together with the snake lemma implies that $K \otimes_A B \to B 
= \colim_{i \to \infty} (\Fil_i(K \otimes_A B) \to \Fil_i(B))$ is injective 
(and not surjective unless $K = A$).
For the completely (faithfully) flat statement, one just runs the above argument with
the assumption that $K$ contains $J$.
\end{proof}

The quasi-syntomic descent of ($F$-)gauges follows from \cite[Theorem 5.5.10 and Remark 5.5.18]{Bha23},
and since our definition (which does not involve the stacky approach) a priori does not agree with the one given in Bhatt's notes,
we give a direct proof.

\begin{proposition}[c.f.~{\cite[Theorem 5.5.10 and Remark 5.5.18]{Bha23}}]
\label{F-gauge sheaf}
	Let $S\to S^{(0)}$ be a quasi-syntomic cover of quasiregular semiperfectoid rings, and let $S^{(\bullet)}$ be the $p$-completed \v{C}ech nerve.
	Then for $*\in \{\emptyset, \perf,\vect\}$, we have a natural equivalence
	\[
	\fG^*(\Spf(S)) \simeq \lim_{[n]\in \Delta} \fG^*(\Spf(S^{(n)})).
	\]
\end{proposition}

\begin{proof}
Let us treat the descent of gauges first, we fix a map $R \to S$ where $R$ is perfectoid
with its associated perfect prism $(A, I = (d))$.
Via Rees's construction (see \cite[\S 2.2]{Bha23}), we can regard gauges as graded $(p, I)$-complete complexes
over the graded ring
$\mathrm{Rees}(\Fil^{\bullet}_N) \coloneqq \bigoplus_{i \in \mathbb{Z}} \Fil^i_N t^{-i}$.
By \Cref{completely filtered fpqc descent}, the descent of $*$-gauges
follows from the following two statements:
\begin{enumerate}
    \item The map $\mathrm{Rees}(\Fil^{\bullet}_N(\Prism_S)) \to \mathrm{Rees}(\Fil^{\bullet}_N(\Prism_{S^{(0)}}))$ is $(p, I)$-completely
    faithfully flat.
    \item For any map of qrsp algebras $S \to \widetilde{S}$ with $\widetilde{S}^{(0)} \coloneqq \widetilde{S} \widehat{\otimes}_S S^{(0)}$, we have a filtered isomorphism:
    \[
    \Fil^{\bullet}_N(\Prism_{\widetilde{S}}) \widehat{\otimes}_{\Fil^{\bullet}_N(\Prism_S)} \Fil^{\bullet}_N(\Prism_{S^{(0)}}) \xrightarrow{\cong} \Fil^{\bullet}_N(\Prism_{\widetilde{S}^{(0)}}).
    \]
\end{enumerate}
Let us prove (1) first. Using \Cref{useful flatness criterion 1} (with the $f$ there being $t$ here),
we see that it suffices to show that both the underlying ring map $\Prism_S \to \Prism_{S^{(0)}}$
and the graded algebra map 
\[
\bigoplus_{i \in \mathbb{N}} \Gr^i_N(\Prism_S) t^{-i} \to \bigoplus_{i \in \mathbb{N}} \Gr^i_N(\Prism_{S^{(0)}}) t^{-i}
\]
are completely faithfully flat.
By \cite[Theorem 12.2]{BS19}, the graded algebra is identified with Rees's construction of the (twisted) conjugate filtered Hodge--Tate algebra 
(in particular the filtration satisfies the assumption in \Cref{useful flatness criterion 2}), namely
\[
\bigoplus_{i \in \mathbb{N}} \Gr^i_N(\Prism_S) t^{-i} \cong \bigoplus_{i \in \mathbb{N}}\Fil_{i}^\conj\overline{\Prism}_S\{i\} t^{-i}
\]
(similarly with $S$ replaced by $S^{(0)}$).
Applying again \Cref{useful flatness criterion 1} (with the $f$ there being $d/t$ here),
we see that it suffices to show
$\overline{\Prism}_S \to \overline{\Prism}_{S^{(0)}}$ and 
$\bigoplus_{i \in \mathbb{N}} \Gr_i^\conj(\overline{\Prism}_S) t^{-i} \to \bigoplus_{i \in \mathbb{N}} \Gr_i^\conj(\overline{\Prism}_{S^{(0)}}) t^{-i}$ 
are completely faithfully flat.
Here we have also used the fact that $\Prism_S \to \Prism_{S^{(0)}}$ is $(p, I)$-completely faithfully flat
if and only if $\overline{\Prism}_S \to \overline{\Prism}_{S^{(0)}}$ is $p$-completely faithfully flat.
Now we apply \Cref{useful flatness criterion 2} to the conjugate filtered Hodge--Tate rings,
we see that all we need to show is the $p$-completely faithfully flatness of the
graded algebra map (of conjugate filtration).
Since the graded algebra of conjugate filtration is the (derived) Hodge algebra, 
the map ($p$-completed) is identified with $\Gamma^*_{S}(\mathbb{L}_{S/R}[-1])^{\wedge} \to \Gamma^*_{S^{(0)}}(\mathbb{L}_{S^{(0)}/R}[-1])^{\wedge}$.
Since the target is $S^{(0)}$-algebra, we see this map can be factored through the base change along $S \to S^{(0)}$
which is $p$-completely faithfully flat. Finally, we are reduced to showing the map
$\Gamma^*_{S^{(0)}}(\mathbb{L}_{S/R}[-1] \widehat{\otimes}_S S^{(0)})^{\wedge} \to \Gamma^*_{S^{(0)}}(\mathbb{L}_{S^{(0)}/R}[-1])^{\wedge}$
(induced by the natural map of cotangent complexes)
is $p$-completely faithfully flat. This follows from the fact that 
\[
\mathrm{Cone}(\mathbb{L}_{S/R}[-1] \widehat{\otimes}_S S^{(0)} \to \mathbb{L}_{S^{(0)}/R}[-1]) = \mathbb{L}_{S^{(0)}/S}[-1]
\]
is a $p$-completely flat $S^{(0)}$-module thanks to the quasi-syntomicity of $S \to S^{(0)}$ and the assumption that $S^{(0)}$ is semiperfectoid.

Let us now show the statement (2) above. Such a filtered map is a filtered isomorphism
if and only if its underlying and graded maps are both isomorphisms.
By \cite[Theorem 12.2]{BS19}, the graded map is identified with:
    \[
    \big(\bigoplus_{i \in \mathbb{N}} \Fil_i^\conj(\overline{\Prism}_{\widetilde{S}}) t^{-i}\big)
    \widehat{\otimes}_{\big(\bigoplus_{i \in \mathbb{N}} \Fil_i^\conj(\overline{\Prism}_S) t^{-i}\big)} 
    \big(\bigoplus_{i \in \mathbb{N}} \Fil_i^\conj(\overline{\Prism}_{S^{(0)}}) t^{-i}\big)
    \to \big(\bigoplus_{i \in \mathbb{N}} \Fil_i^\conj(\overline{\Prism}_{\widetilde{S}^{(0)}}) t^{-i}\big).
    \]
Via the Rees's construction, the above being an isomorphism is equivalent to the following
filtered map being an isomorphism:
\[
\Fil_{\bullet}^\conj(\overline{\Prism}_{\widetilde{S}}) \widehat{\otimes}_{\Fil_{\bullet}^\conj(\overline{\Prism}_S)} 
\Fil_{\bullet}^\conj(\overline{\Prism}_{S^{(0)}}) 
\xrightarrow{\cong} \Fil_{\bullet}^\conj(\overline{\Prism}_{\widetilde{S}^{(0)}}).
\]
Now for a map of increasingly $\mathbb{N}$-indexed $p$-completely exhaustively filtered
complexes, being a filtered isomorphism is equivalent to its graded map being an isomorphism.
After taking graded algebras, we are reduced to showing the following graded
map being an isomorphism:
\[
\Gamma^*_{\widetilde{S}}(\mathbb{L}_{\widetilde{S}/R}[-1])^{\wedge} \widehat{\otimes}_{\Gamma^*_{S}(\mathbb{L}_{S/R}[-1])^{\wedge}} 
\Gamma^*_{S^{(0)}}(\mathbb{L}_{S^{(0)}/R}[-1])^{\wedge}
\xrightarrow{\cong} \Gamma^*_{\widetilde{S}^{(0)}}(\mathbb{L}_{\widetilde{S}^{(0)}/R}[-1])^{\wedge}.
\]
The left hand side is identified with
$\Gamma^*_{\widetilde{S}^{(0)}}\big(\mathrm{Cone}(\mathbb{L}_{S/R}[-1] \widehat{\otimes}_S \widetilde{S}^{(0)} \to
\mathbb{L}_{S^{(0)}/R}[-1] \widehat{\otimes}_{S^{(0)}} \widetilde{S}^{(0)} \oplus 
\mathbb{L}_{\widetilde{S}/R}[-1] \widehat{\otimes}_{\widetilde{S}} \widetilde{S}^{(0)})\big)^{\wedge}$,
and we are further reduced to showing the following map of $\widetilde{S}^{(0)}$-modules
is an isomorphism:
\[
\mathrm{Cone}(\mathbb{L}_{S/R}[-1] \widehat{\otimes}_S \widetilde{S}^{(0)} \to
\mathbb{L}_{S^{(0)}/R}[-1] \widehat{\otimes}_{S^{(0)}} \widetilde{S}^{(0)} \oplus 
\mathbb{L}_{\widetilde{S}/R}[-1] \widehat{\otimes}_{\widetilde{S}} \widetilde{S}^{(0)})
\to \mathbb{L}_{\widetilde{S}^{(0)}/R}[-1]^{\wedge},
\]
which follows from the fact that $\widetilde{S}^{(0)}$ is given by $p$-completely (derived) tensor of
$S^{(0)}$ and $\widetilde{S}$ over $S$ and functoriality of cotangent complexes.
This finishes the proof of the graded part of the Nygaard-filtered base change formula,
and it also simultaneously proves the part about underlying complexes:
Since in the process we have showed the following
filtered map being an isomorphism:
\[
\Fil_{\bullet}^\conj(\overline{\Prism}_{\widetilde{S}}) \widehat{\otimes}_{\Fil_{\bullet}^\conj(\overline{\Prism}_S)} 
\Fil_{\bullet}^\conj(\overline{\Prism}_{S^{(0)}}) 
\xrightarrow{\cong} \Fil_{\bullet}^\conj(\overline{\Prism}_{\widetilde{S}^{(0)}}).
\]
Taking its underlying map, we get the following isomorphism:
$\overline{\Prism}_{\widetilde{S}} \widehat{\otimes}_{\overline{\Prism}_S} \overline{\Prism}_{S^{(0)}}
\xrightarrow{\cong} \overline{\Prism}_{\widetilde{S}^{(0)}}$,
which in turn shows that its $I$-completely deformed map is also an isomorphism:
$\Prism_{\widetilde{S}} \widehat{\otimes}_{\Prism_S} \Prism_{S^{(0)}}
\xrightarrow{\cong} \Prism_{\widetilde{S}^{(0)}}$ (by derived Nakayama's Lemma).

As for the descent of $F$-gauges, since we have established the descent of gauges
we are reduced to showing descent of the Frobenius morphism
which follows from flat descent.
\end{proof}

\begin{lemma}
\label{descent of coherent gauge is coherent}
In \Cref{F-gauge sheaf}, the descent of coherent ($F$-)gauges
is again coherent.
\end{lemma}

\begin{proof}
Perfectness was already shown in loc.~cit.,
so all we need to check is the cohomological concentration property, which is
equivalent to both $\Fil^{\bullet}$ and $\Gr^{\bullet}$ being concentrated in cohomological degree $0$.
Both can be seen by exactly following the proof of \Cref{complete descent of modules is a module}.
\end{proof}

The following is our main theorem in this subsection,
which says that any coherent $I$-torsionfree prismatic $F$-crystal can be canonically extended
to a coherent prismatic $F$-gauge equipped with ``the saturated Nygaardian filtration''.
This extends the special case of $X=\Spf(\mathcal{O}_K)$ in \cite[Theorem 6.6.13]{Bha23} to arbitrary 
smooth $p$-adic formal schemes over $\Spf(\mathcal{O}_K)$.
\begin{theorem}
\label{thm crystal to gauge}
Let $X$ be  a smooth formal scheme over $\mathcal{O}_K$.
\begin{enumerate}
\item There is a functor 
\[
\Pi_X: \FC^{I\text{-}\mathrm{tf}}(X) \longrightarrow \FG^\coh(X),
\]
characterized by the requirement that for any 
$S \in X_\qrsp$, where $\Spf(S) \to X$ is $p$-completely flat,
the restriction of $\Pi_X(\mathcal{E})$ on $\Spf(S)_\Prism$ satisfies
$\Fil^\bullet(\Pi_X(\mathcal{E})(\Prism_S)) = \widetilde{\varphi}_\mathcal{E}^{-1}(I^\bullet \mathcal{E}(\Prism_S))$.
Here we are regarding both the source and target as additive categories.
\item The functor $\Pi_X$ is right adjoint to the forgetful functor from $I$-torsionfree $F$-gauges to
$I$-torsionfree $F$-crystals.
In fact, for any pair of $(E_1, \Fil^{\bullet} E_1, \widetilde{\varphi}_{E_1}) \in \FG^{\coh}(X)$ 
and $(\mathcal{E}_2, \widetilde{\varphi}_{\mathcal{E}_2}) \in \FC^{I\text{-}\mathrm{tf}}(X)$, we have an identification of homomorphisms:
\[
\mathrm{Hom}_{\FC^{\coh}(X)}(E_1, \mathcal{E}_2) = \mathrm{Hom}_{\FG^{\coh}(X)}(E_1, \Pi_X(\mathcal{E}_2)).
\]
\item The functor $\Pi_X$ is compatible with \'{e}tale pullback in $X$.
\end{enumerate}
\end{theorem}

We need some preparatory discussion on local situations first.
We fix the following assumption for the rest of the subsection.
\begin{situation}
\label{local situation}
Let $U = \Spf(R)$ be an affine open of $X$, and let $(A, I = (d))$ be an
oriented Breuil--Kisin prism of $U$
whose existence is again guaranteed by deformation theory (see for example \cite[Example 3.4]{DLMS}).
Denote the perfection $(A_{\perf}, I A_{\perf})$ by $(A^{(0)}, I A^{(0)})$.
We denote the value of $\mathcal{E}$ on $(A, I)$ and $(A^{(0)}, I A^{(0)})$ by $M$ and $M^{(0)}$,
they are equipped with linearized Frobenii $\varphi_M$ and $\varphi_{M^{(0)}}$ respectively.
\end{situation}

\begin{construction}
\label{twist filtration}
Recall that the $F$-crystal structure provides us with a linearized Frobenii
$\varphi^*_A M \xrightarrow{\varphi_M} M[1/I]$ and
$\varphi^*_{A^{(0)}} M^{(0)} \xrightarrow{\varphi_{M^{(0)}}} M^{(0)}[1/I]$.
We define the \emph{twisted filtration} on $\varphi^*_A M$ (resp.~$\varphi^*_{A^{(0)}} M^{(0)}$)
by the following formula
\[
\Fil^{\bullet}_{\mathrm{tw}} \varphi^*_A M \coloneqq \varphi_M^{-1}\big(I^{\bullet} M\big)
\text{ (resp.~}\Fil^{\bullet}_{\mathrm{tw}} \varphi^*_{A^{(0)}} M^{(0)}
\coloneqq \varphi_{M^{(0)}}^{-1}\big(I^{\bullet} M^{(0)}\big)\text{).}
\]
\end{construction}

\begin{lemma}
\label{boundedness results on twist graded}
Let $E$ be an $I$-torsionfree $F$-crystal having height in $[a, b]$,
and denote the associated graded of the twisted filtration on
$\varphi^*_A M = \varphi^*_A E(A)$ above by $N$,
viewed as a graded $R[u] = \Gr_{I^{\mathrm{N}}}(A)$-module.
Then $N$ is $u$-torsionfree,
lives in degrees $\geq a$, and $N^{\deg = i} \xrightarrow{\cdot u} N^{\deg = i+1}$
is an isomorphism provided $i \geq b$.
In particular, both $N$ and $N[1/u]/N$ have bounded $p$-power torsion.
\end{lemma}

\begin{proof}
Since $E$ is $I$-torsionfree, we know that the linearized Frobenius
is injective.
Therefore we have an identification: $\Fil^i_{\mathrm{tw}}(M) = \mathrm{Im}(\varphi_M) \cap I^i M$.
Consequently, we have 
\[
N^{\deg = i} = \mathrm{Im}(\mathrm{Im}(\varphi_M) \cap I^i M
\to I^{i}M/I^{i+1}M),
\]
which immediately implies $u$-torsionfreeness.
The rest of the second sentence follows from the height condition.
It follows that the $R[u]$-module $N$ is generated by its degree $[a,b]$
pieces, 
and since each graded piece is finitely generated over $R$,
we see $N$ is finitely generated over $R[u]$. 
Since $R[u]$ is Noetherian, we see that
$N$ has bounded $p$-power torsion.
The $R$-module $N[1/u]/N$ is given by the direct sum of infinite copies of $M/I$
together with $\mathrm{Coker}(N^{\deg = i} \xrightarrow{\cdot u^{b-i}} N^{\deg = b})$ where $i \in [a,b-1]$.
This explicit description shows the boundedness claim of $N[1/u]/N$.
\end{proof}

We get the following consequence concerning perfectness of the twisted filtration
on $\varphi^*_A M$.

\begin{lemma}
\label{perfectness of twist filtration}
The filtered $I^{\mathbb{N}}A$-module $\Fil^{\bullet}_{\mathrm{tw}} \varphi^*_A M$
is filtered perfect.
\end{lemma}

\begin{proof}
Since $A$ is $I$-adically complete, we may apply \Cref{criteria for filtered perfect},
therefore we need to check that
\begin{enumerate}
\item The twisted filtration is eventually constant (see \Cref{constant filtration definition});
\item The underlying $\varphi^*_A M$ is perfect over $A$; and
\item the filtered base change
$\Fil^{\bullet}_{\mathrm{tw}} \varphi^*_A M \otimes_{I^{\mathbb{N}} A} A/I$
is filtered perfect over $A/I$.
\end{enumerate}
By \Cref{general lemma on eventual constancy}.(1) and \Cref{boundedness results on twist graded},
the twisted filtration is constant after $a$-th filtration.
(2) follows from the fact that $M$ is $A$-perfect.

Below let us show (3) above. Using the resolution
$0 \to I \otimes_A I^{\mathbb{N}}A \langle -1 \rangle \to I^{\mathbb{N}}A \to A/I \to 0$
of the filtered $I^{\mathbb{N}}A$-algebra $A/I = \Gr^0(I^{\mathbb{N}}A)$,
we see that the filtered base change has its filtered pieces given by:
\[
\Fil^i\bigl(\Fil^{\bullet}_{\mathrm{tw}} \varphi^*_A M \otimes_{I^{\mathbb{N}} A} A/I \bigr) = 
\mathrm{Coker}(I \otimes_A \Fil^{i-1}_{\mathrm{tw}} \varphi^*_A M \to \Fil^{i}_{\mathrm{tw}} \varphi^*_A M).
\]
Therefore we see that $\Fil^{\geq (b+1)}$ of the filtered base change is $0$.
Note that by our convention, the Rees's construction $\Rees(A/I) = R[t]$
is graded by $-\mathbb{N}$, and the filtered base change has bounded above grading by the previous sentence.
By \Cref{graded: perf vs cplt perf}, we are reduced to showing
\[
\Rees(\Fil^{\bullet}_{\mathrm{tw}} \varphi^*_A M \otimes_{I^{\mathbb{N}} A} A/I) \otimes_{\Rees(A/I)} A/I
\cong \Gr^{\ast}(\Fil^{\bullet}_{\mathrm{tw}} \varphi^*_A M \otimes_{I^{\mathbb{N}} A} A/I)
\cong N \otimes_{R[u]} R = N/^L u
\]
is graded perfect over $R$.
Here the first equivalence follows from the fact that
$A/I = \mathrm{Cone}(\Rees(A/I) \xrightarrow{\cdot t} \Rees(A/I)$,
and the second equivalence follows from the compatibility in
\Cref{filtered base change and graded base change}.
By \Cref{boundedness results on twist graded}, we see that $N/^L u$
is concentrated in gradings between $a$ and $b$.
By Noetherianity of $R$, we see that in each degree, they are all finitely generated
$R$-modules, therefore they are all perfect $R$-complexes as $R$ is regular.
\end{proof}

We need the following auxiliary lemma which says that the saturated Nygaardian filtration
on $\Pi_U(E) \mid_{A^{(0)}}$ has a model over the filtered ring $(A, I^{\mathbb{N}})$.

\begin{lemma}
\label{twist filtration and saturated filtration}
We have the following filtered base change formula between twisted filtrations
constructed in \Cref{twist filtration}:
\[
\Fil^{\bullet}_{\mathrm{tw}} \varphi^*_A M \widehat{\otimes}_{(A, I^{\mathbb{N}})} (A^{(0)}, I^{\mathbb{N}} A^{(0)})
\xrightarrow{\cong} \Fil^{\bullet}_{\mathrm{tw}} \varphi^*_{A^{(0)}} M^{(0)},
\]
as well as the following filtered base change formula between the twisted filtration
and the saturated Nygaardian filtration:
\[
\Fil^{\bullet}_{\mathrm{tw}} \varphi^*_{A^{(0)}} M^{(0)} \widehat{\otimes}_{(A^{(0)}, I^{\mathbb{N}} A^{(0)}), \varphi^{-1}} 
(A^{(0)}, \varphi^{-1}(I^{\mathbb{N}} A^{(0)}) = \Fil^{\bullet}_N(A^{(0)})) \xrightarrow{\cong} \Fil^{\bullet} M^{(0)}.
\]
\end{lemma}

Here we are using the fact that $A^{(0)}$ is a perfect prism, hence its Frobenius is invertible.

\begin{proof}
Note that we have the following relation between Rees algebras associated with various
filtered rings showing up in the statement:
\[
\mathrm{Rees}(I^{\mathbb{N}}A) \widehat{\otimes}_A A^{(0)} \cong \mathrm{Rees}(I^{\mathbb{N}} A^{(0)})^{\wedge}
\text{ and }
\]
\[
\mathrm{Rees}(I^{\mathbb{N}} A^{(0)}) \widehat{\otimes}_{A^{(0)}, \varphi^{-1}} A^{(0)}
\cong \mathrm{Rees}(\varphi^{-1}(I^{\mathbb{N}} A^{(0)}))^{\wedge} = \mathrm{Rees}(\Fil^{\bullet}_{N}\Prism_{A^{(0)}/I})^{\wedge}.
\]
Via Rees's construction (see \cite[\S 2.2]{Bha23}), our first base change formula is equivalent to the following formula:
\[
\Fil^{\bullet}_{\mathrm{tw}} \varphi^*_A M \widehat{\otimes}_A A^{(0)} 
\xrightarrow{\cong} \Fil^{\bullet}_{\mathrm{tw}} \varphi^*_{A^{(0)}} M^{(0)}
\]
which follows from the fact that $A \to A^{(0)}$ is flat and under this flat map
the $\varphi_M$ is base changed to $\varphi_{M^{(0)}}$.
Similarly, our second base change formula is equivalent to:
\[
\Fil^{\bullet}_{\mathrm{tw}} \varphi^*_{A^{(0)}} M^{(0)} \widehat{\otimes}_{A^{(0)}, \varphi^{-1}} 
A^{(0)} \xrightarrow{\cong} \Fil^{\bullet} M^{(0)}.
\]
Recall that the filtration $\Fil^{\bullet} M^{(0)}$ is defined by the preimage of $I$-adic filtration
under the semi-linear Frobenius $\widetilde{\varphi}_{M^{(0)}}$.
Therefore our second base change formula follows from the fact that
$\varphi^{-1}_{A^{(0)}}$ is an isomorphism and upon base changing the source via $\varphi^{-1}_{A^{(0)}}$
the linearized Frobenius $\varphi_{M^{(0)}}$ becomes the semi-linear Frobenius $\widetilde{\varphi}_{M^{(0)}}$.
\end{proof}

The following crucial proposition, which is inspired by the proof of \cite[Lemma 6.6.10]{Bha23},
shows that the saturated Nygaardian filtration satisfies filtered base change
formula in a particular situation.

\begin{proposition}
\label{saturated Nyg fil satisfies filtered base change proposition}
Let $R^{(0)} \coloneqq A^{(0)}/I$ and let $S$ be a $p$-completely flat quasi-syntomic $R^{(0)}$-algebra, 
then the natural filtered base change
is a filtered isomorphism:
\[
\left(M^{(0)}, \widetilde{\varphi}_\mathcal{E}^{-1}(I^{\bullet} M^{(0)})\right) \widehat{\otimes}_{(A^{(0)}, \Fil^{\bullet}_N)}
(\Prism_S, \Fil^{\bullet}_N) \xrightarrow{\cong} \left(\mathcal{E}(\Prism_S), \widetilde{\varphi}_\mathcal{E}^{-1}(I^{\bullet} \mathcal{E}(\Prism_S))\right).
\]
Consequently, if $S_1 \to S_2$ is a morphism in $X_{\qrsp}$ and suppose they are $p$-completely flat over some $R^{(0)}$
constructed out of an oriented Breuil--Kisin prism $(A, I = (d))$, then the natural filtered base change
is a filtered isomorphism:
\[
\left(\mathcal{E}(\Prism_{S_1}), \widetilde{\varphi}_\mathcal{E}^{-1}(I^{\bullet} \mathcal{E}(\Prism_{S_1}))\right) 
\widehat{\otimes}_{(\Prism_{S_1}, \Fil^{\bullet}_N)}
(\Prism_{S_2}, \Fil^{\bullet}_N) \xrightarrow{\cong} 
\left(\mathcal{E}(\Prism_{S_2}), \widetilde{\varphi}_\mathcal{E}^{-1}(I^{\bullet} \mathcal{E}(\Prism_{S_2}))\right).
\]
\end{proposition}

\begin{proof}
Since the map always induces an isomorphism of the underlying complex, thanks to $\mathcal{E}$ being a crystal,
all we need to check is that the tensor product filtration (defined by the left hand side of the above equation)
and the saturated Nygaardian filtration (defined by the right hand side of the above equation) agree.

Now we recall in the proof of \cite[Lemma 6.6.10]{Bha23}, Bhatt observes the following characterization
of the saturated Nygaardian filtration:
\begin{porism}[{Follows from the proof of \cite[Lemma 6.6.10]{Bha23}}]
\label{Porism in Bha23}
Let $A$ be a ring and let $(M_1, \Fil^{\bullet} M_1) \xrightarrow{\widetilde{\varphi}} (M_2, \Fil^{\bullet} M_2)$
be a map of two decreasing filtered objects in $\DF(A)$ such that 
\begin{enumerate}[label=(\roman*)]
\item $\Fil^{\ll 0} M_1 \to M_1$ is an isomorphism,
all $\Fil^{\bullet} M_1$'s are connective;
\item The graded complexes $\Gr^{\bullet}(M_1)$ are coconnective; 
\item $(\Fil^{\bullet} M_2)$ are $A$-modules with injective transitions; and
\item The graded map $\Gr(\widetilde{\varphi}) \colon \Gr^{\bullet}(M_1) \to \Gr^{\bullet}(M_2)$
have coconnective cone.
\end{enumerate}
Then $(M_1, \Fil^{\bullet} M_1)$ is also an honest decreasingly filtered $A$-module
with its filtration given by $\Fil^{\bullet} M_1 = \widetilde{\varphi}^{-1} (\Fil^{\bullet} M_2)$.
\end{porism}
\begin{proof}
This is an exercise in homological algebra, below we give some hints. 
The conditions (i) and (ii) imply that $(M_1, \Fil^{\bullet} M_1)$ is
an honest decreasingly filtered $A$-module.
Then condition (iii) and (iv) implies that $\Fil^{i+1} M_1 = \Fil^i M_1 \cap \widetilde{\varphi}^{-1} (\Fil^{i+1} M_2)$,
finishing the proof.
\end{proof}
We shall verify that the filtered map
\[
\left(M^{(0)}, \widetilde{\varphi}_\mathcal{E}^{-1}(I^{\bullet} M^{(0)})\right) \widehat{\otimes}_{(A^{(0)}, \Fil^{\bullet}_N)}
(\Prism_S, \Fil^{\bullet}_N) \xrightarrow{\widetilde{\varphi}_\mathcal{E}}
\left(\mathcal{E}(\Prism_S)[1/I], I^{\bullet} \mathcal{E}(\Prism_S)\right).
\]
satisfies the conditions of \Cref{Porism in Bha23}, which will show
that the tensor product filtration agrees with the saturated Nygaardian filtration.
The condition (i) is stable under filtered base
change between honestly filtered algebras, therefore the tensor product filtration
satisfies the condition (i) as $(M^{(0)}, \widetilde{\varphi}_\mathcal{E}^{-1}(I^{\bullet} M^{(0)}))$
clearly satisfies it.
The condition (iii) is easily seen to be satisfied.
Finally we need to verify (ii) and (iv) in the Porism above concerning the behavior of graded pieces.

For (ii): The graded pieces of the tensor product filtration is given by
$\Gr(M^{(0)}) \widehat{\otimes}_{\Gr_N(A^{(0)})} \Gr_N(\Prism_S)$.
By \Cref{twist filtration and saturated filtration}, we get the following description
of $\Gr(M^{(0)}) = \Gr_{\mathrm{tw}}(\varphi^*_A M) \widehat{\otimes}_{R[u]} \Gr_N(A^{(0)})$.
Here the tensor is $p$-completed and the base change map is given by taking graded map
of composition of the following filtered maps (where $u$ is the image of $d$ in $I/I^2$):
\[
(A, I^{\mathbb{N}} A) \to (A^{(0)}, I^{\mathbb{N}} A^{(0)}) \xrightarrow{\varphi^{-1}_{A^{(0)}}}
(A^{(0)}, \Fil^{\bullet}_N).
\]
In the proof of \Cref{F-gauge sheaf}, we have seen that the map
$\Gr_N(A^{(0)}) \to \Gr_N(\Prism_S)$ is $p$-completely flat. 
Precomposing with the $p$-completely flat map $R[u] \to \Gr_N(A^{(0)})$, we see that the induced map
$R[u] \to \Gr_N(\Prism_S)$ is $p$-completely flat.
Therefore we see that $\Gr_{\mathrm{tensor}}(E(\Prism_S)) = 
\Gr_{\mathrm{tw}}(\varphi^*_A M) \widehat{\otimes}_{R[u]} \Gr_N(\Prism_S)$
is concentrated in degree $0$ because $\Gr_{\mathrm{tw}}(\varphi^*_A M)$ has bounded $p$-power torsion
(\Cref{boundedness results on twist graded}).

As for (iv): 
just like the proof of \cite[Lemma 6.6.10]{Bha23}, we may identify the map
\[
\Gr_{\mathrm{tensor}}(E(\Prism_S)) \xrightarrow{\widetilde{\varphi}_\mathcal{E}} \Gr_{I\text{-adic}}(E(\Prism_S)[1/I])
\]
as the $p$-completely base change of $\Gr_N(M^{(0)}) \to \Gr_N(M^{(0)})[1/u]$
along the map $\Gr_N(A^{(0)}) \to \Gr_N(\Prism_S)$.
Using \Cref{twist filtration and saturated filtration} again, we may identify the above
map further as the following $p$-completely base changed map
\[
\left(\Gr_{\mathrm{tw}}(\varphi^*_A M) \to \Gr_{\mathrm{tw}}(\varphi^*_A M)[1/u]\right)
\widehat{\otimes}_{R[u]} \Gr_N(\Prism_S).
\]
Therefore it suffices to notice that 
$\left(\Gr_{\mathrm{tw}}(\varphi^*_A M)[1/u]/\Gr_{\mathrm{tw}}(\varphi^*_A M)\right)$
has bounded $p$-power torsion 
(\Cref{boundedness results on twist graded}).
\end{proof}

Now we are ready to prove \Cref{thm crystal to gauge}, let us stress again that 
it is inspired by \cite[Lemma 6.6.10]{Bha23}.

\begin{proof}[Proof of \Cref{{thm crystal to gauge}}]
(1): we adopt the notation in the discussion right after the statement.
Let $U = \Spf(R)$ be an affine open of $X$.
Let $R^{(0)} \coloneqq A^{(0)}/I$ and let $R^{(\bullet)}$ be the Cech nerve
of the quasi-syntomic cover $R \to R^{(0)}$ with their absolute prismatic
cohomology $A^{(\bullet)} \coloneqq \Prism_{R^{(\bullet)}}$.
We denote $E(A^{(\bullet)})$ by $M^{(\bullet)}$, and by abuse
of notation we denote the semi-linear Frobenii on these $M^{(\bullet)}$
by the same symbol $\widetilde{\varphi}_\mathcal{E}$.

Since $F$-gauges form a quasi-syntomic sheaf (\Cref{F-gauge sheaf}),
we need to first check that the ``saturated Nygaardian filtrations''
on $\mathcal{E}(\Prism_{R^{(i)}})$ satisfies filtered base change with respect
to the various maps induced by the simplicial maps between the $R^{(i)}$'s.
To that end, let $[i] \to [j]$ be an arrow in $\Delta$, we need to verify
the natural map
\[
\left(M^{(i)}, \widetilde{\varphi}_\mathcal{E}^{-1}(I^{\bullet} M^{(i)})\right) \widehat{\otimes}_{(A^{(i)}, \Fil^{\bullet}_N)}
(A^{(j)}, \Fil^{\bullet}_N) \to \left(M^{(j)}, \widetilde{\varphi}_\mathcal{E}^{-1}(I^{\bullet} M^{(j)})\right)
\]
is a filtered isomorphism.
This immediately follows from \Cref{saturated Nyg fil satisfies filtered base change proposition}.

This shows that the saturated Nygaardian filtration defines an $F$-gauge $\Pi_U(\mathcal{E}|_U)$ on $U$,
and the construction is easily seen to be functorial in $\mathcal{E}$ and $U$.
The perfectness follows from combining
\Cref{perfectness of twist filtration} and \Cref{twist filtration and saturated filtration}.
Next we verify that these $\Pi_U(\mathcal{E}|_U)$'s glue to an $F$-gauge on $X$.
To that end, it suffices to check the following
\begin{claim}
\label{saturated claim}
For any $\Spf(S) \in U_{\qrsp}$
with $p$-completely flat structural map to $U$, the value
$\Pi_U(\mathcal{E}|_U)(\Spf(S))$ has its filtration given by the saturated Nygaardian filtration.
\end{claim}

Granting the above claim, then for any two affine opens $U$ and $V$,
the restriction of $\Pi_U(\mathcal{E}|_U)|_{U \cap V}$ and $\Pi_V(\mathcal{E}|_V)|_{U \cap V}$
will be canonically identified: their values on any $\Spf(S)$ which are flat and quasi-syntomic
over $U \cap V$ are canonically
identified (the filtrations are given by the saturated Nygaardian filtration),
finally we just notice that these $\Spf(S)$'s form a basis of $(U \cap V)_{\qrsp}$,
and $F$-gauges form a quasi-syntomic sheaf (\Cref{F-gauge sheaf}).

\begin{proof}[Proof of the above claim:]
Let us base change the Cech nerve $R^{(\bullet)}$ along the map $R \to S$,
and denote the resulting Cech nerve $S^{(\bullet)}$.
The underlying gauge of $\Pi_U(\mathcal{E}|_U)(\Spf(S))$ is the descent of the gauges on $\Spf(S^{(\bullet)})$,
and the latter gauges are given by the filtered base change
$\left(M^{(i)}, \widetilde{\varphi}_\mathcal{E}^{-1}(I^{\bullet} M^{(i)})\right) \widehat{\otimes}_{(A^{(i)}, \Fil^{\bullet}_N)}
(\Prism_{S^{(i)}}, \Fil^{\bullet}_N)$.
Using \Cref{saturated Nyg fil satisfies filtered base change proposition}, we see that the latter
gauges are nothing but the saturated Nygaardian filtered module $\mathcal{E}(\Prism_{S^{(i)}})$.
By \Cref{descent of coherent gauge is coherent}, we see the descent gauge
$(\mathcal{E}(\Prism_S), \Fil_{\mathrm{descent}})$ is coherent,
namely it is an honest decreasingly filtered $\Prism_S$-module.
Finally we apply \Cref{Porism in Bha23} for the Frobenius filtered map 
$(\mathcal{E}(\Prism_S), \Fil_{\mathrm{descent}}^{\bullet})
\xrightarrow{\widetilde{\varphi}_\mathcal{E}} (\mathcal{E}(\Prism_S)[1/I], I^{\bullet}\mathcal{E}(\Prism_S))$:
we have verified conditions (i) and (ii); the condition (iii) follows from coherence of $\mathcal{E}$
and \Cref{value of t-inequality crystal on flat object};
finally the condition (iv) is stable under descent as taking limit preserves co-connectivity,
hence it is satisfied for our map (as this map is the descent of maps which satisfy condition (iv)).
Therefore we may conclude that the descent filtration is again the saturated Nygaardian filtration,
which finishes our proof of the claim.
\end{proof}

To finish the proof of (1), we need to see that for any $\Spf(S) \in X_{\qrsp}$
with $p$-completely flat structural map to $X$, the value
$\Pi_X(\mathcal{E})(\Spf(S))$ has its filtration given by the saturated Nygaardian filtration.
We may choose a Zariski cover of $\Spf(S)$ by affine opens such that each affine open
maps to an affine open in $X$. Then the filtration on $\Pi_X(\mathcal{E})(\Spf(S))$
is given by descent of saturated Nygaardian filtration (by the above claim),
therefore we can use the same argument as above again to conclude that the filtration
we get is the saturated Nygaardian filtration on $\mathcal{E}(\Prism_S)$.

Next we show (2): Let $E_1$ and $\mathcal{E}_2$ be as in the statement.
Notice that mapping spaces 
$\mathrm{Map}_{\FC^{\coh}(X)}(E_1, \mathcal{E}_2)$ and $\mathrm{Map}_{\FG^{\coh}(X)}(E_1, \Pi_X(\mathcal{E}_2))$
are discrete,
we see that the Hom groups are glued from the Hom groups
on affine opens of $U \subset X$, so we have reduced ourselves to the
case of $X$ being an affine.
In this case, the Hom groups are computed by the corresponding Hom groups 
of values of $E_1$, $\mathcal{E}_2$, and $\Pi_X(\mathcal{E}_2)$ at those $\Spf(S) \in X_{\qrsp}$ which
are flat over $X$.
So we are finally reduced to checking the analogous statement for $X$ replaced
by these qrsp $\Spf(S)$.
Since any Frobenius-equivariant map $E_1(\Prism_S) \to \mathcal{E}_2(\Prism_S)$ uniquely extends to
a filtered map $E_1(\Prism_S) \to \Pi_X(\mathcal{E}_2)(\Prism_S)$,
as the latter is equipped with the saturated Nygaardian filtration,
the two Hom groups are canonically identified as desired.

Lastly we deal with (3): 
Let $V \to U \to X$ be two affines in $X_{\acute{e}t}$,
we just need to show the natural functor $\FG^\perf(U) \to \FG^\perf(V)$ sends
$\Pi_U(\mathcal{E} \mid_U) \mapsto \Pi_V(\mathcal{E} \mid_V)$.
Choose a Breuil--Kisin prism $(A, I)$ for $U$, the \'{e}tale map $V \to U$ canonically
lifts $(A, I)$ to a Breuil--Kisin prism $(A', IA')$ over $(A, I)$ in $U_{\Prism}$.
Notice that $A \to A'$ is $(p, I)$-completely flat, hence flat by Noetherianity of $A$.
Tracing through the argument of proving (1), we see that it suffices to show
the twisted filtrations on $\varphi_A^* \mathcal{E}(A)$ and $\varphi_{A'}^* \mathcal{E}(A')$
are related via
\[
\Fil_{\mathrm{tw}} \varphi_A^* \mathcal{E}(A) \widehat{\otimes}_A A'
\xrightarrow{\cong} \Fil_{\mathrm{tw}} \varphi_{A'}^* \mathcal{E}(A'),
\]
which follows directly from the flatness of $A \to A'$
and how the twisted filtrations are defined.
\end{proof}

\begin{remark}
In fact we may define the functor on all of coherent $F$-crystals, by pulling back
filtrations from their $I$-torsionfree quotients.
The resulting graded modules will agree with the graded modules of the $F$-gauges
of their $I$-torsionfree quotients.
Since we do not need this greater generality, we restrict ourselves to working with
$I$-torsionfree coherent $F$-crystals only.
\end{remark}

\subsection{Weight filtration on graded pieces of gauges}
\label{absolute weight filtration subsection}
In this subsection, we introduce a so-called weight filtration on the associated graded of a gauge.

To start, we first define the notion of \emph{weight} for an ($F$-)gauge, 
using a reduction functor that sends the graded piece of a gauge over $X$ to a graded complex over $\mathcal{O}_X$.
\begin{construction}[c.f.~\Cref{reduction of graded B-cplx}]
\label{reduction of gauge}
	Let $X$ be a quasi-syntomic $p$-adic formal scheme, and let $\ast\in \{\emptyset, \perf,\vect\}$.
	We define the \emph{reduction functor} on the category $\DG^\ast_{p\text{-comp}}(X_\qrsp, \Gr_N^\bullet \Prism) \simeq \lim_{S\in X_\qrsp} \DG^\ast_{p\text{-comp}} (\Gr_N^\bullet \Prism_S)$ as the following composition of functors
	\[
	\begin{tikzcd}
		\Red_X \colonequals  \DG^\ast_{p\text{-comp}}(\Gr_N^\bullet \Prism) \arrow[rr, "-\otimes_{\Gr_N^\bullet \Prism} \mathcal{O}_\qrsp"] &&  \DG^\ast_{p\text{-comp}}(\mathcal{O}_X);\\
		 M^\bullet \arrow[rr, mapsto] && M^\bullet\bigotimes_{\Gr_N^\bullet \Prism} \mathcal{O}_\qrsp,
	\end{tikzcd}
    \]
    where $\mathcal{O}_\qrsp = \Gr_N^0 \Prism$ is regarded as a graded $\Gr_N^\bullet \Prism$-algebra via projection
    onto degree $0$ piece.
    The $i$-th graded piece of $\Red_X(-)$ is denoted as $\Red_{i,X}(-)$.
    Here we implicitly use the $p$-completely flat descent of $p$-complete complexes (resp. perfect complexes, resp. vector bundles) for the target category, namely
    \[
    \DG^\ast_{p\text{-comp}}(\mathcal{O}_X) \simeq \lim_{S\in X_\qrsp} \DG^\ast_{p\text{-comp}}(S).
    \]
    In particular, we have the natural limit formula for the reduction functor
    \[
    \Red_X \simeq \lim_{S\in X_\qrsp} \Red_S.
    \]
    
    When the choice of the formal scheme $X$ is clear, we omit $X$ in the subscripts and use $\Red$ and $\Red_i$ to abbreviate $\Red_X$ and $\Red_{i,X}$.
    By a slight abuse of notation, for a gauge $(E,\Fil^\bullet E)$ over $X$, we also abbreviate the notation $\Red_X(\Gr^\bullet E)$ as $\Red_X(E)$ when there is no confusion.
\end{construction}

We then define the notion of weights for a gauge.
\begin{definition}
\label{def of wt of gauge}
Let $X$ be a quasi-syntomic formal scheme, and let $[a, b]$ be an interval in $\mathbb{R}\cup \{-\infty, \infty\}$.
For a graded complex $M^\bullet\in \DG^\ast_{p\text{-comp}}(X_\qrsp, \Gr_N^\bullet \Prism)$,
we say it has \emph{weights} in $[a,b]$ if $M^n = 0$ for $n \ll 0$ and
\[
\Red_i(M^\bullet)=0,~\forall i\notin [a,b].
\]
For a gauge $E=(E,\Fil^\bullet E)$ on $X$, we say it has \emph{weights} in $[a,b]$ if
its associated graded $M^\bullet=\Gr^\bullet E$ is so.
\end{definition}

\begin{remark}
In the stacky language, the reduction functor on the category of gauges can be translated in terms of 
the pullback functor along the closed immersion $X\times B\mathbb{G}_m \to X^N$ as in \cite[Rmk.\ 5.3.14]{Bha23}.
Similarly the ``weights'' defined above corresponds to the ``Hodge--Tate weights'' defined in loc.~cit.
\end{remark}

Using the reduction functor, we now introduce the weight filtration on the associated graded of a gauge.
\begin{theorem}
\label{thm weight filtration of gauge}
Let $X$ be a quasi-syntomic $p$-adic formal scheme, let $\ast\in \{\emptyset, \perf,\vect\}$, 
and let $a\leq b$ be two integers.
Assume $M^\bullet \in \DG^\ast_{p\text{-}\mathrm{comp}}(X_\qrsp, \Gr_N^\bullet \Prism)$ 
has weights in $[a,b]$.
Then there is a unique increasing filtration
$\Fil^\wt_i (M^\bullet)$ on $M^\bullet$ satisfying the following axioms:
\begin{enumerate}
\item The filtration is exhaustive;
\item the filtration ``starts at $0$'': i.e.~$\Fil^\wt_{\ll 0} (M^\bullet) = 0$; and
\item the graded piece $\Gr^\wt_i(M^\bullet) \simeq N_i \otimes_{\mathcal{O}_X} \Gr_N^\bullet \Prism$ 
where $N_i \in \DG^\ast_{p\text{-comp}}(\mathcal{O}_X)$ has grading $i$.
\end{enumerate}
Moreover, the induced filtration $\Fil^\wt_i(\Red(M^{\bullet})) \coloneqq \Red(\Fil^\wt_i (M^\bullet))$
is the one induced by grading: 
\[
\Fil^\wt_i(\Red(M^{\bullet})) \cong \bigoplus_{j \leq i} \Red_j(M^{\bullet}).
\]
In particular, there are natural isomorphisms
$N_i \cong \Red(\Gr^\wt_i(M^\bullet)) \cong \Gr^\wt_i(\Red(M^{\bullet})) \cong \Red_i(M^{\bullet})$
in $\mathcal{D}^\ast_{p\text{-comp}}(\mathcal{O}_X)$.
Therefore, the filtration is indexed by $i \in [a,b]$.
\end{theorem}
Note that in the special case when $M^\bullet=\Gr^\bullet E$ where
$E=(E,\Fil^\bullet E)\in \mathrm{Gauge}^\ast(X)$, we get a weight filtration 
$\Fil^\wt_i(\Gr^\bullet E)$ on $\Gr^\bullet E$ with 
$\Gr^\wt_i(\Gr^\bullet E) \simeq \Red_i (E)\otimes_{\mathcal{O}_X} \Gr^\bullet_N \Prism$.

\begin{proof}
The existence part follows exactly from \Cref{prop weight filtration of graded B-cplx},
by getting such a filtration when evaluated on each $S \in X_{\qrsp}$: the outcome automatically
satisfies the axioms above, hence they will also satisfy graded base change property.
This shows the case of $\ast = \emptyset$, the case of $\ast = \perf \text{ and } \vect$
follows from \Cref{graded: perf vs cplt perf} and \Cref{graded: vb vs cplt vb} respectively.
The uniqueness can be seen by exactly following the argument of \Cref{prop weight filtration of graded B-cplx}.
\end{proof}

\begin{remark}
For an integer $a$ and some $b\in \mathbb{Z}_{\geq a}\cup \{\infty\}$, 
 we let $\DG^\ast_{p\text{-comp}}(X, \Gr^\bullet_N \Prism)_{[a,b]}^+$ be the full subcategory of 
 $\DG^\ast_{p\text{-comp}}(X, \Gr^\bullet_N \Prism)$
 whose objects have weights in $[a,b]$.
Then the natural full embedding $\DG^\ast_{p\text{-comp}}(X, \Gr^\bullet_N \Prism)_{[a+1,b]}^+ \to \DG^\ast_{p\text{-comp}}(X, \Gr^\bullet_N \Prism)_{[a,b]}^+$, admits a left adjoint given by
sending an object $M^\bullet$ to the grade complex $M^\bullet/\Fil^{\wt}_aM^\bullet$.
\end{remark}

We have the following concrete understanding of the reduction of graded pieces of $F$-gauges
which come from $I$-torsionfree $F$-crystals as in \Cref{thm crystal to gauge}.

\begin{theorem}
\label{Reduction is the graded of twisted filtration}
Let $X$ be  a smooth formal scheme over $\mathcal{O}_K$,
and let $(\mathcal{E}, \varphi_\mathcal{E}) \in \FC^{I\text{-}\mathrm{tf}}(X)$ having height in $[a, b]$.
\begin{enumerate}
\item Then the gauge $\Pi_X(\mathcal{E})$ constructed in \Cref{thm crystal to gauge} has
the reduction of its graded pieces given by a graded coherent sheaf
having gradings in $[a,b]$.
\item In the setting of \Cref{local situation}, we have the following concrete description of
\[
\Red_U(\Gr^{\bullet}(\Pi_U(\mathcal{E}|_U))) \cong 
\Gr^{\bullet}_{\mathrm{tw}}(\varphi_A^*M) /^L u \otimes_R \mathcal{O}_{\qrsp}.
\]
\end{enumerate}
\end{theorem}

We remind readers that the twisted filtration on $\varphi_A^*M$ was discussed
around \Cref{twist filtration}, and the symbol $u$ stands for the
image of $d \in I$ in $\Gr^{\bullet}(I^{\mathbb{N}}A)$
(so it has degree $1$).

\begin{proof}
Let us show (2) first. Tracing through the proof of \Cref{thm crystal to gauge},
especially \Cref{twist filtration and saturated filtration}
and \Cref{saturated Nyg fil satisfies filtered base change proposition}
(notice that the $S$ there forms a basis of the qrsp site of $U$),
we get the following relation between twisted filtration on $\varphi_A^*M$
and saturated filtration on $\Pi_U(\mathcal{E})$:
\[
\Fil^{\bullet}_{\mathrm{tw}}(\varphi_A^*M) \otimes_{(A, I^{\mathbb{N}})}
\Fil^{\bullet}_N \Prism \cong \Fil^{\bullet}_N \Pi_U(\mathcal{E}).
\]
Passing to graded pieces and applying the reduction functor,
we get:
\[
\Red_U(\Gr^{\bullet}(\Pi_U(\mathcal{E}|_U))) \cong 
\Gr^{\bullet}_{\mathrm{tw}}(\varphi_A^*M) \otimes_{R[u]} \Gr^{\bullet}_N\Prism 
\otimes_{\Gr^{\bullet}_N\Prism} \mathcal{O}_{\qrsp}
\cong \Gr^{\bullet}_{\mathrm{tw}}(\varphi_A^*M) \otimes_{R[u]} R (=R[u]/u)
\otimes_{R} \mathcal{O}_{\qrsp}.
\]

To prove (1), we may work Zariski locally on $X$.
Now the statement follows from
(2) and the explicit knowledge of $\Gr^{\bullet}_{\mathrm{tw}}(\varphi_A^*M)$
as a graded-$R[u]$-module, see \Cref{boundedness results on twist graded}.
\end{proof}

In the remainder of this subsection, we discuss the relation
between weights of an $F$-gauge and heights of its underlying $F$-crystal.
First we observe that the Frobenius twist of an $F$-gauge is eventually $I$-adically filtered.
\begin{proposition}
\label{filtration on Frobenius twist}
Let $S$ be a quasiregular semiperfectoid ring, let $E=(E,\Fil^\bullet E, \widetilde{\varphi}_E)\in \FG^\perf(X)$ be of weight $[a,b]$.
Denote by $\Fil^\bullet (\varphi^*E)$ the filtered base change of $\Fil^\bullet E$ along $\varphi_{\Prism_S}:\Fil^\bullet_N \Prism_S\to I^\bullet \Prism_S$.
\begin{enumerate}[(i)]
	\item The Frobenius structure $\widetilde{\varphi}_E$ induces a filtered map  $v:\Fil^\bullet (\varphi^*E) \to I^\bullet\otimes E$, which is an isomorphism on $\Fil^{\geq b}(-)$.
	In particular, the filtered complex $\Fil^{\geq b} (\varphi^*E)$ is isomorphic to the $I$-adic filtration $I^{\geq b}\otimes E$.
	\item For $i\in \mathbb{Z}$, the weight filtration induces a finite increasing exhaustive filtration of range $[a,b]$ on the map $\Gr^i (v): \Gr^i(\varphi^*E) \to \bar{I}^{i}E$, such that its $j$-th graded piece of the map is 
 \[
 \begin{cases}
     \Red_j(E)\otimes_S \bar{I}^{i-j} \overline{\Prism}_S \xrightarrow{\simeq} \Red_j(E)\otimes_S \bar{I}^{i-j} \overline{\Prism}_S,~i\geq j;\\
     0 \longrightarrow  \Red_j(E)\otimes_S \bar{I}^{i-j} \overline{\Prism}_S,~i<j.
 \end{cases}
 \]
\end{enumerate}
\end{proposition}
\begin{proof}
	We first notice that as $\Fil^\bullet E$ is a filtered perfect complex over $\Fil^\bullet_N \Prism_S$, its filtered base change
	\[
    \Fil^\bullet E \underset{\Fil^\bullet_N \Prism_S, \varphi_{\Prism_S}}{\bigotimes} I^\bullet \Prism_S
	\]
	is also a filtered perfect complex over $I^\bullet \Prism_S$.
	Moreover, as the ring $\Prism_S$ is  $(p,I)$-complete, the filtered ring $I^\bullet \Prism_S$ is thus filtered complete. 
	In particular, by \Cref{perfect over complete is complete} and the perfectness, the filtered complex $\Fil^\bullet (\varphi^* E)$ and hence $\Fil^{\geq b} (\varphi^* E)$ is also filtered complete and $(p,I)$-complete automatically.
	Notice that by the perfectness again, the $\Prism_S$-complex $E$ is $I$-adic complete.
	So to show the filtered map $\Fil^{\geq b}(v)$ is a filtered isomorphism, it suffices to check the associated map of the graded pieces.
	
	We then note by construction that the map $v$ can be factored as the composition 
	\[
	\begin{tikzcd}
		\Fil^\bullet E \underset{\Fil^\bullet_N \Prism_S, \varphi_{\Prism_S}}{\widehat{\bigotimes}} I^\bullet \Prism_S \arrow[r,"u"]& \Fil^\bullet E \underset{\Fil^\bullet_N \Prism_S, \varphi_{\Prism_S}}{\widehat{\bigotimes}} I^\mathbb{Z} \Prism_S \arrow[r, "\varphi_E", "\simeq"'] & I^\mathbb{Z} E,
	\end{tikzcd}
	\]
	where the map $\varphi_E$ is a filtered isomorphism, 
	and to prove (i) it suffices to show that $\Gr^{\geq b} (u)$ is a graded isomorphism.
	On the other hand, the weight filtration of $\Gr^\bullet E$ induces a finite increasing exhaustive filtration on the graded map $\Gr^\bullet (u)$ above, whose $j$-th associated graded is
	\[
	V_j \otimes_S \bar{I}^\mathbb{N} \overline{\Prism}_S  \longrightarrow V_j \otimes_S \bar{I}^\mathbb{Z} \overline{\Prism}_S, \tag{$\ast$}
	\]
	where $V_j\simeq \Red_j(E)$ is a perfect $S$-complex of graded degrees $j\in [a,b]$.
	This finishes the proof of (i), since $j$ is assumed to be $\leq b$.	
	By applying $\Gr^i(-)$ at the graded map $(\ast)$, we further obtain (ii).
\end{proof}
\begin{remark}
	It is clear from the proof that \Cref{filtration on Frobenius twist} can be extended to more general $F$-gauges $E=(E,\Fil^\bullet E, \widetilde{\varphi}_E)$:
we are only using their weights are in $[a,b]$ and $\Fil^\bullet E$ is filtered complete.
\end{remark}
In the following, we let $\FG^{I\text{-}\mathrm{tf}}(X)$ be the category of coherent $F$-gauges with
$I$-torsionfree underlying $F$-crystals.
\begin{corollary}
\label{weight 1 always saturated}
Assume $X$ is smooth over $\mathcal{O}_K$, 
and $E=(E,\Fil^\bullet E, \widetilde{\varphi}_E)\in \FG^{I\text{-}\mathrm{tf}}(X)$ is of height $[a,a+1]$.
Then $E = \Pi_X(\mathcal{E})$ where $\mathcal{E}$ is the underlying $F$-crystal of $E$
and $\Pi_X$ is the functor obtained in \Cref{thm crystal to gauge}. 
In particular, $E$ is saturated.
\end{corollary}
\begin{proof}
We adopt the notation in \Cref{local situation}.
In particular, we let $U = \Spf(R)$ be an affine open of $X$,
and let $(A, I)$ be a Breuil--Kisin prism of $U$ with perfection $(A^{(0)}, I A^{(0)})$ 
and $S \coloneqq \overline{A^{(0)}}$.

Next we show that the value of $E$ at $\Spf(S)$ is equipped with the saturated
Nygaardian filtration.
As $E$ is of height $\geq a$, we have $\Fil^a E(S) = E(S)$.
Moreover, by \Cref{filtration on Frobenius twist}.(i), 
$\Fil^{\geq a+1} (\varphi_{\Prism_S}^* E(S))$ is isomorphic to the $I$-adic filtration on $I^{a+1}E(S)$
and thus equals to $\varphi_E^{-1}(I^{\geq a+1} E(S))$.
Notice that since $\varphi_{\Prism_S}:\Fil^\bullet_N \Prism_S \to I^\bullet \Prism_S$ is a filtered isomorphism,
we can untwist $\varphi_{\Prism_S}$ to get $\Fil^{\geq a+1} E(S) \simeq 
\widetilde{\varphi}_E^{-1} (I^{\geq a+1} E(S))$.
Hence the value $E(S)$ is saturated.

Now the same proof of \Cref{thm crystal to gauge} (1) shows that the value of
$E$ at $\Spf(\overline{A^{(i)}/I})$ are equipped with saturated Nygaardian filtrations
for all $i$.
Lastly the same proof of \Cref{saturated claim} shows that for any $\Spf(T) \in X_{\qrsp}$
with $p$-completely flat structural map to $U$, the value of $E$ at $\Spf(T)$
has saturated Nygaardian filtration, this finishes the proof.
\end{proof}

We also give a relation on the weight of an $F$-gauge and the height of the underlying $F$-crystal.
\begin{proposition}
\label{weight and height}
Let $X$ be a quasi-syntomic $p$-adic formal scheme, and let $E=(E,\Fil^\bullet E, \widetilde{\varphi}_E)\in \FG^{I\text{-}\mathrm{tf}}(X)$.
\begin{enumerate}[(i)]
	\item 	If $E$ is of weight $[a,b]$, then its underlying $F$-crystal $(\mathcal{E},\varphi_\mathcal{E})$ has height within $[a,b]$ as well.
	\item  Conversely, assume $(\mathcal{E},\varphi_\mathcal{E})$ is of height $[a',b']$ with $E$ being saturated, and there is a quasi-syntomic cover of $X$ by perfectoids.
	Then $E$ is of weight $[a',b']$.
\end{enumerate}		
\end{proposition}

Note that $\Fil^\bullet E$ is saturated when $E$ is an $F$-gauge in vector bundle (\Cref{FGVB is saturated}).
As a consequence, in this case the notion of weight of $E$ and the height of the underlying $F$-crystal coincide.
\begin{proof}
    	As both weight and height can be checked locally with respect to quasi-syntomic topology, it suffices to assume $X=\Spf(S)$ is quasiregular semiperfectoid.
    	We first assume that $E$ is of weight $[a,b]$.
    	By construction (cf. \Cref{thm weight filtration of gauge}), we have $\Fil^i E = E$ for $i\leq a$.
    	In particular, since $\Fil^i E\subseteq \widetilde{\varphi}_E^{-1}(I^iE)$, the image $\widetilde{\varphi}_E(E)$ and hence the image of the linearlized map $\varphi_E(E)$ are contained in $I^a E$;
    	namely the $F$-crystal $(\mathcal{E},\varphi_\mathcal{E})$ has height $\geq a$.
    	Moreover, by \Cref{{filtration on Frobenius twist}}, since the filtered map $v$ is an isomorphism on $\Fil^{\geq b}(-)$, by forgetting the filtration and looking at the underlying map of modules, we get the image of the linearized map $\varphi_E(E)$ contains the submodule $I^bE$, and thus the associated $F$-crystal is of height $\leq b$.
        Here we also note that if $\Red_b(E)$ is nonzero, then by looking at the $b$-th graded factor of the map $(\ast)$, for any $j<b$ we have
        \[
        \Gr^b(V_j\otimes_S \bar{I}^\mathbb{N} \overline{\Prism}_S )=0 \longrightarrow \Gr^b(V_j \otimes_S \bar{I}^\mathbb{Z} \overline{\Prism}_S) = V_j\otimes_S \overline{\Prism}_S \neq 0.
        \]
        In particular, for any given integer $l<b$, the map $\Gr^l u$ and hence $\Fil^l u$ is not surjective and hence not an isomorphism.

        For (ii), we assume  the underlying $F$-crystal of $E$ is of height $[a',b']$, and $S$ is perfectoid.
        Then the image $\varphi_\mathcal{E}(\mathcal{E})$ satisfies the inclusions 
        \[
        I^{b'}\mathcal{E} \subseteq \varphi_\mathcal{E}(\mathcal{E}) \subseteq I^{a'} \mathcal{E}.
        \]
        So by taking the preimage along the $\Prism_S$-linear map $\varphi_\mathcal{E}$, we see $\varphi^{-1}_\mathcal{E}(I^{a'} \mathcal{E}) = \varphi_{\Prism_S}^* \mathcal{E}$, and the map $\varphi_\mathcal{E}$ induces an isomorphism $\varphi^{-1}_\mathcal{E}(I^{\geq b'} \mathcal{E}) \simeq I^{\geq b'} \mathcal{E}$.
        On the other hand, since $\varphi_{\Prism_S}$ is a filtered isomorphism, by the assumption that $E$ has saturated filtration, we get
        \[
        \Fil^\bullet (\varphi_{\Prism_S}^* E) = \varphi^{-1}_\mathcal{E}(I^\bullet \mathcal{E}).
        \]
        This in particular implies that $\Fil^{a'} E=E$, and $E$ is of weight $\geq a'$.
        Moreover, by \Cref{filtration on Frobenius twist}.(ii), the map  $\varphi^{-1}_E(I^i E) \to I^i E$ is not an isomorphism unless $i\geq \max\{l~|~\Red_l(E)\neq 0\}$, where the number $\max\{l~|~\Red_l(E)\neq 0\}$ is by definition the largest possible weight of $E$ (\Cref{def of wt of gauge}).
        Thus $E$ is of weight $\leq b'$.
\end{proof}

\subsection{$F$-gauges in vector bundles and $p$-divisible groups}

Recall that in \cite{ALB23} the authors established,
for every quasi-syntomic $p$-adic formal scheme $X$,
an anti-equivalence between $p$-divisible groups over $X$ and 
admissible prismatic Dieudonn\'{e} crystals over $X$ (see \cite[Theorem 1.4.4]{ALB23}).
Here a prismatic Dieudonn\'{e} crystal over $X$ is the same as an $F$-crystal in vector bundles
over $\Spf(R)_{\Prism}$ which is effective of height $\leq 1$,
and admissibility is a technical condition introduced in \cite[Def.~4.5]{ALB23}, 
and is automatic if $X$ is regular and admits a quasi-syntomic cover by perfectoids,
(see \cite[Definition 1.3.1 and Remark 1.4.9]{ALB23}\footnote{Although the authors only
		stated that admissibility is automatic for complete regular local rings with a perfect characteristic $p$
		residue field, their proof works for
		any $p$-complete ring $R$ admitting a quasi-syntomic cover by a perfectoid.}).
Below, for  quasi-syntomic $X$, we let $\FC^{\mathrm{adm},\vect}_{[0,1]}(X)$ be the subcategory of $\FC^\vect_{[0,1]}(X)$ that consists of admissible prismatic $F$-crystals of height $[0,1]$.
We then establish an equivalence of the above
category with yet another category.
	
As a preparation, we show that the filtration on an $F$-gauge in vector bundles is always saturated.
\begin{proposition}\label{FGVB is saturated}
Let $S$ be a quasiregular semiperfectoid ring, and let $E=(E,\Fil^\bullet E, \widetilde{\varphi}_E)$ be an $F$-gauge in vector bundles over $S$.
Then $E$ is saturated, or more explicitly we have $\Fil^i E = \widetilde{\varphi}_E^{-1}(I^i E)$ for $i\in \mathbb{Z}$.
\end{proposition}
\begin{corollary}
    	Let $X$ be a quasi-syntomic $p$-adic formal scheme.
    	Then the forgetful functor below from $F$-gauges in vector bundles to its underlying $F$-crystals is fully faithful
    	\[
    	\FG^\vect(X) \longrightarrow \FC^\vect(X).
    	\]
\end{corollary}
\begin{proof}
    	By assumption, as $\Fil^\bullet E$ is a filtered vector bundle over the filtered ordinary ring $\Fil^\bullet_N \Prism_S$, each $\Fil^i E$ is a submodule inside $E$ such that $\widetilde{\varphi}_E(\Fil^i E) \subset I^iE$.
    	Thus $\Fil^i E \subseteq \widetilde{\varphi}_E^{-1}(I^i E)$, and we want to show that each inclusion is an equality.
    	Here we note that by the assumption of $\Fil^\bullet E$, we know the inclusion is an equality when $i<<0$.

    	For each $i\in \mathbb{Z}$, the Frobenius structure $\widetilde{\varphi}_E$ induces the following commutative diagram of short exact sequences
    	\[
    	\begin{tikzcd}
    		0 \ar[r] & \Fil^{i+1} E \arrow[d, "\widetilde{\varphi}_E^{i+1}"] \ar[r] & \Fil^i E \arrow[d, "\widetilde{\varphi}_E^i"] \ar[r] & \Gr^i E \arrow[d, "\Gr^i \widetilde{\varphi}"] \ar[r] & 0\\
    		0 \ar[r] & I^{i+1} E \ar[r] & I^i E \ar[r] & I^iE/I^{i+1}E \ar[r] &0,
    	\end{tikzcd}
        \]
        where $\widetilde{\varphi}_E^i$ is the restriction of $\widetilde{\varphi}_E$ on $\Fil^i E$.
        We then claim that the map $\Gr^i E \to I^iE/I^{i+1}E$ is injective for each $i\in \mathbb{Z}$.
        Granting the claim, a simple diagram chasing implies that $\Fil^{i+1} E= (\widetilde{\varphi}_E^i)^{-1}(I^{i+1} E) = \widetilde{\varphi}_E^{-1} ( I^{i+1} E)$.
        Hence we get the saturatedness of $\Fil^\bullet E$ by induction and the observation that $\Fil^i E=E$ for $i<<0$.
        
        Now we check the above claim.
        Recall that the graded map $\Gr^\bullet \widetilde{\varphi}:\Gr^\bullet E \to \bar{I}^\bullet E$ factors through the linearlization 
        \[
        \begin{tikzcd}
        	\Gr^\bullet E \ar[r] & \Gr^\bullet E \underset{\Gr^\bullet_N \Prism_S, \Gr^\bullet \varphi_{\Prism_S}}{\widehat{\bigotimes}} \bar{I}^\mathbb{N} \overline{\Prism}_S \arrow[rr, "\Gr^\bullet \varphi_E"] && \bar{I}^\bullet E,
        \end{tikzcd}
        \]
        such that a further base change to $\bar{I}^\mathbb{Z} \overline{\Prism}_S$ is a graded isomorphism (cf. \Cref{def F-gauge qrsp}).
        In particular, to show the injection of $\Gr^i E \to \bar{I}^i E$, it suffices to show that the following map is a graded injection
        \[
        \Gr^\bullet E \longrightarrow  \Gr^\bullet E\underset{\Gr^\bullet_N \Prism_S, \Gr^\bullet \varphi_{\Prism_S}}{\widehat{\bigotimes}} \bar{I}^\mathbb{Z} \overline{\Prism}_S.
        \]
        We then notice that since $\Gr^\bullet E$ is a graded vector bundle over $\Gr^\bullet_N \Prism_S$, its reduction $\Red_S (E)$, which is the graded base change of $\Gr^\bullet E$ from $\Gr^\bullet_N \Prism_S$ to $S$, is a graded vector bundle over $S$.
        Moreover, by \Cref{thm weight filtration of gauge}, the complex $\Gr^\bullet E$ is filtered by a finite exhaustive weight filtration with graded pieces being $\Red_{j}(E) \otimes_S \Gr^\bullet_N \Prism_S$.
        Thus by a finite induction, it suffices to assume $\Gr^\bullet E$ is of the form $V\otimes_S \Gr^\bullet_N \Prism_S$ for a finite projective $S$-module $V$.
        
        Finally, as $S$ is quasiregular semiperfectoid, by \cite[Thm.~12.2]{BS19} the graded map $\Gr^\bullet \varphi_{\Prism_S}: \Gr^\bullet_N \Prism_S \to \bar{I}^\mathbb{N} \overline{\Prism}_S$ and thus $\Gr^\bullet_N \Prism_S \to \bar{I}^\mathbb{Z} \overline{\Prism}_S$ is injective.
        So by the flatness of $V$ over $S$, we get the injection of 
        \[
        V\otimes_S \Gr^i_N \Prism_S \longrightarrow \Gr^i\left( (V\otimes_S \Gr^\bullet_N \Prism_S) \underset{\Gr^\bullet_N \Prism_S, \Gr^\bullet \varphi_{\Prism_S}}{\widehat{\bigotimes}} \bar{I}^\mathbb{Z} \overline{\Prism}_S \right) \simeq V\otimes_S \bar{I}^i \overline{\Prism}_S, 
        \]
        hence the claim.
\end{proof}
\begin{theorem}
\label{p-div and Fgauge}
Let $X$	be a quasi-syntomic $p$-adic formal scheme.
There is a natural equivalence
\[
\FG^\vect_{[0,1]}(X) \simeq \FC^{\mathrm{adm},\vect}_{[0,1]}(X).
\]
It is compatible with \Cref{thm crystal to gauge} when $X$ is smooth over  $\Spf(\mathcal{O}_K)$.
\end{theorem}
The subscript $[0,1]$ denotes objects which are effective of height $\leq 1$.
\begin{proof}
		Let $(\mathcal{E},\varphi_\mathcal{E})$ be a given object in $\FC^{\mathrm{adm},\vect}_{[0,1]}(X)$.
		We use $(M_S, \Fil^\bullet M_S, \widetilde{\varphi}_{M_S})$ to denote the value of $\Pi_{\Spf(S)}(\mathcal{E}|_{\Spf(S)_\qrsp})$ at $S\in X_\qrsp$, which by \Cref{thm crystal to gauge} (1) satisfies
		\[
		\Fil^n M_S = \widetilde{\varphi}_{M_S}^{-1} (I^nM_S) \subset M_S.
		\]
		We then claim that 
		\begin{enumerate}[(1)]
			\item\label{claim p div 1} the filtered module $\Fil^\bullet M_S$ is a filtered vector bundle over $\Fil^\bullet_N \Prism_S$, and
			\item\label{claim p div 2} for a map of rings $S_1\to S_2$ in $X_\qrsp$, the natural filtered linearlization map below is a filtered isomorphism
			\[
			\Fil^\bullet M_1 \otimes_{\Fil^\bullet_N \Prism_{S_1}} \Fil^\bullet_N \Prism_{S_2} \longrightarrow \Fil^\bullet M_2.
			\]
		\end{enumerate}
	    Here we use $M_i$ to simplify the notation $M_{S_i}$.
	    Granting the above two claims, we see every admissible prismatic $F$-crystal in vector bundles of height $[0,1]$ naturally underlies an $F$-gauge in vector bundles of weight $[0,1]$.
	    Conversely, given an $F$-gauge in vector bundles of weight $[0,1]$, its underlying $F$-crystal is an $F$-crystal in vector bundle of height $[0,1]$ by \Cref{weight and height}.(i).
	    
	Since both statements can be checked locally on $S$, we are allowed to localize $S$.
 By doing so, we may first assume that the ideal $I$ in $\Prism_S$ is generated by $d$.
	    By \cite[Prop.~4.29, Lem.~4.23, Prop.~4.22]{ALB23} (applied at $(A,\Fil A,\varphi,\varphi_1)=(\Prism_S,\Fil^1_N \Prism_S,\varphi_{\Prism_S}, \frac{1}{d}\varphi_{\Prism_S})$, which satisfies the assumption of \cite[Prop.~4.22]{ALB23} by loc.~cit. Lem~4.27, Lem.~4.28), there are finite projective $\Prism_S$-modules $L_0$ and $L_1$, such that
	    \[
	    M_S \simeq L_0\oplus L_1, \quad \Fil^1 M_S \simeq \Fil^1_N \Prism_S \otimes_{\Prism_S} L_0 \oplus L_1.
	    \]
	    Moreover, as in loc.~cit., if we localize on $S$ so that $L_0$ and $L_1$ are finite free
     and choose a basis $e_i$ (resp.~$f_j$) of $L_0$ (resp.~$L_1$), then we know that 
     $\{\varphi(e_i), \varphi_1(f_j)\}$ is a basis of $M$.
        This implies that as a filtered module over $\Fil^\bullet_N \Prism_S$, we have
        \[
        \Fil^\bullet M = \Fil^\bullet_N \Prism_S \otimes_{\Prism_S} L_0 \oplus \Fil^{\bullet - 1}_N \Prism_S \otimes_{\Prism_S} L_1.
        \]
        
        For \ref{claim p div 2}, we apply \ref{claim p div 1} at $S=S_1$, to get the aforementioned decomposition.
        By the same reasoning as before, we see $\Fil^\bullet M_2 \simeq \Fil^\bullet_N \Prism_{S_2} \otimes_{\Prism_{S_1}} L_0 \oplus \Fil^{\bullet -1}_N \Prism_{S_2} \otimes_{\Prism_{S_1}} L_1$,
        and the natural filtered linearization $\Fil^\bullet_N \Prism_{S_2} \otimes_{\Fil^\bullet_N \Prism_{S_1}} \Fil^\bullet M_1 \to \Fil^\bullet M_2$ is a filtered isomorphism.     
        
        Finally we show the essential surjectivity.
        By \Cref{FGVB is saturated}, an $F$-gauge in vector bundle has the saturated Nygaardian filtration and is uniquely determined by the underlying $F$-crystal.
        So it is left to check that the underlying $F$-crystal is admissible in the sense of \cite[Def.~4.5, Def.~4.10]{ALB23}.
        Namely for $S\in X_\qrsp$, we want to show that the image of $M_S \xrightarrow{\widetilde{\varphi}_{M_S}} M_S \to M_S/IM_S$ is a vector bundle $F_S$ over $S$, such that the linearization $ F_S \otimes_S \overline{\Prism}_S \to M_S/IM_S$ is injective.
        By \Cref{FGVB is saturated}, $\Fil^1 M_S$ is equal to $\widetilde{\varphi}^{-1}(IM_S)$, and thus the image $F_S$ is equal to $M_S/\Fil^1 M_S=\Gr^0 M_S$.
        The zero-th graded piece $\Gr^0 M_S=F_S$ is a vector bundle over $\Gr^0_N \Prism_S=S$, thanks to the assumption that $\Fil^\bullet M_S$ is a filtered vector bundle over $\Fil^\bullet_N \Prism_S$.
        Moreover, by taking the zero-th graded piece of the filtered isomorphism $\varphi_{M_S}:\Fil^\bullet M_S \otimes_{\Fil^\bullet_N \Prism_S, \varphi_{\Prism_S}} I^\mathbb{Z} \Prism_S \simeq I^\mathbb{Z} M_S$, the required injectivity of the linearlization above is equivalent to the injectivity of the map 
        \[
        \Gr^0 M_S \otimes_S \overline{\Prism}_S \longrightarrow \Gr^0 \left(\Fil^\bullet M_S \otimes_{\Fil^\bullet_N \Prism_S, \varphi_{\Prism_S}} I^\mathbb{Z} \Prism_S \right).
        \]
        Here by an induction on the weight filtration, we may assume $\Gr^\bullet M_S =V \otimes_S \Gr^\bullet_N \Prism_S$ for a finite projective $S$-module $V$ of graded degree $0$ or $1$.
        In both cases, the injection follows from that of the structure sheaf, which finishes the proof.
\end{proof}

\section{Relative prismatic $F$-gauges and realizations}
We now turn our focus to a relative version of prismatic $F$-gauges, 
and its relation with the absolute version studied in the previous section.
Throughout this section, let us fix $(A, I)$ a bounded prism
with $\overline{A} \coloneqq A/I$.

\subsection{Quasi-syntomic sheaves in relative setting}
Let $Z$ be a quasi-syntomic $\overline{A}$-formal scheme.
Consider $(Z/\overline{A})_{\qrsp} \subset Z_{\qsyn}$,
where $(Z/\overline{A}){\qrsp}$ denote the category of large quasi-syntomic 
$\overline{A}$-algebras (in the sense of \cite[Definition 15.1]{BS19})
together with a quasi-syntomic map to $Z$ equipped with quasi-syntomic topology.
Then we have the following analog of the ``\emph{unfolding process}''
in \cite[Proposition 4.31]{BMS19}:

\begin{proposition}
\label{unfolding process in relative setting}
Restriction to qrsp objects induces an equivalence
\[
\mathrm{Shv}_{\mathcal{C}}(Z_{\qsyn}) \cong 
\mathrm{Shv}_{\mathcal{C}}((Z/\overline{A})_{\qrsp})
\]
for any presentable $\infty$-category $\mathcal{C}$.
\end{proposition}

\begin{proof}
The proof is identical to that in loc.~cit.
\end{proof}

\begin{remark}
\label{how to compute unfolding remark}
Concretely, suppose we are given a quasi-syntomic sheaf $\mathcal{F}$ on 
$(Z/\overline{A})_{\qrsp}$. Let $R$ be an object in $Z_{\qsyn}$,
and suppose we are given a large quasi-syntomic cover $R \to S^{(0)}$ with
its Cech nerve being $S^{(\bullet)}$,
then the unfolding evaluated at $R$ can be computed as the cosimplicial limit
of $\mathcal{F}(S^{(\bullet)})$.
\end{remark}

\begin{example}[see {\cite[\S 1.3 and \S 5.2]{BMS19}} and {\cite[\S 15]{BS19}}]
Here are some examples of quasi-syntomic sheaves on $\Spf(\overline{A})_{\qsyn}$:
the relative prismatic cohomology $\Prism_{-/A}$ and its $\varphi_A^*$-variant
$\Prism^{(1)}_{-/A} \coloneqq \Prism_{-/A} \widehat{\otimes}_{A, \varphi_A} A$,
the $i$-th Nygaard filtration on $\Prism^{(1)}_{-/A}$ for any $i$,
the relative Hodge--Tate cohomology $\overline{\Prism}_{-/A} \coloneqq \Prism_{-/A}/I$
and its $i$-th conjugate filtration for any $i$,
the $p$-adic derived de Rham cohomology $\mathrm{dR}_{-/\overline{A}}$ 
and its $i$-th Hodge filtration for any $i$.
By restriction, we get quasi-syntomic sheaves on $Z_{\qsyn}$.
\end{example}

The following relation between the various filtrations appearing above
will often be used (see \cite[Theorem 15.2 and 15.3]{BS19}):
we have fiber sequences
\[
\Fil^{i+1}_N \Prism^{(1)}_{-/A} \to \Fil^{i}_N \Prism^{(1)}_{-/A} \to 
\Gr^{i}_N \Prism^{(1)}_{-/A} \cong \Fil_i^{\conj} \overline{\Prism}_{-/A} \{i\}; \text{ and }
\]
\[
\Fil^{i-1}_N \Prism^{(1)}_{-/A} \otimes_A I \to \Fil^{i}_N \Prism^{(1)}_{-/A}
\to \Fil^i_H \mathrm{dR}_{-/\overline{A}}.
\]
In terms of Rees's construction, we can formulate them uniformly:

\begin{proposition}
\label{structure of Nygaard filtration on twisted-prismatic cohomology}
For simplicity, let us assume $(A, I = (d))$ is a bounded oriented prism.
Now we view $\Rees(\Fil^{\bullet}_N \Prism^{(1)}_{-/A}) \coloneqq 
\bigoplus_{i \in \mathbb{Z}} \Fil^i_N \Prism^{(1)}_{-/A} \cdot t^{-i}$
as a graded $\Rees(I^{\mathbb{N}}A) \coloneqq A[t, \frac{d}{t}]$-algebra.
Then we have
\[
\Rees(\Fil^{\bullet}_N \Prism^{(1)}_{-/A}) /^L (t) \cong 
\bigoplus_{i \in \mathbb{N}} \Fil_i^{\conj} \overline{\Prism}_{-/A} \cdot (\frac{d}{t})^i; \text{ and }
\]
\[
\Rees(\Fil^{\bullet}_N \Prism^{(1)}_{-/A}) /^L (\frac{d}{t}) \cong
\Rees(\Fil^{\bullet}_H \mathrm{dR}_{-/\overline{A}}) \coloneqq 
\bigoplus_{i \in \mathbb{Z}} \Fil^i_H \mathrm{dR}_{-/\overline{A}} \cdot t^{-i}.
\]
\end{proposition}

The following flatness result is the analog of (proof of) \Cref{F-gauge sheaf}.

\begin{proposition}
\label{flatness of Nygaard filtered ring map}
Let $S_1 \to S_2$ be a quasi-syntomic flat (resp.~faithfully flat) map
of two quasi-syntomic $\overline{A}$-algebras and let $S_1 \to S_3$ be an arbitrary map
of quasi-syntomic $\overline{A}$-algebras with $S_4 \coloneqq S_2 \widehat{\otimes}_{S_1} S_3$, then:
\begin{enumerate}
\item The map $\mathrm{Hdg}(S_1/\overline{A}) \to \mathrm{Hdg}(S_2/\overline{A})$
is $p$-completely flat (resp.~faithfully flat), and the graded map
\[
\mathrm{Hdg}(S_2/\overline{A}) \widehat{\otimes}_{\mathrm{Hdg}(S_1/\overline{A})} \mathrm{Hdg}(S_3/\overline{A})
\to \mathrm{Hdg}(S_4/\overline{A})
\]
is a graded isomorphism;
\item The filtered map $\Fil^{\conj}_{\bullet}(\overline{\Prism}_{S_1/A})
\to \Fil^{\conj}_{\bullet}(\overline{\Prism}_{S_2/A})$ has its corresponding Rees's
construction map being $p$-completely flat (resp.~faithfully flat), and the filtered map
\[
\Fil^{\conj}_{\bullet}(\overline{\Prism}_{S_2/\overline{A}}) 
\widehat{\otimes}_{\Fil^{\conj}_{\bullet}(\overline{\Prism}_{S_1/\overline{A}})} 
\Fil^{\conj}_{\bullet}(\overline{\Prism}_{S_3/\overline{A}})
\to \Fil^{\conj}_{\bullet}(\overline{\Prism}_{S_4/\overline{A}})
\]
is a filtered isomorphism;
\item The filtered map $\Fil^{\bullet}_H(\mathrm{dR}_{S_1/\overline{A}})
\to \Fil^{\bullet}_H(\mathrm{dR}_{S_2/\overline{A}})$ has its corresponding Rees's
construction map being $p$-completely flat (resp.~faithfully flat), and the filtered map
\[
\Fil^{\bullet}_H(\mathrm{dR}_{S_2/\overline{A}}) 
\widehat{\otimes}_{\Fil^{\bullet}_H(\mathrm{dR}_{S_1/\overline{A}})} 
\Fil^{\bullet}_H(\mathrm{dR}_{S_3/\overline{A}})
\to \Fil^{\bullet}_H(\mathrm{dR}_{S_4/\overline{A}})
\]
is a filtered isomorphism;
\item The filtered map $\Fil^{\bullet}_N(\Prism^{(1)}_{S_1/A})
\to \Fil^{\bullet}_N(\Prism^{(1)}_{S_2/A})$ has its corresponding Rees's
construction map being $(p, I)$-completely flat (resp.~faithfully flat), and the filtered map
\[
\Fil^{\bullet}_N(\Prism^{(1)}_{S_2/A}) 
\widehat{\otimes}_{\Fil^{\bullet}_N(\Prism^{(1)}_{S_1/A})} 
\Fil^{\bullet}_N(\Prism^{(1)}_{S_3/A})
\to \Fil^{\bullet}_N(\Prism^{(1)}_{S_4/A})
\]
is a filtered isomorphism.
\end{enumerate}
\end{proposition}

\begin{proof}
The proof follows from similar argument appearing in the proof of
\Cref{F-gauge sheaf}: the proof of (1) has appeared there.
Then similar to there, with \Cref{structure of Nygaard filtration on twisted-prismatic cohomology}
in mind, we can use (1) together with 
\Cref{useful flatness criterion 1} and \Cref{useful flatness criterion 2} to prove (2), and use
(2) and \Cref{useful flatness criterion 1} to prove both of (3) and (4).
\end{proof}

\subsection{Relative prismatic $F$-gauges}
\label{Relative prismatic $F$-gauges}
Let $(A,I)$ be a fixed bounded prism.
For a quasi-syntomic $\overline{A}$-formal scheme $Z$, we denote by $(Z/\overline{A})_\qrsp$ the category of large quasi-syntomic 
$\overline{A}$-algebra over $Z$ together with quasi-syntomic topology, in the sense of \cite[Definition 15.1]{BS19}.
\begin{definition}[{c.f.~\cite[Definition 2.14]{Tang22}}]
\label{def rel F-gauge qrsp}
Let $S$ be a large quasi-syntomic $\overline{A}$-algebra.
A \emph{prismatic gauge} $E=(E, \Fil^\bullet E^{(1)})$
over $\Spf(S)/A$ consists of the following data:
\begin{itemize}
\item a derived $(p,I)$-complete complex $E$ over $\Prism_{S/A}$;
\item a derived $(p,I)$-complete filtration $\Fil^\bullet E^{(1)}$ on $E^{(1)} \coloneqq \varphi_A^* E$,
linear over the Nygaard filtered relative prismatic cohomology
$\Fil_N^\bullet \Prism_{S/A}^{(1)}$.
\end{itemize}
A \emph{prismatic $F$-gauge} $E=(E, \Fil^\bullet E^{(1)}, \widetilde{\varphi}_{E})$
over $\Spf(S)/A$ consists of the following data:
\begin{itemize}
\item a prismatic gauge $E=(E, \Fil^\bullet E^{(1)})$ over $\Prism_{S/A}$
called the underlying gauge of $E$;
\item a map of filtered complexes
\[
\widetilde{\varphi}_{E} \colon \Fil^\bullet E^{(1)} 
\longrightarrow I^\mathbb{Z} E = I^\mathbb{Z}\Prism_{S/A} \widehat{\otimes}_{\Prism_{S/A}} E,
\]
such that the filtered linearization 
\[
\varphi_{E} \colon \Fil^\bullet E^{(1)} \widehat{\otimes}_{\Fil_N^\bullet \Prism_{S/A}^{(1)}, \varphi_{\Prism_{S}}} 
I^\mathbb{Z} \Prism_S \longrightarrow I^\mathbb{Z} E
\]
is a filtered isomorphism in $\DF_{(p,I)\text{-comp}}(I^\mathbb{Z}\Prism_{S/A})\simeq \mathcal{D}_{(p,I)\text{-comp}}(\Prism_{S/A})$.
\end{itemize}
We use $\fG(\Spf(S)/A)$ to denote the natural $\infty$-category whose objects are as above.
\end{definition}

\begin{remark}[{c.f.~\cite[Proposition 2.1 and Remark 2.5]{Tang22}}]
\label{equivalent def rel F-gauge qrsp}
Here we notice that by \cite[Thm.\ 15.2]{BS19}, the filtered morphism of $A$-algebras 
\[
\varphi_{\Prism_{S}} \colon \Fil^\bullet_N \Prism_{S/A}^{(1)} \to I^\bullet \Prism_{S/A}
\]
identifies the target as $(p, I)$-completely inverting $I \subset \Fil^1_N$ of the source
as filtered $A$-algebras.
In particular, 
the above linarization condition can be rewritten as a filtered isomorphism
\[
\left( \Fil^\bullet E^{(1)}[1/I] \right)^\wedge_{(p,I)} \xlongrightarrow[\cong]{\varphi_{E}} I^\mathbb{Z} E,
\]
or equivalently 
\[
\left( \colim_{i\in \mathbb{Z}} (\cdots \to \Fil^i E^{(1)} \otimes I^{-i} \to \Fil^{i+1} E^{(1)} \otimes I^{-i-1}\to \cdots) \right)^\wedge_{(p,I)} \xrightarrow[\cong]{\varphi_E} E.
\]
\end{remark}

\begin{definition}
\label{relative F-gauge subcategory definition}
As in \Cref{def F-gauge category}, we say a prismatic $F$-gauge
$E=(E, \Fil^\bullet E^{(1)}, \widetilde{\varphi}_{E})$
over $\Spf(S)/A$ is:
\begin{enumerate}
\item \emph{finite projective} if $E$ is $(p, I)$-completely finite projective over $\Prism_{S/A}$
and $\Fil^\bullet E^{(1)}$ is filtered $(p, I)$-completely finite projective
over $\Fil^{\bullet}_N \Prism_{S/A}^{(1)}$;
\item \emph{perfect} if $E$ is $(p, I)$-completely perfect over $\Prism_{S/A}$
and $\Fil^\bullet E^{(1)}$ is filtered $(p, I)$-completely perfect
over $\Fil^{\bullet}_N \Prism_{S/A}^{(1)}$;
\item \emph{coherent} if it is perfect with all of 
$E$, $\Fil^{\bullet} E^{(1)}$ and $\Gr^{\bullet} E^{(1)}$
concentrated in cohomological degree $0$.
\end{enumerate}
We use $\fG^{\ast}(\Spf(S)/A)$ to denote the natural $\infty$-category whose objects are as above,
where $*\in\{\vect, \perf, \coh\}$ respectively.
\end{definition}

The analog of \Cref{Gauge: cplt vs alg conditions} still holds true:
\begin{proposition}
\label{relative Gauge: cplt vs alg conditions}
Let $S$ be a large quasi-syntomic $\overline{A}$-algebra and let 
$(E, \Fil^\bullet E^{(1)}) \in \mathrm{Gauge}(\Spf(S)/A)$. Then, when making the
\Cref{relative F-gauge subcategory definition}.(1)-(2),
we may drop $(p, I)$-completely from the condition without changing the definition.
\end{proposition}

\begin{proof}
The part concerning $E$ follows from \Cref{vbperf vs cplt vbperf}.
The part concerning $\Fil^\bullet E^{(1)}$ follows from the same proof of
\Cref{Gauge: cplt vs alg conditions}, as long as we can show the following
variant of \cite[Lemma 4.28]{ALB23}: The pair $(\Prism_{S/A}^{(1)}, \Fil^1_N \Prism_{S/A}^{(1)})$
is henselian.
As $\Prism_{S/A}^{(1)}$ is $I$-adically complete, we may check the claim after modulo $I$.
By the discussion before \Cref{structure of Nygaard filtration on twisted-prismatic cohomology},
we are reduced to showing that $(\mathrm{dR}_{S/\overline{A}}, \Fil^1_H \mathrm{dR}_{S/\overline{A}})$
is a henselian pair. Our claim now follows from the fact that
$\mathrm{dR}_{S/\overline{A}}$ is $p$-adically complete and the fact
that $\Fil^1_H$ always admits a (derived) divided power structure,
see \cite[p.~5 and \S 4]{Mag24}.
\end{proof}

\begin{remark}
\label{functorial rel}
Similar to \Cref{functorial abs}, for a map of large quasi-syntomic $\overline{A}$-algebras $S_1\to S_2$
and $*\in\{\emptyset, \vect, \perf\}$,
the completed filtered base change induces a natural functor
\begin{align*}
\Phi_{(S_1,S_2)}:\fG^*(\Spf(S_1)/A) &\longrightarrow \fG^*(\Spf(S_2)/A),\\
(E, ~\Fil^\bullet E^{(1)}, (\varphi_{E})) &
\longmapsto (E \widehat{\bigotimes}_{\Prism_{S_1/A}} \Prism_{S_2/A}, 
\Fil^\bullet E^{(1)} \widehat{\bigotimes}_{\Fil^\bullet_N \Prism_{S_1/A}^{(1)}} 
\Fil^\bullet_N \Prism_{S_2/A}^{(1)}, (\widetilde{\varphi}_{E} \otimes_{\varphi_{\Prism_{S_1/A}}} \varphi_{\Prism_{S_2/A}})).
\end{align*}
\end{remark}
\begin{definition}
\label{def rel F-gauge}
	Let $Z$ be a quasi-syntomic $\overline{A}$-formal scheme.
	The category of \emph{prismatic ($F$-)gauges} over $Z/A$ is defined as the limit of $\infty$-categories
	\[
	\fG^*(Z/A)\colonequals \lim_{S\in (Z/\overline{A})_\qrsp} \fG^*(\Spf(S)/A),
	\]
	where $*\in \{\emptyset, \vect, \perf, \coh\}$.
\end{definition}

As in the absolute case (\Cref{F-gauge sheaf}), we have the following sheaf property.
\begin{proposition}
\label{F-gauge sheaf relative}
	Let $S\to S^{(0)}$ be a quasi-syntomic cover of large quasi-syntomic $\overline{A}$-algebras, and let $S^{(\bullet)}$ be the $p$-completed \v{C}ech nerve.
	Then for $*\in \{\emptyset, \perf,\vect\}$, we have a natural equivalence
	\[
	\fG^*(\Spf(S)/A) \simeq \lim_{[n]\in \Delta} \fG^*(\Spf(S^{(n)})/A).
	\]
\end{proposition}

\begin{proof}
The proof strategy is exactly the same as that of \Cref{F-gauge sheaf}: via Rees's dictionary,
the descent property follows from \Cref{flatness of Nygaard filtered ring map} (4).
\end{proof}

Our main construction in this section is the following.
\begin{theorem}
\label{restriction functor}
	Let $X$ be a quasi-syntomic formal scheme, and let $Z$ be a quasi-syntomic $\overline{A}$-formal scheme, 
 together with a map $Z\to X$.
	Then for $*\in\{\emptyset, \perf, \vect\}$, there is a natural functor
	\[
BC_{X, Z/A} \colon \fG^*(X) \longrightarrow \fG^*(Z/A),
	\]
	compatible with the natural base change functor of prismatic $F$-crystals $\fC^*(X)\to \fC^*(Z/A)$.
\end{theorem}
As a preparation, we first consider the special case of quasiregular semiperfectoid rings.
\begin{lemma}
\label{restriction functor qrsp}
Let $S$ be a quasiregular semiperfectoid ring, 
and let $S'$ a large quasi-syntomic $\overline{A}$-algebra, 
together with a map $\Spf(S') \to \Spf(S)$.
        Then there is a functorial commutative diagram of filtered rings
	\[
	\begin{tikzcd}
		\Fil^\bullet_N \Prism_S \arrow[d, "\varphi_{\Prism_S}"] \ar[r]& \Fil^\bullet_N  \Prism_{S'/A}^{(1)} \arrow[d, "\varphi_{\Prism_{S'/A}^{(1)}}"] \\
		I^\bullet \Prism_S \ar[r] & I^\bullet \Prism_{S'/A},
	\end{tikzcd}
	\]
	such that the top row underlies the composition of homomorphisms of rings $\Prism_S \to \Prism_{S'/A} \to \Prism_{S'/A}^{(1)}$.
\end{lemma}
\begin{proof}
There is a natural commutative diagram:
\[
\begin{tikzcd}
\Prism_{S} \ar[r] \arrow[dd, "\varphi_{\Prism_{S}}"] & \Prism_{S'/A} \arrow[d, "\mathrm{id} \otimes 1"]\\
& \Prism_{S'/A}^{(1)} \coloneqq \Prism_{S'/A} \widehat{\otimes}_{A, \varphi_A} A
\arrow[d, "\varphi_{rel}"] \\
\Prism_{S} \ar[r]  & \Prism_{S'/A},
\end{tikzcd}
\]
where the horizontal maps follow from the initial property of $\Prism_{S}$ in the category $\Spf(S)_\Prism$,
and the right vertical map is the natural factorization of $\varphi_{\Prism_{S'/A}}$,
with $\varphi_{rel}$ being the relative Frobenius map for prismatic cohomology of $S'$ over $A$.
	Moreover, if an element $x\in \Prism_{S}$ is sent into $I^n\Prism_{S}$ under $\varphi_{\Prism_{S}}$, then by the commutativity, the image of $x$ in $\Prism_{S'/A}^{(1)}$ is sent into $I^n\Prism_{S'/A}$ under $\varphi_{rel}$.
	By definition of Nygaard filtration for quasiregular semiperfectoid rings \cite[Theorem 15.2.(1)]{BS19}, this in particular means that the above diagram carries $\Fil^n_N \Prism_{S}$ into $\Fil^n_N \Prism_{S'/A}^{(1)}$ under the top right composition:
	\[
	\begin{tikzcd}
		\Prism_{S} \ar[r] & \Prism_{S'/A} \arrow[r, "\mathrm{id} \otimes 1"] & \Prism_{S'/A}^{(1)}.
	\end{tikzcd} 
	\]
	Thus we get the filtered map as needed.	
\end{proof}
\begin{proof}[Proof of \Cref{restriction functor}]
We prove the version with $F$, the one without $F$ is similar and easier.
Assume $S\in X_\qrsp$ is any quasiregular semiperfectoid ring over $X$.
Then since $\Spf(S)\times_X Z$ is a quasi-syntomic formal scheme over $Z$, 
and $Z$ is quasi-syntomic over $\Spf(\overline{A})$, 
we see $\Spf(S)\times_X Z$ admits quasi-syntomic covers by formal spectra of large quasi-syntomic $\overline{A}$-algebras.
Let $S'$ be any large quasi-syntomic $\overline{A}$-algebra that is  quasi-syntomic over $\Spf(S)\times_X Z$.
Since both $F$-gauges and relative $F$-gauges satisfy quasi-syntomic descent
(\Cref{F-gauge sheaf} and \Cref{F-gauge sheaf relative}),
it suffices to construct functors $BC_{(S,S')} \colon \FG^*(\Spf(S)) \longrightarrow \FG^*(\Spf(S')/A)$
compatible with base change in both $S$ and $S'$
(defined in \Cref{functorial abs} and \Cref{functorial rel} respectively).

Lastly we define the functor $BC_{(S,S')}$, for any pair $(S, S')$ as above, by taking the filtered base change:
\begin{align*}
BC_{(S,S')} \colon \FG^*(\Spf(S)) &\longrightarrow \FG^*(\Spf(S')/A),\\
E=(E, \Fil^\bullet E, \varphi_{E}) &\longmapsto 
\left(E \widehat{\bigotimes}_{\Prism_{S}} \Prism_{S'/A},
\Fil^\bullet E \widehat{\bigotimes}_{\Fil^\bullet_N \Prism_{S}} 
\Fil^\bullet_N \Prism_{S'/A}^{(1)}, 
\widetilde{\varphi}_{E} \otimes \varphi_{\Prism_{S'/A}^{(1)}} \right).
\end{align*}
Here, unraveling definitions, the filtered map $\widetilde{\varphi}_{E} \otimes \varphi_{\Prism_{S'/A}^{(1)}}$ is given by 
\[
\begin{tikzcd}
\Fil^\bullet E \widehat{\bigotimes}_{\Fil_N^\bullet \Prism_S} \Fil_N^\bullet \Prism_{S'/A}^{(1)}
\arrow[rrr, "\widetilde{\varphi}_{E} \otimes_{\varphi_{\Prism_{S}}} \varphi_{\Prism_{S'/A}^{(1)}}"] &   & &
I^\bullet E \widehat{\bigotimes}_{I^\bullet \Prism_S} I^\bullet \Prism_{S'/A} 
\simeq I^\bullet (E \widehat{\otimes}_{\Prism_S} \Prism_{S'/A}),
\end{tikzcd}
\]
and by the commutative diagram in \Cref{restriction functor qrsp}, 
its linearization is 
\begin{align*}
\left( \Fil^\bullet E \widehat{\bigotimes}_{\Fil_N^\bullet \Prism_S} \Fil_N^\bullet \Prism_{S'/A}^{(1)} \right)
\widehat{\bigotimes}_{\Fil_N^\bullet \Prism_{S'/A}^{(1)}} I^\bullet \Prism_{S'/A} & 
\cong \left(\Fil^\bullet E \widehat{\bigotimes}_{\Fil^\bullet_N \Prism_{S}} I^\bullet \Prism_S \right)
\widehat{\bigotimes}_{I^\bullet \Prism_S} I^\bullet \Prism_{S'/A} \\
& \xlongrightarrow{\varphi_E \otimes \mathrm{id}} I^\bullet E \widehat{\bigotimes}_{I^\bullet \Prism_S} I^\bullet \Prism_{S'/A},
\end{align*}
which is a filtered isomorphism by assumption of 
$E=(E, \Fil^\bullet E, \varphi_{E}) \in \FG(\Spf(S))$.
\end{proof}

\begin{remark}
Let us note that the functor in \Cref{restriction functor} satisfies
base change with respect to the base prism: in the setting of the theorem,
suppose we are given a map of prisms $(A, I) \to (\widetilde{A}, I\widetilde{A})$
and let $\widetilde{Z} \coloneqq Z \times_{\overline{A}} \overline{\widetilde{A}}$.
Then we have a commutative diagram:
\[
\xymatrix{
\fG^*(X) \ar[rr] \ar[rd] & & \fG^*(Z/A) \ar[ld]^{- \widehat{\otimes}_A \widetilde{A}} \\
& \fG^*(\widetilde{Z}/\widetilde{A}), &
}
\]
where the two unlabelled arrows are what constructed in \Cref{restriction functor}.
Indeed, tracing through the construction in the proof of \Cref{restriction functor},
we are reduced to a similar statement with $X$ and $Z$ replaced
by $S$ and $S'$ (and $\widetilde{Z}$ replaced by 
$\widetilde{S'} \coloneqq S' \widehat{\otimes}_{\overline{A}} \overline{\widetilde{A}}$),
which immediately follows from the base change formula of
prismatic cohomology
\[
\Prism_{S'/A} \widehat{\otimes}_A \widetilde{A} \cong \Prism_{\widetilde{S'}/\widetilde{A}},
\]
as well as similar base change formulas for Nygaard filtrations (\cite[Theorem 15.2.(3)]{BS19}).
\end{remark}

Similar to \Cref{absolute weight filtration subsection},
there is a natural notion of weight and weight filtration on the associated graded of a relative prismatic $F$-gauge.
\begin{definition}
\label{wt of rel F-gauge}
Let $Z$ be a quasi-syntomic $\overline{A}$-formal scheme,
and let $[a, b]$ be an interval in $\mathbb{R}\cup \{-\infty, \infty\}$.
For a graded complex $M^\bullet \in \DG^\ast_{p\text{-comp}}((Z/\overline{A})_\qrsp, \Gr_N^\bullet \Prism_{-/A}^{(1)})$.
\begin{enumerate}
\item The reduction of $M^\bullet$, as a graded complex over $((Z/\overline{A})_\qrsp, \mathcal{O}_\qrsp)$, 
is defined as the graded base change 
\[
\Red_{Z/A}(M^\bullet) \colonequals M^\bullet \widehat{\bigotimes}_{\Gr_N^\bullet \Prism_{-/A}^{(1)}} \mathcal{O}_\qrsp.
\]
The $i$-th graded piece of $\Red_{Z/A}(M^\bullet)$ is denoted as $\Red_{i,Z/A}(M^\bullet)$.
\item The graded complex $M^\bullet$ is said to \emph{have weights in} $[a,b]$ if $M^i=0$ for $i\ll 0$ and 
\[
\Red_{i,Z/A}(M^\bullet)=0, \forall i\notin [a,b].
\]
\end{enumerate}
For a relative ($F$-)gauge $E=(E,\Fil^\bullet E^{(1)}, \varphi_E)$ over $Z/A$ we define its reduction
and weights by that of its associated graded $M^\bullet=\Gr^\bullet E^{(1)}$.
\end{definition}

\begin{proposition}
\label{wt filtration rel F-gauge}
Let $Z$ be a quasi-syntomic $\overline{A}$-formal scheme, 
let $\ast\in \{\emptyset, \perf,\vect\}$, and let $a\leq b$ be two integers.
\begin{enumerate}[label=(\roman*)]
\item Assume $M^\bullet\in \DG^\ast_{p\text{-}\mathrm{comp}}(Z_\qrsp, \Gr_N^\bullet \Prism_{-/A}^{(1)})$ 
has weights in $[a,b]$.
Then there is a unique increasing filtration
$\Fil^\wt_i (M^\bullet)$ on $M^\bullet$ satisfying the following axioms:
\begin{enumerate}
\item The filtration is exhaustive;
\item the filtration ``starts at $0$'': i.e.~$\Fil^\wt_{\ll 0} (M^\bullet) = 0$; and
\item the graded piece $\Gr^\wt_i(M^\bullet) \simeq 
N_i \widehat{\otimes}_{\mathcal{O}_Z} \Gr^\bullet_N \Prism_{-/A}^{(1)}$ 
where $N_i \in \DG^\ast_{p\text{-comp}}(\mathcal{O}_Z)$ has grading $i$.
\end{enumerate}
Moreover, the induced filtration $\Fil^\wt_i(\Red_{Z/A}(M^{\bullet})) \coloneqq \Red_{Z/A}(\Fil^\wt_i (M^\bullet))$
is the one induced by grading: 
\[
\Fil^\wt_i(\Red_{Z/A}(M^{\bullet})) \cong \bigoplus_{j \leq i} \Red_{j,Z/A}(M^{\bullet}).
\]
In particular, there are natural isomorphisms:
\[
N_i \cong \Red_{Z/A}(\Gr^\wt_i(M^\bullet)) \cong \Gr^\wt_i(\Red_{Z/A}(M^{\bullet})) \cong 
\Red_{i, Z/A}(M^{\bullet})
\]
in $\DG^\ast_{p\text{-comp}}(\mathcal{O}_Z)$.
Therefore, the filtration is indexed by $i \in [a,b]$.
\item Assume there is a map of $p$-adic formal schemes $Z \to X$ such that  $X$ is quasi-syntomic.
Let $E=(E,\Fil^\bullet E) \in \mathrm{Gauge}^\ast(X)$ having weights in $[a,b]$, 
and let $E'$ be the associated relative gauge over $Z/A$ as in \Cref{restriction functor}.
Then $E'$ also has weights in $[a,b]$.
In fact we have a natural graded tensor product formula:
\[
\Red_{Z/A}(E') \cong \Red_X(E)\widehat{\otimes}_{\mathcal{O}_X} \mathcal{O}_Z.
\]
Moreover the weight filtrations on $M^{\bullet} \coloneqq \Gr^{\bullet}(E)$ and 
$N^{\bullet} \coloneqq \Gr^{\bullet}(E')$ are related via a natural filtered tensor product formula:
\[
\Fil^{\wt}_i(N^{\bullet}) \cong \Fil^{\wt}_i(M^{\bullet}) \widehat{\otimes}_{\Gr^\bullet_N} \Gr^\bullet_N \Prism_{-/A}^{(1)}.
\]
\end{enumerate}
\end{proposition}

\begin{proof}
	Part (i) is identical with the proof of \Cref{thm weight filtration of gauge} and we do not repeat here.
	For part (ii), it suffices to check the claim quasi-syntomic locally, and we assume there is a quasiregular semiperfectoid algebra $S\in X_\qrsp$, a large quasi-syntomic algebra $S'\in (Z/\overline{A})_\qrsp$, together with a map $S\to S'$ that is compatible with $X\to Z$.
	Then by proof of \Cref{restriction functor}, we have $\Fil^\bullet E'(S')^{(1)}= \Fil^\bullet E(S)\widehat{\otimes}_{\Fil^\bullet_N \Prism_S} \Fil^\bullet_N \Prism_{S'/A}^{(1)}$.
	Notice that we also have a commutative diagram of graded rings
	\[
	\begin{tikzcd}
	S \arrow[r, "\iota"] \ar[d] & \Gr^\bullet_N \Prism_S \ar[d] \arrow[r, "\Gr^0"] & S \ar[d] \\ 
	S' \arrow[r, "\iota"] &	\Gr^\bullet_N \Prism_{S'/A}^{(1)} \arrow[r, "\Gr^0"] & S'.
	\end{tikzcd}
    \]
    As a consequence, by the local formula of the reduction functors (\Cref{reduction of gauge}, \Cref{wt of rel F-gauge}), we get the tensor product formula 
    \[
    \Red_{S'/A}(E') \simeq \Red_S(E) \widehat{\otimes}_S S',
    \]
and the last formula follows from uniqueness of the weight filtration on $N^{\bullet}$.
\end{proof}

\subsection{Filtered Higgs complex and Hodge--Tate realization}
\label{sec filtered Higgs}
In eye of \Cref{structure of Nygaard filtration on twisted-prismatic cohomology},
we study the mod $t$ specialization of a relative prismatic $F$-gauge
in this subsection.

We start by considering the notion of filtered Higgs complex and the induced conjugate filtration on Hodge--Tate cohomology. 
Let $(A,I)$ be a bounded prism.
Recall that for a large quasi-syntomic formal scheme $S$ over $\overline{A}$, the relative Hodge--Tate cohomology ring 
$\overline{\Prism}_{S/A}$ admits an increasing filtration called \emph{conjugate filtration}, 
whose $i$-th graded factor is $\wedge^i\mathbb{L}_{S/\overline{A}}\{-i\}[-i]$
(see \cite[Construction 7.6]{BS19}).
We use $\Gr^\conj_\bullet \overline{\Prism}_{S/A} = \bigoplus_{i \in \mathbb{N}} \wedge^i\mathbb{L}_{S/\overline{A}}\{-i\}[-i]$
to denote the associated graded algebra over $S$.
\begin{definition}
\label{def Higgs}
	Let $Z$ be a quasi-syntomic $\overline{A}$-formal scheme.
	\begin{enumerate}[label=(\roman*)]
		\item The category of \emph{filtered Higgs fields} over $Z/A$ is defined as the limit of $\infty$-categories
		\[
		\FH(Z/A)\colonequals   \lim_{S \in (Z/\overline{A})_\qrsp} \FH(\Spf(S)/A),
		\]
		where $\FH(\Spf(S)/A)\colonequals \DF_{p\text{-comp}}(\Fil^\conj_\bullet \overline{\Prism}_{S/A})$ is the category of $p$-complete filtered complexes over the filtered ring $\Fil^\conj_\bullet \overline{\Prism}_{S/A}$
        for any large quasi-syntomic $\overline{A}$-formal scheme.
		\item The category of \emph{graded Higgs fields} over $Z/A$ is defined as the limit of $\infty$-categories  
		\[
		\GH(Z/A) \colonequals \lim_{S \in (Z/\overline{A})_\qrsp} \GH(\Spf(S)/A),
		\]
		where $\GH(\Spf(S)/A)$ is the category of $p$-complete graded complexes over the graded ring $\Gr^\conj_\bullet  \overline{\Prism}_{S/A}$ for any large quasi-syntomic $\overline{A}$-formal scheme.
	\end{enumerate}
Here, similar to \Cref{functorial abs}, for a map of large quasi-syntomic $\overline{A}$-algebras $S_1\to S_2$,
the completed filtered/graded base change induces natural functors between
$\FH(\Spf(S_i)/A)$ and $\GH(\Spf(S_i)/A)$.

Similar to \Cref{def F-gauge category}, for $*\in \{\vect, \perf\}$, 
we use $\FH^*(S/A)$ and $\GH^*(S/A)$ to denote the corresponding subcategory 
spanned by $p$-completely finite projective or 
$p$-completely perfect objects.
Then we may define $\FH^*(Z/A)$ and $\GH^*(Z/A)$ by similar limit formulae.
\end{definition}

The analog of \Cref{Gauge: cplt vs alg conditions} still holds true:
\begin{proposition}
\label{Higgs: cplt vs alg conditions}
Let $S$ be a large quasi-syntomic $\overline{A}$-formal scheme and let 
$Z$ be a quasi-syntomic $\overline{A}$-formal scheme. Then for $*\in \{\vect, \perf\}$, 
we may drop ``$p$-completely'', when defining $\FH^*(S/A)$, $\FH^*(Z/A)$, $\GH^*(S/A)$ and $\GH^*(Z/A)$,
without changing the definition.
\end{proposition}

\begin{proof}
The case of general $Z$ follows from the case of large quasi-syntomic $\overline{A}$-algebras
$S$. In this case, we are reduced to \Cref{graded: vb vs cplt vb}.
\end{proof}

\begin{construction}[Associated graded functor]
\label{def graded functor Higgs}
There is a natural functor by taking associated graded:
	\[
	\Gr: \FH^*(Z/A) \longrightarrow \GH^*(Z/A),
	\]
	for $*=\{\emptyset, \vect, \perf\}$.
	When $Z = \Spf(S)$ is large quasi-syntomic, it sends a filtered 
 $\Fil^\conj_\bullet \overline{\Prism}_{S/A}$-complex 
 $\Fil_\bullet M$
 to the graded $\Gr^\conj_\bullet \overline{\Prism}_{S/A}$-complex
 $\Gr_\bullet M = \bigoplus_{i \in \mathbb{N}} \Fil_i M/ \Fil_{i-1} M$.
\end{construction}

We also have the quasi-syntomic sheaf property for categories of filtered and graded Higgs complexes.
\begin{proposition}
\label{Higgs sheaf}
Let $S \to S^{(0)}$ be a quasi-syntomic cover of large quasi-syntomic rings over $\overline{A}$,
and let $S^{(\bullet)}$ be the $p$-completed \v{C}ech nerve.
	Then for $*\in \{\emptyset, \perf,\vect\}$, we have a natural equivalence
	\begin{align*}
		\FH^*(\Spf(S)/A) & \simeq \lim_{[n]\in \Delta} \FH^*(\Spf(S^{(n)})/A),\\
		\GH^*(\Spf(S)/A) & \simeq \lim_{[n]\in \Delta} \GH^*(\Spf(S^{(n)})/A).
	\end{align*}
\end{proposition}
\begin{proof}
One just mimics the proof of \Cref{F-gauge sheaf},
and uses \Cref{flatness of Nygaard filtered ring map} (2) as the main ingredient.
\end{proof}

\begin{remark}
Let us justify the appearance of ``Higgs'' in the above discussion.
By construction, there is a forgetful functor from $\FH(Z/A)$ to the category of Hodge--Tate crystals over
$(Z/A)_\Prism$, by forgetting the filtration structure.
Moreover, when $Z$ is affine and smooth over $\overline{A}$,
showing in \cite[Corollary 6.6]{BL22b} (independently by \cite[Thm.\ 4.12]{Tian21}, and for affine smooth $Z$ that is small by \cite[Thm.\ 1.1]{MW22}) Hodge--Tate crystals over $Z/\overline{A}$ are equivalent to quasi-nilpotent Higgs fields in complexes of $Z/A$.
So we can think of objects in $\FH(Z/A)$ as derived filtered generalizations of the usual notion of Higgs fields, over $p$-adic formal schemes that may not be smooth.
\end{remark}

To relate the relative prismatic ($F$-)gauge with the filtered Higgs field, we first recall the following observation on Nygaard graded pieces of relative prismatic cohomology from \cite[\S 15]{BS19}.
Assume the ideal $I$ is generated by an element $d$, which we fix for the rest of this subsection.
Then one can define a twisted version of the conjugate filtered Hodge--Tate cohomology, by
\[
S\longmapsto \left( \colim~\Fil^\conj_\bullet \overline{\Prism}_{S/A} \{\bullet\} \right)^\wedge_p,
\]
where the transition map is given by the tensor product of the canonical map $\Fil^\conj_i \overline{\Prism}_{S/A} \to \Fil^\conj_{i+1} \overline{\Prism}_{S/A}$ with the twisting map $d:I^i/I^{i+1} \simeq I^{i+1}/I^{i+2}$.

\begin{fact}\label{Nygaard graded and conjugate filtered}
	Let $d$ be a generator of the ideal $I$.
	For a large quasi-syntomic $\overline{A}$-algebra $S$, the relative Frobenius 
 $\varphi_{rel} \colon \Prism_{S/A}^{(1)} \to \Prism_{S/A}$ identifies the graded ring $\Gr^\bullet_N \Prism_{S/A}^{(1)}$ with the Rees construction of the twisted conjugate filtered Hodge--Tate cohomology $ \Fil^\conj_\bullet \overline{\Prism}_{S/A}\{\bullet\}$.
	As a consequence, for $*\in\{\emptyset, \perf,\vect\}$, we have an equivalence
	\[
	\DG_{p\text{-}\mathrm{comp}}^*(\Gr^\bullet_N \Prism_{S/A}^{(1)}) \simeq
 \DF_{p\text{-}\mathrm{comp}}^*( \Fil^\conj_\bullet \overline{\Prism}_{S/A}).
	\]
\end{fact}
\begin{proof}
	The first identification of rings is in \cite[Thm.\ 15.2.(2)]{BS19}
 and also recorded in \Cref{structure of Nygaard filtration on twisted-prismatic cohomology}.
	To get the equivalence of categories, we simply notice that there is an equivalence of filtered derived categories
	\begin{align*}
		\DF(\Fil^\conj_\bullet \overline{\Prism}_{S/A}) & \longrightarrow \DF(\Fil^\conj_\bullet \overline{\Prism}_{S/A}\{\bullet\}),\\
		\Fil_\bullet M &\longmapsto \Fil_\bullet M\otimes_{\overline{A}} \bar{I}^\bullet/\bar{I}^{\bullet+1},
	\end{align*}
    which is compatible with the twisting of the conjugate filtered Hodge--Tate cohomology.
\end{proof}
By taking limit ranging over all $S \in (Z/\overline{A})_\qrsp$,
the above naturally extends to general quasi-syntomic $\overline{A}$-formal schemes.
\begin{corollary}
\label{Nygaard graded and conjugate filtered general}
	Let $Z$ be a quasi-syntomic $\overline{A}$-formal scheme, and let $*\in\{\emptyset, \perf,\vect\}$.
	There is an equivalence of categories
	\[
	\DG_{p\text{-}\mathrm{comp}}^*((Z/\overline{A}))_\qrsp, \Gr^\bullet_N \Prism_{-/A}^{(1)}) \simeq \FH^*(Z/A).
	\]
\end{corollary}
\begin{definition}
\label{def HT realization}
	Let $Z$ be a quasi-syntomic $\overline{A}$-formal scheme, and let $*\in\{\emptyset, \perf,\vect\}$.
	We define the \emph{Hodge--Tate realization} functor
	\[
	\fG^*(Z/A) \longrightarrow \FH^*(Z/A)
	\]
	to be the composition of the associated graded functor with the equivalence in \Cref{Nygaard graded and conjugate filtered general}.
	Namely it is the limit of compositions
	\[
	\fG^*(\Spf(S)/A) \longrightarrow \DG_{p\text{-comp}}^*(\Gr^\bullet_N \Prism_{S/A}^{(1)}) \simeq \DF_{p\text{-comp}}^*( \Fil^\conj_\bullet \overline{\Prism}_{S/A}) = \FH^*(\Spf(S)/\overline{A}),
	\]
where $S$ ranges over all $(Z/\overline{A})_\qrsp$.
\end{definition}
The following relates Hodge--Tate cohomology and Higgs cohomology.

\begin{proposition}\label{Higgs associated to F-gauge cohomology}
	Let $Z$ be a quasi-syntomic formal scheme over $\overline{A}$, and let $*\in \{\emptyset, \perf, \vect\}$.
    Let $E=(E,\Fil^\bullet E^{(1)},\varphi_{E})\in \FG^*(Z/A)$  and let $(M, \Fil_\bullet M)$ 
    be the associated filtered Higgs field.
		Then $\varphi_{E}$ induces the following isomorphisms 
		\begin{enumerate}
			\item twisted graded pieces $\Gr^i E^{(1)}\{-i\}$ of $\Fil^\bullet E^{(1)}$ and filtrations $\Fil_i M$ of $M$;
			\item Hodge--Tate cohomology of $E$ and Higgs cohomology of $M$.
		\end{enumerate}
\end{proposition}
As a consequence, for a relative $F$-gauge $E$ over $X/A$,
we get a natural filtration on $R\Gamma((X/A)_\Prism, \overline{E})$,
which we call the \emph{conjugate filtration} on Hodge--Tate cohomology of $E$.

\begin{proof}[Proof of \Cref{Higgs associated to F-gauge cohomology}]
Let $E=(E, \Fil^\bullet E^{(1)}, \varphi_{E})$ be an object in $\FG^*(\Spf(S)/A)$ for a given $S\in (Z/\overline{A})_\qrsp$, and let $(M,\Fil_\bullet M)$ be the associated filtered Higgs field.
	On the one hand, by \Cref{equivalent def rel F-gauge qrsp} the map $\varphi_{E}$ induces an isomorphism of complexes
	\[
	\left( \colim_{i\in \mathbb{Z}} (\cdots \to \Fil^i E^{(1)} \otimes I^{-i} \to \Fil^{i+1} E^{(1)} \otimes I^{-i-1}\to \cdots) \right)^\wedge_{(p,I)} \simeq E,
	\]
	whose reduction mod $I$ is the Hodge--Tate cohomology, namely
	\[
	\left( A/I \otimes_A \colim_{i\in \mathbb{Z}} (\cdots \to \Fil^i E^{(1)} \otimes I^{-i} \to \Fil^{i+1} E^{(1)} \otimes I^{-i-1}\to \cdots) \right)^\wedge_p \simeq \overline{E}.
	\]
	Moreover, by taking the mod $I$ 
 (where $I$ is having filtration degree $1$) reduction within the colimit, we can rewrite the left hand side as
	\[
	\left( \colim_{i\in \mathbb{Z}} (\cdots \to \Gr^i E^{(1)} \otimes \bar{I}^{-i} \to \Gr^{i+1} E^{(1)} \otimes \bar{I}^{-i-1}\to \cdots) \right)^\wedge_p \simeq \overline{E}.
	\]
	On the other hand, 
	the twisted graded factor $\Gr^\bullet E^{(1)}\{-\bullet\}$ is the Rees construction for the associated filtered Higgs field $\Fil_\bullet M$.
	In particular, the Higgs cohomology, which is the underlying complex $M$ of the increasingly filtered objects $\Fil_i M$, is equal to
	\[
	\left( \colim_{i\in \mathbb{Z}} (\cdots \to \Fil_i M \to \Fil_{i+1} M \to \cdots ) \right)^\wedge_p \simeq \left( \colim_{i\in \mathbb{Z}} (\cdots \to \Gr^i E^{(1)}\{-i\} \to \Gr^{i+1} E^{(1)}\{-i-1\} \to \cdots ) \right)^\wedge_p.
	\]
	Hence the relative Frobenius identifies the Hodge--Tate cohomology of the $F$-gauge $E$ with the Higgs cohomology of the associated filtered Higgs field $M$, and sends $\Gr^i E^{(1)}\{-i\}$ isomorphically onto $\Fil_i M$.
	For arbitrary quasi-syntomic $\overline{A}$-formal scheme $Z$ and a relative $F$-gauge $E \in \FG^*(Z/A)$, 
	since the above isomorphisms for $E(S)$ is functorial in $S \in (Z/\overline{A})_\qrsp$, 
	by passing to limit we may conclude the isomorphism of cohomology in general.
\end{proof}

The identification in \Cref{Nygaard graded and conjugate filtered general} in particular allows us to study the notion of weights on filtered Higgs complexes.
Following \Cref{wt of rel F-gauge}, for $\ast\in \{\emptyset, \perf, \vect\}$, we let $\Red_{Z/A}$ be the base change functor
\[
\begin{tikzcd}
	\FH^\ast(Z/\overline{A}) \arrow[rrr, "-\otimes_{\Fil^\conj_\bullet \overline{\Prism}_{-/A}} \mathcal{O}_\qrsp"] &&&  \DG^\ast_{p\text{-comp}}(\mathcal{O}_Z),
\end{tikzcd}
\]
which by construction factors through the quotient functor $\Gr \colon \FH^\ast(Z/\overline{A}) \to \GH^\ast(Z/\overline{A})$.
A filtered Higgs complex $M$ is of \emph{weight} $[a,b]$ if $\Red_{Z/A}(M)$ lives within degree $[a,b]$.

Translating  \Cref{wt filtration rel F-gauge} using \Cref{Nygaard graded and conjugate filtered general}, we get the weight filtration on filtered Higgs fields.
\begin{corollary}
\label{wt filtration of filtered Higgs}
Fix a generator $d$ of the ideal $I$, let $Z$ be a quasi-syntomic $p$-adic formal scheme over $\overline{A}$.
Let $*\in\{\emptyset, \perf,\vect\}$, and let $a\leq b$ be two integers.
For a filtered Higgs complex $M \in \FH^\ast(Z/\overline{A})$ 
of weight $[a,b]$, there is a unique increasing filtration
$\Fil^\wt_i M$ on $M$ satisfying the following axioms:
\begin{enumerate}
\item The filtration is exhaustive;
\item the filtration ``starts at $0$'': i.e.~$\Fil^\wt_{\ll 0} M = 0$; and
\item the graded piece $\Gr^\wt_i M \simeq 
N_i \otimes_{\mathcal{O}_Z} \Fil^\conj_\bullet \overline{\Prism}_{-/A}$ 
where $N_i \in \DG^\ast_{p\text{-comp}}(\mathcal{O}_Z)$ has grading $i$.
\end{enumerate}
Moreover, the induced filtration $\Fil^\wt_i(\Red_{Z/A}(M)) \coloneqq \Red_{Z/A}(\Fil^\wt_i M)$
is the one induced by grading: 
\[
\Fil^\wt_i(\Red_{Z/A}(M)) \cong \bigoplus_{j \leq i} \Red_{j,Z/A}(M).
\]
In particular, there are natural isomorphisms:
\[
N_i \cong \Red_{Z/A}(\Gr^\wt_i(M^\bullet)) \cong \Gr^\wt_i(\Red_{Z/A}(M)) \cong 
\Red_{i, Z/A}(M)
\]
in $\DG^\ast_{p\text{-comp}}(\mathcal{O}_Z)$.
Therefore, the filtration is indexed by $i \in [a,b]$.
\end{corollary}

By taking the direct image to the Zariski site of $Z$, we can analyze the graded factors of the conjugate filtration on Hodge--Tate cohomology.
\begin{construction}
\label{meaning of lambda*}
	Denote by $\lambda \colon Z_\qrsp \to Z_\mathrm{Zar}$ the natural map of ringed sites.
	Then there is a natural graded derived pushforward functor between two graded derived  category of presheaves
	\[
	R\lambda_* \colon \DG((Z/\overline{A})_\qrsp,\mathcal{O}_\qrsp^\mathrm{PSh}) \to \DG(Z_\mathrm{Zar}, \mathcal{O}_{Z}^\mathrm{PSh}).
	\]	
\end{construction}

\begin{remark}
For any open $U \subset Z$, the evaluation of $R\lambda_*$ at $U$
of quasi-syntomic sheaves is given by the unfolding
process, see \Cref{unfolding process in relative setting}
and \Cref{how to compute unfolding remark}.
\end{remark}

We now use weight filtration to subdivide the Hodge--Tate cohomology and Higgs cohomology into smaller pieces, where each of them can be understood using the reduction functor and differential forms.

\begin{corollary}
\label{conjugated graded of HT complex}
Let $Z$ be a smooth $\overline{A}$-formal scheme of relative dimension $n$, let $*\in \{\vect,\perf\}$, and let $M\in \FH^*(Z/\overline{A})$ be of weight $[a,b]$.
The $j$-th graded factor $\Gr_j R\lambda_* M$ is zero unless $j\in [a,b+n]$.
		In the latter case it admits a finite increasing filtration such that the $i$-th graded factor of this filtration is 
		\[
		(\Red_{a+i, Z/A}(M)) \otimes_{\mathcal{O}_X} \Omega^{j-a-i}_{Z/\overline{A}}\{a+i-j\}[a+i-j],\quad \max\{0,j-a-n\}\leq i \leq \min\{j,b\}-a.
		\]
		In particular, if $\bigoplus_l \Red_{l,Z/A}(M)$ is given by a coherent sheaf over $\mathcal{O}_Z$, then we have
		\[
		\Gr_j R\lambda_* M \in D^{[\max\{j-b,0\}, \min\{j-a,n\}]}_\mathrm{coh}(\mathcal{O}_Z).	\]
\end{corollary}
	Here in the calculation, we implicitly use the fact that for a smooth $\overline{A}$-formal scheme of relative dimension $n$, the relative sheaf of differential $\Omega^l_{Z/\overline{A}}$ is zero unless $l\in [0,n]$.

	For the future usage, we also record a special case when an $F$-gauge comes from an $I$-torsionfree coherent $F$-crystal, where we can show that the reduction of the associated filtered Higgs field (which is the same as the reduction of the gauge by \Cref{Nygaard graded and conjugate filtered general}) is coherent.
	This is summarized in the next result.
\begin{corollary}
\label{reduction of relative F-gauge that comes from F-crys}
Let $f:X\to Y$  be a smooth morphism of smooth $p$-adic formal schemes over $\mathcal{O}_K$, and let $(\mathcal{E},\varphi_{\mathcal{E}})$ be an $I$-torsionfree coherent $F$-crystal over $X$.
    	For a bounded prism $(A,I)\in Y_\Prism$ such that $\Spf(\overline{A})\to Y$ is $p$-completely flat, let
     $E' = BC_{X, X_{\overline{A}}/A}(\Pi_X(\mathcal{E}))$ be the associated relative $F$-gauge over $X_{\overline{A}}/A$ via 
    \Cref{thm crystal to gauge} and \Cref{restriction functor}.
    	Then the reduction $\Red_{X_{\overline{A}}/A}(E')$ is a graded coherent sheaf over $\mathcal{O}_X$.
    \end{corollary}
    \begin{proof}
    	Using the identification in \Cref{Nygaard graded and conjugate filtered general}, this follows from  the complete base change formula $\Red_{X_{\overline{A}}/A}(E')\simeq \Red_X(E)\widehat{\otimes}_{\mathcal{O}_X} \mathcal{O}_{X_{\overline{A}}}$ in \Cref{wt filtration rel F-gauge}.(ii): since $\Red_X(E)$ is coherent over $\mathcal{O}_X$ (\Cref{Reduction is the graded of twisted filtration}) and $X_{\overline{A}}\to X$ is $p$-completely flat, the complete base change is also coherent. 
    \end{proof}

\subsection{Completeness of Nygaard filtration via filtered de Rham realization}
\label{Completeness of Nygaard filtration via filtered de Rham realization}
In this subsection, in eye of \Cref{structure of Nygaard filtration on twisted-prismatic cohomology},
we study the mod $\frac{d}{t}$ specialization of a relative prismatic $F$-gauge
in this subsection.
The following definition and sheaf-property proposition is completely analogous
to previous subsections.

\begin{definition}
\label{def connections}
Let $Z$ be a quasi-syntomic $\overline{A}$-formal scheme.
The category of \emph{filtered connections} over $Z/\overline{A}$ is defined as the limit of $\infty$-categories
\[
\FConn(Z/\overline{A})\colonequals   \lim_{S \in (Z/\overline{A})_\qrsp} \FConn(\Spf(S)/\overline{A}),
\]
where $\FConn(\Spf(S)/\overline{A}) \colonequals 
\DF_{p\text{-comp}}(\Fil^\bullet_H \mathrm{dR}(S/\overline{A}))$ is the category of $p$-complete filtered complexes
over the filtered ring $\Fil^\bullet_H \mathrm{dR}(S/\overline{A})$.
Here, similar to \Cref{functorial abs}, for a map of large quasi-syntomic $\overline{A}$-algebras $S_1\to S_2$,
the completed filtered base change induces natural functors between
$\FConn(\Spf(S_i)/\overline{A})$.

Similar to \Cref{def F-gauge category}, for $*\in \{\vect, \perf\}$, 
we use $\FConn^*(Z/\overline{A})$ to denote the subcategory spanned by $p$-completely
finite projective (resp.~perfect) objects.
\end{definition}

The analog of \Cref{Gauge: cplt vs alg conditions} still holds true:
\begin{proposition}
\label{Connections: cplt vs alg conditions}
Let $S$ be a large quasi-syntomic $\overline{A}$-formal scheme and let 
$Z$ be a quasi-syntomic $\overline{A}$-formal scheme. Then for $*\in \{\vect, \perf\}$, 
we may drop ``$p$-completely'', when defining $\FConn^*(S/\overline{A})$ and 
$\FConn^*(Z/\overline{A})$, without changing the definition.
\end{proposition}

\begin{proof}
The case of $\FConn^*(Z/\overline{A})$ follows from the case of 
$\FConn^*(S/\overline{A})$ by descent. The proof of $\FConn^*(S/\overline{A})$
follows from the exact proof of \Cref{relative Gauge: cplt vs alg conditions}:
Note that in the said proof, we have explained why 
$(\mathrm{dR}_{S/\overline{A}}, \Fil^1_H \mathrm{dR}(S/\overline{A}))$
is a henselian pair.
\end{proof}

Just like before, all these form quasi-syntomic sheaves.

\begin{proposition}
\label{FilConn sheaf}
Let $S \to S^{(0)}$ be a quasi-syntomic cover of large quasi-syntomic rings over $\overline{A}$,
and let $S^{(\bullet)}$ be the $p$-completed \v{C}ech nerve.
	Then for $*\in \{\perf,\vect\}$, we have a natural equivalence
	\begin{align*}
		\FConn^*(\Spf(S)/\overline{A}) & \simeq \lim_{[n]\in \Delta} \FConn^*(\Spf(S^{(n)})/\overline{A}).
	\end{align*}
\end{proposition}
\begin{proof}
One just mimics the proof of \Cref{F-gauge sheaf},
and uses \Cref{flatness of Nygaard filtered ring map} (3) as the main ingredient.
\end{proof}

In \cite[Lemma 8.6]{BS19} one finds an interesting example of $(p, I)$-completely descendable map
(see \Cref{descendable discussion})
between two $(p, I)$-complete $\mathbb{E}_\infty$-rings,
what we need is the following slight generalization.
\begin{proposition}
\label{descendable proposition}
Let $(A, I = (d))$ be a bounded prism with $\overline{A} \coloneqq A/I$, let 
$\overline{A}\langle \underline{X} \rangle \to R$
be a $p$-completely \'{e}tale map, and let 
$R_{\infty} \coloneqq R \widehat{\otimes}_{\overline{A}\langle \underline{X} \rangle} 
\overline{A}\langle \underline{X}^{1/p^\infty} \rangle$.
Then the filtered map
\[
\Fil^{\bullet}_N \mathrm{R\Gamma}(\Spf(R)_{\qsyn}, \Prism^{(1)}_{-/A})
\to \Fil^{\bullet}_N \Prism^{(1)}_{R_{\infty}/A}
\]
induces a $(p, I)$-completely descendable map between their underlying Rees algebras.
Similarly, the filtered map
\[
\Fil^{\bullet}_H \mathrm{dR}(R/\overline{A})
\to \Fil^{\bullet}_H \mathrm{dR}(R_{\infty}/\overline{A})
\]
induces a $p$-completely descendable map between their underlying Rees algebras.
\end{proposition}

\begin{proof}
Since descendability is preserved under base change, we are immediately
reduced to the case where $\overline{A}\langle \underline{X} \rangle = R$.
It suffices to prove the statement when there is only one variable,
as the general case follows from tensoring up one variable case.
Just like the proof of \cite[Lemma 8.6]{BS19}, we use $F$ to denote
the fiber in both Rees algebra maps, and we just need to show
$F^{\otimes 2}$ has only nullhomotopic maps to the Rees algebra
for the filtered sheaves evaluated at $\Spf(R)$.

For the case of Nygaard filtered $\Prism^{(1)}_{-/A}$, 
by $\frac{d}{t}$-completeness, it suffices to check
the above claim after modulo $\frac{d}{t}$ (using notation
from \Cref{structure of Nygaard filtration on twisted-prismatic cohomology});
by the \Cref{structure of Nygaard filtration on twisted-prismatic cohomology}
we are reduced to the case of Hodge filtered $\mathrm{dR}(-/\overline{A})$.
At this point, we simply proceed as in the proof of \cite[Lemma 8.6]{BS19}:
By completeness we may further reduce mod $p$, then our map has a model
defined over $\mathbb{F}_p$, and the proof as in loc.~cit.~applies
to show the relevant mapping space has only one component.
\end{proof}

Combining all the above discussions, we get an equivalent definition of ($F$-)gauges and 
filtered connections as certain filtered complexes over the filtered cohomology rings.
Roughly speaking, this is saying that the relative versions of the filtered prismatization $(X/A)^N$, the syntomication $(X/A)^\mathrm{syn}$ and the filtered de Rham stack $(X/\overline{A})^\mathrm{dR+}$ can be built easily from the $\mathbb{G}_m$-quotients of affine stacks. 

\begin{corollary}
Let $(A, I = (d))$ be a bounded oriented 
prism with $\overline{A} \coloneqq A/I$, and let $X=\Spf(R)$ be a $p$-complete smooth affine formal scheme over $\overline{A}$.
Let $\ast\in\{\emptyset, \perf\}$.
\begin{enumerate}
\item The category $\mathrm{Gauge}^*(X/A)$ is equivalent to the category of pairs
$(E,\Fil^\bullet E^{(1)})$, where $E\in D^*_{(p,I)\text{-comp}}(\Prism_{X/A})$ and 
$\Fil^\bullet E^{(1)}\in \DF^*_{(p,I)\text{-comp}}(\Fil^\bullet_N \Prism^{(1)}_{R/A})$ with underlying
complex $E^{(1)} \coloneqq E \widehat{\otimes}_{A, \varphi_A} A$.
\item The category $\FG^*(X/A)$ is equivalent to the category of triples $(E,\Fil^\bullet E^{(1)}, \widetilde{\varphi}_E)$, where $(E,\Fil^\bullet E^{(1)})$ is as in (1) and $\widetilde{\varphi}_E:\Fil^\bullet E^{(1)} \widehat{\otimes}_{\Fil^\bullet_N \Prism^{(1)}_{R/A}, \varphi} I^\mathbb{Z} \Prism_{X/A} \to E\otimes_{\Prism_{X/A}} I^\mathbb{Z} \Prism_{R/A}$ is a filtered isomorphism.
\item The category $\FConn^*(X/\overline{A})$ is equivalent to the category $\DF^*_{p\text{-comp}}(\mathrm{dR}(R/\overline{A}))$.
\end{enumerate}
\end{corollary}

\begin{proof}
This follows from combining \Cref{completely filtered descendable descent}
and \Cref{descendable proposition}.
\end{proof}

Using the descendability, below we also show the completeness of the filtration of an ($F$-)gauge in perfect complexes over a smooth $p$-adic formal scheme.

\begin{theorem}
\label{thm: filtered completeness}
Let $Z$ be a smooth $\overline{A}$-formal scheme.
Let $E=(E, \Fil^\bullet E^{(1)})$ be a prismatic gauge in perfect complexes
over $Z/A$,
let $\mathcal{E}$ be a prismatic crystal in perfect complexes over $Z/A$,
and let $(\mathcal{F}, \Fil^\bullet \mathcal{F})$ be a filtered perfect complex
with connection over $Z/\overline{A}$.
\begin{enumerate}
\item For any affine open $U = \Spf(R) \subset Z$
which admits a formal \'{e}tale map to $\widehat{\mathbb{A}^n}_{\overline{A}}$,
the filtered complexes $\Fil^{\bullet} E^{(1)}(U)$ and $\Fil^\bullet \mathcal{F}(U)$
are $(p, I)$-completely and $p$-completely perfect as filtered $\Fil^{\bullet}_N \Prism^{(1)}_{U/A}$-complexes
and $\Fil^{\bullet}_H \mathrm{dR}(U/\overline{A})$-complexes respectively.
\item The assignments $S \mapsto \mathcal{E}(\Prism_{S/A})$
and $S \mapsto \mathcal{E}^{(1)}(\Prism_{S/A}) \coloneqq
\mathcal{E}(\Prism_{S/A}) \widehat{\otimes}_{A, \varphi_A} A$
defines quasi-syntomic sheaves on $(Z/A)_{\qrsp}$.
With notation as in (1), the complex $\mathcal{E}(U)$ is $(p, I)$-completely perfect
over $\Prism_{U/A}$. Moreover the natural map 
\[
\mathcal{E}(U) \widehat{\otimes}_{A, \varphi_A} A \to
\mathcal{E}^{(1)}(U)
\]
is an isomorphism.
\item The filtrations $\Fil^{\bullet} E^{(1)}(Z)$ and $\Fil^\bullet \mathcal{F}(Z)$
are complete.
\end{enumerate}
\end{theorem}

Here the evaluation at $U$ of quasi-syntomic sheaves is defined by the unfolding
process,
see \Cref{unfolding process in relative setting}
and \Cref{how to compute unfolding remark}.

\begin{proof}
For (1): Fix a $p$-completely \'{e}tale map $\overline{A}\langle \underline{X} \rangle \to R$.
Let $R^{(\bullet)}$ be the Cech nerve of $R \to R_{\infty}$
where $R_{\infty} \coloneqq R \widehat{\otimes}_{\overline{A}\langle \underline{X} \rangle} 
\overline{A}\langle \underline{X}^{1/p^\infty} \rangle$.
By \Cref{how to compute unfolding remark},
the values $\Fil^{\bullet} E^{(1)}(U)$ and $\Fil^\bullet \mathcal{F}(U)$
are given by the descent of their values on $R^{(\bullet)}$.
By \Cref{filtered perfect criterion} and \Cref{Definition: filtered complete perfect},
we need to check their associated
Rees's construction is perfect after derived reduction mod $(p, I)$ and $p$ respectively.
Therefore we are reduced to \Cref{descendable proposition}, by descent of completely perfectness
along completely descendable map (\Cref{descendable discussion}).

For (2): the first statement follows from the fact that $\Prism_{-/A}$
and $\Prism^{(1)}_{-/A}$ are quasi-syntomic sheaves (by Hodge--Tate
and de Rham comparison respectively),
and the perfectness assumption on $\mathcal{E}$.
Then the completely perfectness of $\mathcal{E}(U)$
over $\Prism_{U/A}$ can be proved exactly as in (1).
Lastly, with notations as in the proof of (1), we need to show the following natural arrow:
\[
\mathcal{E}(U) \widehat{\otimes}_{A, \varphi_A} A \to 
\lim_{\bullet \in \Delta} \mathcal{E}^{(1)}(R^{(\bullet)})
\]
is an isomorphism.
We first observe base change formulas:
$\mathcal{E}(U) \widehat{\otimes}_{A, \varphi_A} A \cong \mathcal{E}(U) \widehat{\otimes}_{\Prism_{U/A}} \Prism^{(1)}_{U/A}$
as well as
$\mathcal{E}^{(1)}(R^{(\bullet)}) \cong \mathcal{E}(U) \widehat{\otimes}_{\Prism_{U/A}} \Prism^{(1)}_{R^{(\bullet)}/A}$.
By completely perfectness of $\mathcal{E}(U)$ over $\Prism_{U/A}$, using the above base change formulas,
we are reduced to showing that the unfolding of $\Prism^{(1)}_{-/A}$
at $U$ is $\Prism_{U/A} \widehat{\otimes}_{A, \varphi_A} A$:
This follows from the de Rham comparison as in \cite[Theorem 15.3]{BS19}.

As for (3): quasi-syntomic sheaves are in particular Zariski sheaves,
so it suffices to check completeness for the values on a basis of opens
in $Z_{\mathrm{Zar}}$.
Note that by derived Nakayama lemma, we are reduced to checking completeness
after applying $-/^L (p, I)$ and $-/^L p$ respectively to these values.
To that end, we simply use (1) and \Cref{perfect over complete is complete}
to reduce ourselves to the completeness of the Nygaard (resp.~Hodge) filtration
on the Frobenius-twisted prismatic cohomology (resp.~derived de Rham cohomology)
of smooth algebras. The Hodge filtration for smooth formal schemes 
is a finite filtration, hence complete.
The claim about Nygaard filtration can be found in \cite[Lemma 7.8.(1)]{LL20},
the proof is quite easy: by $\frac{d}{t}$-completeness, we may check the completeness
of filtration after modulo $\frac{d}{t}$, but then
\Cref{structure of Nygaard filtration on twisted-prismatic cohomology}
reduces us to checking completeness of Hodge filtration for smooth formal schemes.
\end{proof}

In a private communication, Bhatt told us that \cite[Theorem 7.17]{BL22b}
is proved by a similar argument.

\section{Height of prismatic cohomology}
\label{Height of prismatic cohomology}
In this section, we study Verschiebung operator and the Frobenius height of prismatic cohomology, with coefficients in coherent prismatic $F$-crystals.

Recall that in \Cref{meaning of lambda*}
we defined a natural graded derived pushforward functor between two graded derived
category of presheaves
\[
R\lambda_*: \DG((X/\overline{A})_\qrsp,\mathcal{O}_\qrsp^\mathrm{PSh}) \to \DG(X_\mathrm{Zar}, \mathcal{O}_{X}^\mathrm{PSh}).
\]	
Also recall that in \Cref{thm crystal to gauge} we have attached an $F$-gauge $\Pi_X(\mathcal{E})$
on $X$ to an $F$-crystal $(\mathcal{E}, \varphi_{\mathcal{E}})$ on $X$,
then according to \Cref{restriction functor}, we obtain a relative $F$-gauge
$BC_{X, X_{\overline{A}}/A}(\Pi_X(\mathcal{E}))$ over $X_{\overline{A}}/A$.
Using these constructions, our first main theorem is the following.

\begin{theorem}
\label{thm height of coh local}
		Let $f \colon X\to Y$  be a smooth morphism of smooth $p$-adic formal schemes over $\mathcal{O}_K$ that is of relative dimension $n$, and let $(A,I)$ be a bounded prism over $Y$ such that $\Spf(\overline{A})\to Y$ is $p$-completely flat.
		Assume  $(\mathcal{E},\varphi_{\mathcal{E}})\in \FC^\coh(X)$ is $I$-torsionfree of height $[a,b]$,
  with $(E,\Fil^\bullet E^{(1)},\widetilde{\varphi}_E)$ the associated $F$-gauge.
		\begin{enumerate}[label=(\roman*)]
			\item\label{thm height of coh local graded piece} For each $j\in \mathbb{Z}$, the complex $R\lambda_* \Gr^jE^{(1)}$ lives in $D^{[0,n]}_{\coh}(\mathcal{O}_{X_{\overline{A}}})$.
			\item\label{thm height of coh local with I-adic filtration} The filtered complex $R\lambda_* \Fil^{\geq n+b}E^{(1)}$ is equivalent to the $I$-adic filtration $R\lambda_* I^{\geq n+b} E^{(1)}=I^{\geq n+b}\otimes R\lambda_* E^{(1)}$.
			\item\label{thm height of coh local bound} The map $\widetilde{\varphi}_E$ induces a natural isomorphism of the truncations
			\[
			\varphi_{R^{\leq i} \lambda_*}: R^{\leq i} \lambda_* (\Fil^{i+b} E^{(1)}) \simeq R^{\leq i} \lambda_* (I^{i+b}  E)=I^{i+b}\otimes R^{\leq i}\lambda_* E, ~\forall i\in \mathbb{N}.
			\] 
		\end{enumerate}
\end{theorem}
 
Combine (i) and (ii) above, we see that $R\lambda_* \mathcal{E}$  and $R\lambda_* \Fil^\bullet E^{(1)}$ live in $D^{[0,n]}(X_{\overline{A}})$.

\begin{remark}
	Note that in the special case when $\mathcal{E}$ is the prismatic structure sheaf $\mathcal{O}_\Prism$, this is shown in \cite[Lem.\ 7.8]{LL20}, which essentially follows from the calculation of graded pieces of Nygaard filtration as in \cite[Thm.\ 15.2]{BS19}.
\end{remark}

\begin{proof}
	Denote by $(\overline{\mathcal{E}}, \Fil_\bullet \overline{\mathcal{E}})$ the associated filtered Higgs field as in \Cref{def HT realization}.
	By  \Cref{reduction of relative F-gauge that comes from F-crys}  and \Cref{conjugated graded of HT complex}, for each $l\in \mathbb{Z}$, we have 
	\[
	R\lambda_* \Gr_l \overline{\mathcal{E}} \in D_\coh^{[0,n]}(\mathcal{O}_X).
	\]
	Since $R\lambda_* \Fil_j \overline{\mathcal{E}}$ admits a finite increasing and exhaustive filtration with graded pieces being $R\lambda_* \Gr_l \overline{\mathcal{E}}$,
	we see the cohomological bound and coherence in (i) hold for $R\lambda_* \Fil_j \overline{\mathcal{E}}$.
	
To compare the two filtrations, by the filtered completeness of $\Fil^{\bullet}E^{(1)}$ in \Cref{thm: filtered completeness} and that of the $I$-adic filtration, it suffices to compare the map of associated graded pieces. 
		We let $C^j$ be the cone of the Frobenius map $\varphi_{\mathcal{E}}:\Fil^j E^{(1)} \to I^j E$, regarded as a descending filtered complex via the filtration $C^{\geq j}$.
		Then for each $j\in \mathbb{Z}$, the graded piece $\Gr^j C^\bullet$ by definition is 
		\[
		\mathrm{Cone}(\varphi_\mathcal{E}:\Gr^j E^{(1)} \longrightarrow \overline{\mathcal{E}}\{j\}),
		\]
		with the left hand side identified with $\Fil_j \overline{\mathcal{E}}\{j\}$ by \Cref{Higgs associated to F-gauge cohomology}.\ (i).
		So we get $\Gr^j C^\bullet \simeq \Fil_{\geq j+1} \overline{\mathcal{E}}\{j\}$.
		By \Cref{reduction of relative F-gauge that comes from F-crys} and \Cref{conjugated graded of HT complex}, the filtered complex $R\lambda_* ( \Fil_{\geq j+1} \overline{\mathcal{E}}\{j\})$ has a finite increasing and exhaustive filtration ranging from the filtered degree $j+1-b$ to $n$ (and the entire term vanishes when $j+1-b>n$.),
		such that the $l$-th graded piece lives in  $D_\coh^{[\max\{l,0\},n]}(\mathcal{O}_{X_{\overline{A}}})$.
		So by induction and the aforementioned facts in \Cref{conjugated graded of HT complex} again, we have 
		\[
		R\lambda_* \Gr^j C^\bullet \in D_\coh^{[\max\{j+1-b,0\}, n]}(\mathcal{O}_{X_{\overline{A}}}), \tag{$\ast$}
		\]
		which vanishes when $j+1-b>n$.
		This means that the Frobenius map induces an isomorphism from $R\lambda_* \Gr^j E^{(1)}$ to $R\lambda_* \overline{\mathcal{E}}\{j\}$ for $j+1-b>n$, and $R\lambda_* (\Fil^\bullet E^{(1)})$ is eventually the $I$-adic filtration.
		
		Finally by the completeness of the filtration, item (iii) follows from a reformulation of $(\ast)$ that for each $j\geq i+b$, we have $R\lambda_* \Gr^j C^\bullet$ lives in $D_\coh^{[i+1, n]}(\mathcal{O}_{X_{\overline{A}}})$.
\end{proof}

By taking the higher direct image along a proper smooth morphism, we can estimate the height of each individual prismatic cohomology.
This is analogous to the work of Kedlaya bounding slopes of rigid cohomology as in \cite[Thm.~6.7.1]{Ked06}.
\begin{theorem}
\label{thm height of coh global}
	Let $f \colon X\to Y$  be a smooth morphism of smooth $p$-adic formal schemes over $\mathcal{O}_K$ that is of relative dimension $n$, and let $(A,I)\in Y_\Prism$ such that $\Spf(\overline{A})\to Y$ is $p$-completely flat.
	Assume  $(\mathcal{E},\varphi_{\mathcal{E}})\in \FC^\coh(X)$ is $I$-torsionfree of height $[a,b]$, and we denote $E=(E,\Fil^\bullet E^{(1)},\widetilde{\varphi}_E)$ the associated relative $F$-gauge over $X_{\overline{A}}/A$.
	For an integer $j\geq a$ and a non-negative integer $i$, the Frobenius structure on the $i$-th relative prismatic cohomology induces a natural commutative diagram
	\[
	\begin{tikzcd}
		\mathrm{H}^i((X/\overline{A})_\qrsp, \Fil^j E^{(1)}) \arrow[r, "u^j"] \arrow[d, "\varphi^j"'] & \mathrm{H}^i((X/\overline{A})_\qrsp,E^{(1)}) \arrow[d, "\varphi"] \\
		I^j\otimes \mathrm{H}^i((X/\overline{A})_\qrsp, E) \arrow[r]& I^a\otimes \mathrm{H}^i((X/\overline{A})_\qrsp, E),
	\end{tikzcd}
    \]
    where the horizontal arrows are defined by the canonical maps and the vertical arrows are given by Frobenius morphisms on $E$.
    Then we have:
	\begin{enumerate}[label=(\roman*)]
		\item\label{thm height of coh global upper} the map $\varphi^j$ is an isomorphism when $j \geq b+\min\{i,n\}$, and		\item\label{thm height of coh global lower} the map $u^j$ is an isomorphism when $j \leq  a+\max\{0,i-n\}$. 
	\end{enumerate}
\end{theorem}
\begin{proof}
	We first notice that as the $F$-crystal $(\mathcal{E},\varphi_\mathcal{E})$ is $I$-torsionfree and of height $[a,b]$, the Frobenius map sends $\mathcal{E}^{(1)}$ into $I^a\mathcal{E}\simeq I^a\otimes \mathcal{E}\subset \mathcal{E}[1/I]$, and sends $\Fil^\bullet E^{(1)}$ into $I^a \otimes E$.
 So the diagram in the statement can be obtained by applying the functor $\mathrm{H}^i(X,-)$ 
 at the following diagram of complexes over $(X_{\overline{A}})_\mathrm{Zar}$: 
	\[
	\begin{tikzcd}
		R\lambda_* \Fil^j E^{(1)} \arrow[r, "v^j"] \arrow[d, "R\lambda_*\widetilde{\varphi}^j_E"'] & R\lambda_* E^{(1)} \arrow[d, "R\lambda_*\widetilde{\varphi}_E"] \\
		I^j\otimes R\lambda_* E \arrow[r] & R\lambda_* (I^a\otimes E)  \simeq  I^a \otimes R\lambda_* E.
	\end{tikzcd}
    \]
    To show item \ref{thm height of coh global upper}, we apply $\mathrm{H}^i(X,-)$ at the isomorphism of sheaves from \Cref{thm height of coh local}.\ref{thm height of coh local bound},
    which immediately implies that the map $\varphi^j$ is an isomorphism when $j \geq b+i$.
    It is then left to show that the map $\varphi^j$ is an isomorphism when $j \geq b+n$, and we assume now that $i \geq n$.
    Since both $R\lambda_* \Fil^\bullet E^{(1)}$ and $R\lambda_* E$  live in cohomological degree $[0,n]$ (\Cref{thm height of coh local}.\ref{thm height of coh local graded piece} and \ref{thm height of coh local with I-adic filtration}), the Leray spectral sequence for the map $\varphi^j=\mathrm{H}^i(X, R\lambda_*\widetilde{\varphi}_E^j)$ is computed by the following $n$-terms
    \[
    \mathrm{H}^{i-n}(X, R^n\lambda_* \widetilde{\varphi}_E^j), \ldots, \mathrm{H}^i(X, R^0\lambda_* \widetilde{\varphi}_E^j).
    \]
    Notice that by \Cref{thm height of coh local}.\ref{thm height of coh local bound}, the map $R^l\lambda_*\widetilde{\varphi}_E^j$ is an isomorphism when $j\geq b+l$.
    Thus the map $ \varphi^j$ is an isomorphism when $j \geq b+n$, as the number $l$ ranges from $0$ to $n$.
    
    For \ref{thm height of coh global lower}, the sheaf of complexes $\mathrm{Cone}(v^j)$ admits a decreasing filtration, which by \Cref{Higgs associated to F-gauge cohomology} satisfies
    \[
    \Gr^l(\mathrm{Cone}(v^j))=
    \begin{cases}
    	\Fil_l \bar{\mathcal{E}}\{l\},~ l \leq j-1,\\
    	0,~ l \geq j.
    \end{cases}
    \]
    Moreover, by \Cref{reduction of relative F-gauge that comes from F-crys} and \Cref{conjugated graded of HT complex}, each $\Fil_l \bar{\mathcal{E}}\{l\}$ lives in $ D^{[0, \min\{l-a,n\}]}_\mathrm{coh}(\mathcal{O}_X)$.
    So by the filtered completeness of $\mathrm{Cone}(v^j)$, the complex  $\mathrm{Cone}(v^j)$ itself lives in $D^{[0, \min\{ j-1-a,n\}]}_\mathrm{coh}(\mathcal{O}_X)$, which is zero when $j-1-a <0$, or equivalently $j \leq a$.
    This in particular implies that when $j\leq a$, the map $v^j$ and hence $u^j$ are always isomorphisms for all the integers $i$.
    To finish the proof of \ref{thm height of coh global lower}, we are left to consider the case when $i \geq n$.
    Using Leray spectral sequence again, the cone of the map $u^j=\mathrm{H}^i(v^j)$ is computed by the following terms
    \[
    \mathrm{H}^{i-n}(X, \mathcal{H}^n(\mathrm{Cone}(v^j)), \ldots, \mathrm{H}^n(X, \mathcal{H}^{i-n}(\mathrm{Cone}(v^j)),
    \]
    where we implicitly use the fact that $\mathrm{H}^l(X,\mathcal{H}^{i-l}(\mathrm{Cone}(v^j))=0$ for $l >n$.
    In particular, when $j \leq a+i-n$, as the complex $\mathrm{Cone}(v^j)$ lives in $D^{[0, \min\{ j-1-a,n\}]}_\mathrm{coh}(\mathcal{O}_X) \subset D^{[0, \min\{ i-n-1,n\}]}_\mathrm{coh}(\mathcal{O}_X)$, the terms computing $\mathrm{Cone}(u^j)$ above are all equal to zero.
    Hence the map $u^j$ is an isomorphism when $j\leq a+i-n$, under the assumption that $i \geq n$.
\end{proof}

Analogous to \cite[Cor.\ 15.5]{BS19}, we can extend the Verschiebung operator on individual prismatic cohomology sheaf to general coefficients.
\begin{corollary}
\label{Verschiebung}
Keep the same assumptions and notations as in \Cref{thm height of coh local}.
	Then for $i\in \mathbb{N}$, there is a natural map 
	\[
	\psi: I^{i+b}\otimes_A R^{\leq i} \lambda_* E \longrightarrow R^{\leq i} \lambda_* E^{(1)},
	\]
	such that its compositions with Frobenius morphism $\varphi: R^{\leq i} \lambda_* E^{(1)} \to R^{\leq i} \lambda_* E$ are canonical maps induced from inclusions $I^{i+b}\subset A$, namely
	\begin{align*}
			&\varphi\circ \psi \simeq \mathrm{can}:I^{i+b}\otimes R^{\leq i} \lambda_* E\longrightarrow  R^{\leq i} \lambda_* E;\\
			&\psi\circ\varphi \simeq \mathrm{can} :I^{i+b}\otimes R^{\leq i} \lambda_* E^{(1)} \longrightarrow R^{\leq i} \lambda_* E^{(1)}.
	\end{align*}
\end{corollary}
\begin{proof}
	Let $C$ be the sheaf of complexes $R^{\leq i} \lambda_* E$, and let $C^{(1)}$ be its $\varphi_A$-linear twist, which by flatness is $R^{\leq i} \lambda_* E^{(1)}$.
	Similarly we let $\Fil^j C^{(1)}$ be the complex $R^{\leq i} \lambda_* (\Fil^j E^{(1)})$.
	Define the map $\psi$ by the following composition
	\[
	\begin{tikzcd}
		C\otimes I^{i+b} \arrow[rr, "\varphi^{-1}_{R^{\leq i} \lambda_*}", "\simeq"'] && \Fil^{i+b} C^{(1)} \arrow[r, "\mathrm{can}"] & C^{(1)}.
	\end{tikzcd}
    \]
    The first  composition formula then follows from the following enlarged commutative diagram
	\[
	\begin{tikzcd}
		 C\otimes I^{i+b} \arrow[rr, "\varphi^{-1}_{R^{\leq i} \lambda_*}", "\simeq"'] \arrow[rrd, "\psi"'] && \Fil^{i+b} C^{(1)} \arrow[d, "\mathrm{can}"] \arrow[r, "\varphi", "\simeq"'] & I^{i+b}\otimes C\arrow[d, "\mathrm{can}"] \\ 
		&& C^{(1)} \arrow[r, "\varphi"] & C.
	\end{tikzcd}
    \]
    For the second formula, it follows from the observation that the composition below is the canonical map
    \[
    \begin{tikzcd}
    	I^{i+b}\otimes C^{(1)} \arrow[r, "\varphi"] & I^{i+b}\otimes C  \arrow[rr, "\varphi^{-1}_{R^{\leq i} \lambda_*}", "\simeq"'] && \Fil^{i+b} C^{(1)}.
    \end{tikzcd}
    \]
\end{proof}

Recall that evaluating at prisms whose reduction is $p$-completely flat is $t$-exact,
see \Cref{Evaluating at flat prism is t-exact}.
By taking all the bounded prisms $(A,I)\in Y_\Prism$ such that $\Spf(\overline{A})$ is $p$-completely flat over $Y$,
the above implies the following result on individual relative prismatic cohomology crystal.

\begin{corollary}
\label{bound of Frob height of cohomology}
Let $f:X\to Y$ be a smooth proper morphism of smooth $p$-adic formal schemes over $\mathcal{O}_K$, and assume $(\mathcal{E},\varphi_{\mathcal{E}})\in \FC^\coh(X)$ is $I$-torsionfree of height $[a,b]$.
For $i\in \mathbb{N}$, the $R^i f_{\Prism, *} \mathcal{E}$
is a coherent prismatic $F$-crystal over $Y_\Prism$, such that the image of its Frobenius morphism $\varphi$ within $\mathcal{I}_\Prism^a\otimes R^i f_{\Prism, *} \mathcal{E}$ satisfies the inclusions:
\[
\mathrm{Im}(\mathcal{I}_\Prism^{b+\min\{i,n\}}\otimes R^i f_{\Prism, *} \mathcal{E}) \subseteq \mathrm{Im}(\varphi) \subseteq \mathrm{Im}(\mathcal{I}_\Prism^{a+\max\{0 ,i-n\}}\otimes R^i f_{\Prism, *} \mathcal{E}).
\]
In particular, the $I$-torsionfree quotient of $R^i f_{\Prism, *} \mathcal{E}$ has Frobenius height in
$[a+\max\{0 ,i-n\}, b+\min\{i,n\}]$.
\end{corollary}

Here we implicitly use \Cref{thm: filtered completeness}.(2) to identify $R^if_{\Prism,*} \mathcal{E}^{(1)}$ with $\varphi_{Y_\Prism}^* R^if_{\Prism,*} \mathcal{E}$ for the Frobenius morphism of $R^if_{\Prism,*} \mathcal{E}$.

The following lemma says that passing to $I$-power torsions or $I$-torsionfree quotient
preserves $F$-crystals.

\begin{lemma}
\label{DLMS lemma}
Let $X$ be a smooth $p$-adic formal schemes over $\mathcal{O}_K$, 
and let $(\mathcal{E},\varphi_{\mathcal{E}}) \in \FC^\coh(X)$ be a coherent $F$-crystal.
Then $\mathcal{F} \coloneqq \mathcal{E}[I^{\infty}]$ is $(I, p)$-power-torsion,
hence $\varphi_{X_\Prism}^* \mathcal{F}[1/I] = 0 = \mathcal{F}[1/I]$.
In particular, we have a coherent $F$-crystal $(\mathcal{F}, 0)$, and
the $\varphi_{\mathcal{E}}$ induces a Frobenius isogeny
on $\mathcal{E}/\mathcal{F}$ making it an $I$-torsionfree coherent $F$-crystal.
\end{lemma}

\begin{proof}
On Breuil--Kisin prisms $(A, I) \in X_{\Prism}$, we have
$\mathcal{F}(A) = \mathcal{E}(A)[I^{\infty}]$.
It suffices to know that $\mathcal{F}(A)$ is $p$-power torsion.
This follows from \cite[Proposition 4.13]{DLMS}: the authors showed that $\mathcal{E}(A)[p^{-1}]$
is a finite projective $A[p^{-1}]$-module, hence $\mathcal{F}(A)[p^{-1}] = 0$.
\end{proof}

Next we show the derived pushforward of $I$-power torsion prismatic $F$-crystals
will have isogenous Frobenius, for trivial reasons.

\begin{proposition}
\label{pushforward of I-power torsion F-crystal}
Let $f \colon X\to Y$ be a qcqs
smooth morphism of smooth $p$-adic formal schemes over $\mathcal{O}_K$, 
and let $(\mathcal{E},\varphi_{\mathcal{E}}) \in \FC^\coh(X)$ be an $I^{\infty}$-torsion $F$-crystal
on $X_{\Prism}$.
Then both $Rf_{\Prism, *} \mathcal{E}[1/I] = 0$ and
$\varphi_{Y_{\Prism}}^*Rf_{\Prism, *} \mathcal{E}[1/I] = 0$.
In particular, we have $(Rf_{\Prism, *} \mathcal{E}, Rf_*(\varphi_{\mathcal{E}})) \in \FC^{\perf}(Y)$.
\end{proposition}

\begin{proof}
The claim about $Rf_{\Prism, *} \mathcal{E}[1/I] = 0$ follows immediately
from the assumption that $E[1/I] = 0$ and the map being qcqs.
Below we show the other claim. By standard argument, we are immediately
reduced to the case of $Y = \Spf(R_0)$ is an affine and
$X = \Spf(R)$ is also an affine which moreover admits an \'{e}tale chart:
Namely we may assume that there is an \'{e}tale map $R_0\langle T_i^{\pm 1} \rangle \to R$.

Mimicking \cite[Example 3.4]{DLMS}, we see that one can find a Breuil--Kisin prism
$(A, I)$ covering $Y_{\Prism}$ as well as a Breuil--Kisin prism $(B, J)$
covering the relative prismatic site $(X/A)_{\Prism}$, and moreover the relative Frobenius
$\varphi_{B/A} \colon B \widehat{\otimes}_{A, \varphi_A} A \to B$
is faithfully flat.
Similar to \Cref{BK covering property}, the object $(B, J = IB)$ is weakly final in $(X/A)_{\Prism}$,
and its $(i+1)$-fold self product in $(X/A)_{\Prism}$ exist for all $i \in \mathbb{N}$
and are all given by affine objects denoted as $(B^{(i)}, IB^{(i)})$.
With the above notation, the vanishing we are after translates\footnote{To see this translation,
we refer readers to the discussion around \cite[4.17-4.18 and their footnote 10.]{BS19}} to:
\[
\lim_{\bullet \in \Delta} \mathcal{E}(B^{(\bullet)}, IB^{(\bullet)})
\widehat{\otimes}_{A, \varphi_A} A[1/I] = 0.
\]
To that end, it suffices to show $\mathcal{E}(B^{(\bullet)}, IB^{(\bullet)})
\widehat{\otimes}_{A, \varphi_A} A[1/I] = 0$,
and in fact it suffices to show this when $\bullet = 0$
(as the latter ones are base changed from this case).
Finally, since the relative Frobenius on $B = B^{(0)}$ is faithfully flat,
we are reduced to knowing $\mathcal{E}(B, IB)
\widehat{\otimes}_{B, \varphi_B} B[1/I] = 0$.
But this follows from the Frobenius isogeny property of $\mathcal{E}$:
\[
\mathcal{E}(B, IB)
\widehat{\otimes}_{B, \varphi_B} B[1/I] \xrightarrow[\varphi_{\mathcal{E}}]{\cong}
\mathcal{E}(B, IB) [1/I]
\]
and the assumption that $\mathcal{E}$ is $I^{\infty}$-torsion
(so the latter module in the above equation is $0$).

As for the last sentence, we just note that
(see for instance \cite[Proposition 5.11]{GR22})
$Rf_{\Prism, *}\mathcal{E}$ is a prismatic crystal in perfect complex over $Y$.
\end{proof}

Now we are ready to generalize the ``Frobenius isogeny property'' and ``weak
\'{e}tale comparison'' in \cite{GR22}.

\begin{theorem}
\label{direct image vs etale realization}
Let $f \colon X \to Y$ be a smooth proper morphism
between smooth formal schemes over $\Spf(\mathcal{O}_K)$, then derived pushforward
of $F$-crystals in perfect complexes on $X$ are $F$-crystals in perfect complexes on $Y$,
moreover the following diagram commutes functorially:
\[
\xymatrix{
\FC^{\perf}(X) \ar[d]_{Rf_{\Prism, *}} \ar[r]^{T(-)} & 
D^{(b)}_{lisse}(X_{\eta}, \mathbb{Z}_p) \ar[d]^{Rf_{\eta, *}} \\
\FC^{\perf}(Y) \ar[r]^{T(-)} & 
D^{(b)}_{lisse}(Y_{\eta}, \mathbb{Z}_p).
}
\]
\end{theorem}
\begin{proof}
Let $(\mathcal{E},\varphi_\mathcal{E})$ be in $\FC^\perf(X)$.
It is known (see for instance \cite[Proposition 5.11]{GR22})
that $Rf_{\Prism, *}\mathcal{E}$ is a prismatic crystal in perfect complexes over $Y$.
To see the induced Frobenius on $Rf_{\Prism, *}\mathcal{E}$ must be an isogeny,
we first reduce ourselves to the case where $(\mathcal{E},\varphi_{\mathcal{E}}) \in \FC^\coh(X)$.
By considering the $I^{\infty}$-torsions of $\mathcal{E}$ and its $I$-torsionfree
quotient, which are coherent $F$-crystals by \Cref{DLMS lemma},
we are done thanks to \Cref{pushforward of I-power torsion F-crystal} and
\Cref{Verschiebung} respectively.

To see the commutative diagram, we recall by construction in \cite[Cor.\ 3.7]{BS21} that the \'etale realization $T(Rf_{\Prism, *} \mathcal{E})$ is isomorphic to the sheaf of complexes over $Y_{\eta,\mathrm{pro\acute{e}t}}$, sending an affinoid perfectoid Huber pair $(S[1/p],S)$ over $Y_\eta$ onto the following complexes
	\[
	\mathrm{fib}\left(\begin{tikzcd}
		(Rf_{\Prism, *} \mathcal{E}[1/I]^\wedge_p)(\mathrm{A}_{\inf}(S),I) \arrow[rr, "\varphi-\mathrm{id}"] && (Rf_{\Prism, *} \mathcal{E}[1/I]^\wedge_p)(\mathrm{A}_{\inf}(S),I),
	\end{tikzcd}\right)
	\]
	namely
	\[
	T(Rf_{\Prism, *} \mathcal{E}): (S[1/p],S) \longmapsto \left( (Rf_{\Prism, *} \mathcal{E}[1/I]^\wedge_p) (\mathrm{A}_{\inf}(S),I)\right)^{\varphi=1}.
	\]
	On the other hand, by applying derived global section at weak \'etale comparison in \cite[Thm.\ 6.1]{GR22}, for any perfect prism $(A,I)$ in $Y_\Prism$, there is a natural isomorphism of $\mathbb{Z}_p$-complexes
	\[
	R\Gamma((X_{\overline{A}})_{\eta, \mathrm{pro\acute{e}t}},T(\mathcal{E})) \simeq \left( (Rf_{\Prism, *} \mathcal{E}[1/I]^\wedge_p)(A,I)\right)^{\varphi=1},
	\]
	where $X_{\overline{A}}$ is the complete base change $X\times_Y \Spf(\overline{A})$.
    As a consequence, by taking the inverse system with respect to perfect prisms associated to affinoid perfectoid Huber pairs $(S[1/p], S)$ over $Y_\eta$, we get a natural isomorphism of $\mathbb{Z}_p$-complete complexes over $Y_\eta$
    \[
    Rf_{\eta,*}T(\mathcal{E}) \simeq T(Rf_{\Prism, *} \mathcal{E}).
    \]
\end{proof}

Combining the above with \Cref{etale realization is t-exact},
we get the following slight refinement of the
``weak \'{e}tale comparison'' \cite[Theorem 6.1]{GR22}.

\begin{corollary}\label{weak etale comparison}
Let $f \colon X \to Y$ be a smooth proper morphism
between smooth formal schemes over $\Spf(\mathcal{O}_K)$, and let 
$(\mathcal{E},\varphi_\mathcal{E}) \in \FC^{\perf}(X)$.
Then $T(R^i f_{\Prism, *} \mathcal{E}) = R^i f_{\eta, *}(T(\mathcal{E}))$.
\end{corollary}

\bibliographystyle{amsalpha}
\bibliography{template}

\providecommand{\bysame}{\leavevmode\hbox to3em{\hrulefill}\thinspace}
\providecommand{\MR}{\relax\ifhmode\unskip\space\fi MR }
\providecommand{\MRhref}[2]{%
  \href{http://www.ams.org/mathscinet-getitem?mr=#1}{#2}
}
\providecommand{\href}[2]{#2}
\begin{thebibliography}{DLMS22}

\bibitem[ALB23]{ALB23}
{Anschütz, Johannes} and {Arthur-César} Le~Bras, \emph{Prismatic {Dieudonné} theory}, Forum of Mathematics, Pi \textbf{11} (2023), e2.

\bibitem[BBD82]{BBDG82}
A.~A. Be\u{\i}linson, J.~Bernstein, and P.~Deligne, \emph{Faisceaux pervers}, Analysis and topology on singular spaces, {I} ({L}uminy, 1981), Ast\'{e}risque, vol. 100, Soc. Math. France, Paris, 1982, pp.~5--171. \MR{751966}

\bibitem[Bha22]{Bha23}
Bhargav Bhatt, \emph{Prismatic {$F$}-gauges}, 2022, Lecture notes available at \url{https://www.math.ias.edu/~bhatt/teaching/mat549f22/lectures.pdf}.

\bibitem[BL22a]{BL22a}
Bhargav {Bhatt} and Jacob {Lurie}, \emph{{Absolute prismatic cohomology}}, arXiv e-prints (2022), arXiv:2201.06120.

\bibitem[BL22b]{BL22b}
\bysame, \emph{{The prismatization of $p$-adic formal schemes}}, arXiv e-prints (2022), arXiv:2201.06124.

\bibitem[BMS18]{BMS18}
Bhargav Bhatt, Matthew Morrow, and Peter Scholze, \emph{Integral {$p$}-adic {H}odge theory}, Publ. Math. Inst. Hautes \'{E}tudes Sci. \textbf{128} (2018), 219--397. \MR{3905467}

\bibitem[BMS19]{BMS19}
\bysame, \emph{Topological {H}ochschild homology and integral {$p$}-adic {H}odge theory}, Publ. Math. Inst. Hautes \'{E}tudes Sci. \textbf{129} (2019), 199--310. \MR{3949030}

\bibitem[BS17]{BS17}
Bhargav Bhatt and Peter Scholze, \emph{Projectivity of the {W}itt vector affine {G}rassmannian}, Invent. Math. \textbf{209} (2017), no.~2, 329--423. \MR{3674218}

\bibitem[BS22]{BS19}
\bysame, \emph{Prisms and prismatic cohomology}, Ann. of Math. (2) \textbf{196} (2022), no.~3, 1135--1275. \MR{4502597}

\bibitem[BS23]{BS21}
\bysame, \emph{Prismatic {$F$}-crystals and crystalline {G}alois representations}, Camb. J. Math. \textbf{11} (2023), no.~2, 507--562. \MR{4600546}

\bibitem[DL21]{DuLiu21}
Heng {Du} and Tong {Liu}, \emph{{A prismatic approach to $(\varphi, \hat G)$-modules and $F$-crystals}}, arXiv e-prints (2021), arXiv:2107.12240.

\bibitem[DLMS22]{DLMS}
Heng {Du}, Tong {Liu}, Yong~Suk {Moon}, and Koji {Shimizu}, \emph{{Completed prismatic $F$-crystals and crystalline $\mathbf{Z}_p$-local systems}}, arXiv e-prints (2022), arXiv:2203.03444.

\bibitem[{Dri}20]{Drin20}
Vladimir {Drinfeld}, \emph{{Prismatization}}, arXiv e-prints (2020), arXiv:2005.04746.

\bibitem[GL21]{GL21}
Haoyang Guo and Shizhang Li, \emph{Period sheaves via derived de {R}ham cohomology}, Compos. Math. \textbf{157} (2021), no.~11, 2377--2406. \MR{4323988}

\bibitem[GR22]{GR22}
Haoyang {Guo} and Emanuel {Reinecke}, \emph{{A prismatic approach to crystalline local systems}}, arXiv e-prints (2022), arXiv:2203.09490.

\bibitem[Ked06]{Ked06}
Kiran~S. Kedlaya, \emph{Fourier transforms and {$p$}-adic `{W}eil {II}'}, Compos. Math. \textbf{142} (2006), no.~6, 1426--1450. \MR{2278753}

\bibitem[LL20]{LL20}
Shizhang {Li} and Tong {Liu}, \emph{{Comparison of prismatic cohomology and derived de Rham cohomology}}, arXiv e-prints (2020), arXiv:2012.14064, to appear in J.~Eur.~Math.~Soc.

\bibitem[Lur09]{Lur09}
Jacob Lurie, \emph{Higher topos theory}, Annals of Mathematics Studies, vol. 170, Princeton University Press, Princeton, NJ, 2009. \MR{2522659}

\bibitem[Lur17]{Lur17}
Jacob Lurie, \emph{Higher algebra}, 2017, available at \url{https://www.math.ias.edu/~lurie/papers/HA.pdf}.

\bibitem[Lur18]{Lur18}
\bysame, \emph{Spectral algebraic geometry}, 2018, available at \url{https://www.math.ias.edu/~lurie/papers/SAG-rootfile.pdf}.

\bibitem[Mag24]{Mag24}
Kirill Magidson, \emph{Divided powers and derived de rham cohomology}, 2024.

\bibitem[Mat16]{Mat16}
Akhil Mathew, \emph{The {G}alois group of a stable homotopy theory}, Adv. Math. \textbf{291} (2016), 403--541. \MR{3459022}

\bibitem[Mat22]{Mat22}
\bysame, \emph{Faithfully flat descent of almost perfect complexes in rigid geometry}, J. Pure Appl. Algebra \textbf{226} (2022), no.~5, Paper No. 106938, 31. \MR{4332074}

\bibitem[MT20]{MT20}
Matthew {Morrow} and Takeshi {Tsuji}, \emph{{Generalised representations as q-connections in integral $p$-adic Hodge theory}}, arXiv e-prints (2020), arXiv:2010.04059.

\bibitem[MW21]{MW21}
Yu~{Min} and Yupeng {Wang}, \emph{{Relative $(\varphi,\Gamma)$-modules and prismatic $F$-crystals}}, arXiv e-prints (2021), arXiv:2110.06076.

\bibitem[MW22]{MW22}
\bysame, \emph{{P-adic Simpson correpondence via prismatic crystals}}, arXiv e-prints (2022), arXiv:2201.08030.

\bibitem[Shi18]{Shimizu18}
Koji Shimizu, \emph{Constancy of generalized {H}odge-{T}ate weights of a local system}, Compos. Math. \textbf{154} (2018), no.~12, 2606--2642. \MR{3873527}

\bibitem[{Sta}23]{stacks-project}
The {Stacks Project Authors}, \emph{\textit{Stacks Project}}, \url{https://stacks.math.columbia.edu}, 2023.

\bibitem[{Tan}22]{Tang22}
Longke {Tang}, \emph{{Syntomic cycle classes and prismatic Poincar{\'e} duality}}, arXiv e-prints (2022), arXiv:2210.14279.

\bibitem[{Tia}21]{Tian21}
Yichao {Tian}, \emph{{Finiteness and Duality for the cohomology of prismatic crystals}}, arXiv e-prints (2021), arXiv:2109.00801.

\bibitem[Wu21]{Wu21}
Zhiyou Wu, \emph{Galois representations, {$(\varphi,\Gamma)$}-modules and prismatic {F}-crystals}, Doc. Math. \textbf{26} (2021), 1771--1798. \MR{4353339}

\end{thebibliography}

\end{document}